\documentclass[12pt,a4paper]{amsart}
\usepackage{amsmath,amssymb,amsthm,amscd,epsfig }
\usepackage{enumerate}
\usepackage{enumitem}
\usepackage{color}
\usepackage{bbm}
\usepackage[colorinlistoftodos]{todonotes}
\usepackage{version}
\usepackage[margin=2.6cm]{geometry}

\DeclareFontFamily{OT1}{rsfs}{}
\DeclareFontShape{OT1}{rsfs}{n}{it}{<-> rsfs10}{}
\DeclareMathAlphabet{\curly}{OT1}{rsfs}{n}{it}

\newtheorem{Thm}{Theorem}[section]
\newtheorem{Lem}[Thm]{Lemma}

\newtheorem{Cor}[Thm]{Corollary}
\newtheorem{Rem}[Thm]{Remark}

\newtheorem{Prop}[Thm]{Proposition}
\newtheorem{conj}[Thm]{Conjecture}
\newtheorem{``Conj"}[Thm]{``Conjecture"}

\newtheorem{Claim}[Thm]{Claim}

\theoremstyle{remark}
\newtheorem{ex}[Thm]{Example}
\theoremstyle{definition}
\newtheorem{defn}[Thm]{Definition}

\newtheorem{Property}{Property}
\newtheorem{Step}{Step}
\newtheorem{Setup}{Setup}
\newtheorem*{ack}{Acknowledgments}

\newtheorem{ntt}{Notation}
\newtheorem*{Ass*}{Assumption}

\newcommand{\gr}{\mathop{\mathrm{gr}}\nolimits}

\def\dim{\operatorname{dim}}

\def\R{\mathbb{R}}

\def\C{\mathbb{C}}
\def\G{\mathbb{G}}
\def\Z{\mathbb{Z}}
\def\O{\mathcal{O}}
\def\P{\mathbb{P}}
\def\Q{\mathbb{Q}}

\def\A{\mathbb{A}}

\def\X{\mathcal{X}}
\def\Y{\mathcal{Y}}

\def\V{\mathcal{V}}

\excludeversion{memo}

\begin{document}

\title{Canonical torus action on symplectic singularities}
\author{Yoshinori Namikawa, Yuji Odaka}
\date{\today}
\maketitle
\thispagestyle{empty}

\begin{abstract}
We show that any symplectic singularity lying on a smoothable projective 
symplectic variety locally admits a good action of $(\C^*)^r$, which is 
canonical. 
Under mild assumptions, 
we actually prove such singularity germ is 
the cone vertex over a contact orbifold 
with weak K\"ahler-Einstein metric, forcing $r=1$. 
In particular, it admits a (canonical) 
good $\mathbb{C}^*$-action, which also 
extends to (canonical) actions of $\mathbb{H}^*\supset SU(2)$. 
These settle Kaledin's conjecture conditionally but in a substantially 
stronger form by establishing the canonicity, the extensibility of the 
action, for instance. 
Our key idea is to use the Donaldson-Sun theory 
on local K\"ahler metrics in complex differential geometry 
to connect with the theory of Poisson deformations of 
symplectic  varieties. 

For general symplectic singularities, 
we prove the same assertions, assuming that the Donaldson-Sun theory extends to such singularities along with suitable singular (hyper)K\"ahler metrics. 
Conversely, our results can also be used to study the local 
behavior of such metrics around the germ. 
\end{abstract}

\section{Introduction}

\subsection{Background and the main result}
We work over the complex number field $\C$. 
After the work of Beauville \cite{Beauville}, a normal algebraic variety $V$ is called a {\it symplectic variety} if 
there is an (algebraic)  symplectic form $\sigma_V$ on the 
smooth locus $V^{sm}$ such that, for any resolution $f: \tilde{V} \to V$, the pullback of $\sigma_V$ to $f^{-1}(V^{sm})$ extends to a 
holomorphic 2-form on $\tilde{V}$. An algebraic variety $X$ has a {\it symplectic singularity} at a point if the point admits a symplectic variety $V$ as an open neighborhood. When an algebraic torus $T$ acts effectively on an affine symplectic variety $V$ fixing a point $0 \in V$, we call the action good if the closure of any $T$-orbit contains $0$. If the symplectic form $\sigma_V$ is homogeneous for the good $T$-action, then $V$ is called a {\it conical symplectic variety} (with respect to the 
$T$-action). 

\vspace{2mm}
About two decades ago, 
D.~Kaledin (\cite[cf., Remark 4.2, \S 4]{Kaledin}, \cite[Conjecture 1.8]{Kaledin 2})  conjectured  that, for any symplectic singularity $x \in X$, its analytic germ 
$(X, x)$ is actually isomorphic to 
the analytic germ  $(V,0)$ of a 
conical symplectic variety with 
a good $\mathbb{G}_m$-action. 

In this paper, we prove the following, 
which proves the conjecture  conditionally (see also Remark \ref{Kaledin's conj}) but in a significantly stronger form. 

\begin{Thm}[=Theorem \ref{Mthm}]\label{Mthm.intro}
Let $(\bar{X},L)$ be a polarized projective symplectic variety.  Suppose that $(\bar{X},L)$
satisfies either of the following equivalent conditions  (cf. Theorem \ref{smoothing} for the equivalence): 
\begin{enumerate}
\item \label{1.1i} $\bar{X}$ has a projective 
symplectic resolution, or 
\item \label{1.1ii} $(\bar{X},L)$ has a smoothing (as a polarized variety). 
\end{enumerate}

Then, the analytic germ of $x\in \bar{X}$ is that of a (canonical) 
conical affine symplectic variety $C$ at the vertex $0\in C\curvearrowleft (\mathbb{G}_m)^r$ with $r\ge 1$. 

Furthermore, $0\in C$ has a (singular) hyperK\"ahler cone  metric, 
which in particular induces a canonical action 
of 
the multiplicative group $\R_{>0}$, 
as rescaling up of the metric. This action is the restriction of  
the algebraic action of $(\C^*)^r$ 
via some embedding $\R_{>0}\hookrightarrow (\C^*)^r$ 
as Lie groups. 
\end{Thm}
For general symplectic singularities, our later 
Theorem \ref{Mthm2.intro} proves the same under 
a differential geometric assumption. 

\vspace{3mm}

We emphasize that the above Theorem \ref{Mthm.intro} (and 
Theorem \ref{Mthm2.intro}) contains 
substantial enhancement compared with 
the original Kaledin's conjecture, for at least two 
aspects. 

The first aspect is that, 
as explained shortly below, our arguments give {\it canonical}  
$C$ and the {\it canonical actions} of 
$(\G_m)^r$ and $\R_{>0}$, not only their existence. 
\begin{Rem}
As we explain in section \ref{sec:2}, 
the rank $r$ is intrinsically determined by the analytic 
germ of $x\in \overline{X}$ in general, and 
sometimes writes as $r(\xi)$ in later sections. 
We actually conjecture $r=1$ for general 
symplectic singularities (Conjecture \ref{r1conj}) and 
prove it conditionally - for any $x\in \overline{X}$ 
if $b_2(\overline{X})\ge 6$ for instance. Or alternatively, we also prove 
the same ($r=1$) when $c_1(L)^{\perp}(\subset H^2(\overline{X},\Z))$ 
has a non-zero isotropic element $e$ with respect to the 
Beauville-Bogomolov form. The latter type assumption can be morally 
regarded as a 
version of Strominger-Yau-Zaslow type conjectures. 

This $r=1$ statement also gives a new systematic construction of 
(log) K-polystable Fano pairs, without calculating any invariants 
from the K-stability theory (such as 
\cite{FO,BlJ}). 
See Theorem \ref{r1prop} and Remark \ref{LSconj} 
for more details. 
\end{Rem}

The second aspect of our improvements from the Kaledin's works 
and the expectation 
(\cite{Kaledin, Kaledin 2})
is that, 
as the last paragraph of Theorem \ref{Mthm.intro} states, our work 
connects with {\it complex differential geometry}. 
Indeed, our proof crucially contains arguments on metric structure, notably 
Donaldson-Sun theory \cite{DSII} 
on the local metric tangent cone of {\it local 
singular K\"ahler-Einstein metrics}.
\footnote{After more algebraic original approach 
to this problem with a partial progress, 
Kaledin wrote in  \cite[Remark 4.2]{Kaledin}: ``{\it it  seems that this would require a radically different approach}". One of our main new inputs is connection 
with geometry of canonical K\"ahler metrics (and related algebraic geometry).} 
For that, first recall that either the crepant 
resolutions of $\overline{X}$ (in the case \eqref{1.1i}) 
or smoothings (in the case \eqref{1.1ii}) 
have 
Ricci-flat K\"ahler metrics \cite{Yau} and as their 
limit (\cite{RZ, DS, JianSong}), $\overline{X}$ 
has a unique singular Ricci-flat K\"ahler metric on 
$\bar{X}$ in the class $2\pi c_1(L)$ (cf., also 
\cite{EGZ}). 

Then, 
the cone $(0 \in C)$ in Theorem \ref{Mthm.intro} is the  
local metric tangent cone of $g_{\overline{X}}$ near $x$, which has a 
natural structure of an affine algebraic variety (\cite[1.3]{DSII}). 
The metric also turns out to be 
(singular) hyperK\"ahler metric in our setup 
(cf., Theorem \ref{ssps}). 
This structure of an algebraic variety actually depends only on the analytic germ of 
$x \in \bar{X}$, not on the metric 
(\cite[3.22]{DSII}, \cite{LWX}). 
The smooth locus $C^{sm}$ of $C$ is a metric cone $C(S)$ over a real 
$2n-1$ dimensional Riemannian Sasakian manifold $S$ with the 
Reeb vector field $\xi$. Here 
$n := \dim_{\mathbb C}\bar{X}$. 
The Reeb vector field generates a subgroup of the isometry group 
of $C$ and its closure can be complexified into an algebraic torus $T \simeq(\mathbb{G}_m)^r$, which acts on $C$. The Reeb vector field, after untwist by the complex structure $J$, 
also corresponds to a real Lie  subgroup $\mathbb{R}_{>0}$ of 
$T$. The rescaling action of 
$\R_{>0}$ on $C$ is unique once we are given an 
affine variety structure on $C$, without a priori knowing the cone metric structure of $C$, 
by the so-called volume 
minimization principle (\cite{MSY1, MSY2}). Since the affine variety structure on $C$ only depends  
on the analytic germ of $\bar{X}$ at $x$, the $T$-action on $C$ also only does.

Our key idea is to relate the Donaldson-Sun theory 
to the theory of {\it Poisson deformations} of 
(non-compact) symplectic varieties, 
which inherit the Poisson structure \cite{Ginzburg-Kaledin, Nama, Namb}. 

\subsection{Outline of the proof}\label{sec:outline}
We outline a little more details of our proof of Theorem \ref{Mthm.intro}.  
Let $x \in X \subset \bar{X}$ be 
an affine open neighborhood of $x$. 
Recall that, in \cite[\S 3]{DSII}, Donaldson and Sun have given a finer description of $C$ by 2-step degenerations 
$$(x\in X)\rightsquigarrow (0\in W)\rightsquigarrow (0\in C_x(X)=:C),$$
in terms of more holomorphic or even algebraic data, 
rather than the local metric. 
Here $(0 \in W)$ and $(0 \in C)$ are both 
affine normal varieties with good $T$-actions. 

Roughly speaking, we realize these 2-step degenerations as Poisson deformations. 
More precisely, we take a subgroup $\mathbb{G}_m \subset T$ 
(or, in a complex geometric term, 
$\G_m(\C)=\C^*\subset T(\C)\simeq (\C^*)^r$) 
so that the derivative along $\R_{>0}\subset \C^*(\subset T(\C))$ 
approximates the Reeb vector field $\xi$ enough and 
construct a $\mathbb{G}_m$-equivariant flat deformation $$\mathcal{X} \to 
\mathbb{A}^1$$ in such a way that the fiber over $0 \in \mathbb{A}^1$ is $W$ 
and others fibers are all isomorphic to $X$. Here $\mathbb{G}_m$ acts on 
the base $\mathbb{A}^1$ with a negative weight fixing $0 \in \mathbb{A}^1$. 
The 1-parameter subgroup $\mathbb{G}_m 
\subset T$ determines an element 
$\xi' \in \mathrm{Lie}(T(\C))$ and 
our $\mathcal{X}$ actually depends on $\xi'$; hence, we denote it by $\mathcal{X}_{\xi'}$ in the main text of this article (especially Lemma \ref{XtoW.C} and \S \ref{sec:XW}). 
A general fiber $X$ admits an (algebraic) symplectic form $\sigma_X$ on  $X^{sm}$ by assumption. As one of the keys of our whole arguments, 
we prove that, when $X$ degenerates to $W$ in the flat family, the symplectic form $\sigma_X$ also extends along the 
family to a $T$-homogeneous symplectic form $\sigma_W$ on $W^{sm}$. 
Intriguingly, our proof of the existence of such extension 
is {\it by contradiction}, applying delicate Diophantine approximation of a 
certain irrational vector as in \cite{Schmid} (and classical \cite{Kronecker, Weyl}) to analyze the local  metric behaviour. 
As a result, every fiber of the flat family admits a Poisson structure. In particular,  $\mathcal{X} \to \mathbb{A}^1$ becomes a $\mathbb{G}_m$-equivariant Poisson deformation. We call such a deformation a scale up Poisson deformation, which is the Poisson realization of the 1-st step  degeneration. 

We next construct a $T$-equivariant flat deformation  $$\mathcal{W}_D \to D$$ over a smooth curve $q \in D$ in such a way 
that the fiber over $q$ is $C$ and other fibers are all isomorphic to $W$. Here, $T$ acts on $\mathcal{W}_D$ fiberwise. 
A general fiber $W$ admits a $T$-homogeneous symplectic form $\sigma_W$ on $W^{sm}$. We again prove that, when $W$ degenerates to 
$C$ in the flat family, the symplectic form $\sigma_W$ also degenerates to a $T$-homogeneous symplectic form $\sigma_C$ on 
$C^{sm}$. Namely $\mathcal{W}_D \to D$ becomes a $T$-equivariant Poisson deformation. This is the Poisson realization of the $2$-nd step  degeneration. 

Then, we prove that these two Poisson deformations have certain 
{\it rigidity} properties. Let $(X, x)^{\hat{}}$ be the formal completion of $X$ at $x$ and let $(W, 0)^{\hat{}}$ be the formal completion of $W$ at $0$. The symplectic forms $\sigma_X$ and $\sigma_W$ respectively determine Poisson structures on $(X, x)^{\hat{}}$ and $(W, 0)^{\hat{}}$. Then the scale up Poisson deformation $\mathcal{X} \to \mathbb{A}^1$ induces a trivial Poisson deformation of $(W, 0)^{\hat{}}$. In particular, there is an isomorphism 
$(X, x)^{\hat{}} \cong (W, 0)^{\hat{}}$ of formal Poisson schemes. On the other hand, the $T$-equivariant Poisson deformation $\mathcal{W}_D \to D$ induces a trivial Poisson deformation of $C$ itself and hence 
a $T$-equivariant isomorphism $W \cong C$ of conical symplectic varieties. 
The proofs of these depend on the work of \cite{Nama, Namb, Namf}. 
Therefore, we have an isomorphism 
$(X, x)^{\hat{}} \cong (C, 0)^{\hat{}}$ of formal Poisson schemes. By the Artin's approximation (\cite{Artin}), there is also an isomorphism 
$(X, x) \cong (C, 0)$ of analytic germs, which is nothing but the claim of the theorem. 

\subsection{Variant results}

Now, we explain more technical aspects of the statements and 
explain our variant theorems we prove in this paper. 
Recall that \cite{DS, DSII} requires the global assumption of 
Theorem \ref{Mthm.intro} type for their use of 
H\"ormander type construction 
of solutions of $\bar{\partial}$-equation 
with $L^2$-norm bounds. 
That is the only reason we (at the moment) assume in Theorem 
\ref{Mthm.intro} for 
$x\in X$ to be realized globally in 
$(\overline{X},L)$. 

Note that \cite{DSII} as a theory of singular 
K\"ahler metric in differential geometry, 
is expected to extend to more general normal log terminal singularities 
(cf., e.g., \cite[\S 3]{Zha24} which also makes progress) 
as a folklore among experts. 

\begin{conj}[after \cite{DSII}]\label{conj:genDS}
For any germ of (kawamata) log terminal 
singularity $x\in X$, 
there are certain good singular Ricci-flat K\"ahler metric 
$g_X$ with which the Donaldson-Sun theory \cite{DSII} works. 
To be precise, $(x\in X,g_X)$  has a Donaldson-Sun degeneration data 
(to be defined in later Theorem \ref{DS.maps}), 
which realizes the unique metric tangent cone 
$\lim_{c\to \infty}(x,\in X,c^2 g_X,J_X)$ 
and the algebraic local conification (stable degeneration) 
in Theorem \ref{AGlc}. 
\end{conj}

By singular K\"ahler metric above, 
we mean a K\"ahler metric on 
the smooth locus which extends 
with (at least) 
bounded potential across 
the singular locus as in the 
pluri-potential theory (cf., e.g., 
\cite{EGZ, GZ}). 
See also related recent developments on the local metric existence 
(\cite{Fu, GGZ}) and on the understanding of regularity 
(e.g., \cite{CS2,Sze24, CCHSTT}). 

In this context, 
we discuss how much we can generalize Theorem 
\ref{Mthm.intro} 
to more local setup. 
From the structure of our arguments, 
Theorem \ref{Mthm.intro} naturally generalizes to 
the following form with essentially the same proof. 

\begin{Thm}[=Theorem \ref{Mthm2}]\label{Mthm2.intro}
Suppose a symplectic singularity 
$x\in X$ has 
a singular hyperK\"ahler metric $g_X$ and 
a holomorphic symplectic form $\sigma_X$ which is 
parallel with respect to $g_X$ on $X^{\rm sm}$, 
with which Conjecture \ref{conj:genDS} holds. 

Then, the same statements as Theorem \ref{Mthm.intro} holds: that is, the analytic germ of $x\in X$ is that of (canonical) 
affine conical symplectic variety $C$ at the vertex $0\in C\curvearrowleft (\mathbb{G}_m)^r$. 
Furthermore, $0\in C$ has a (singular) hyperK\"ahler cone  metric, 
which in particular induces a canonical action 
of 
the multiplicative group $\R_{>0}$, 
as rescaling up of the metric. This action is the restriction of  
the algebraic action of $(\C^*)^r$ 
via some embedding $\R_{>0}\hookrightarrow (\C^*)^r$ 
as Lie groups. 
\end{Thm}

In general, we also confirm that the rank $r$ is $1$ if further $x\in X$ is 
an isolated singularity, 
or more generally, if $g_X$ satisfies Claim \ref{claim:completevf3} 
(see the discussions for Theorem \ref{r1prop} for details). 

To provide interesting examples, in our subsection \S 
\ref{subsec:HKQ}, 
we show that {\it hyperK\"ahler quotients} 
\cite{HKLR} 
(e.g., 
Nakajima quiver varieties \cite{Nakajima} 
and toric hyperK\"ahler varieties \cite{Goto, BD, HSt}) conditionally 
satisfy Conjecture \ref{conj:genDS} so that the above Theorem 
\ref{Mthm2.intro} applies and we also confirm $r=1$ in these cases. 

Put simply, this Theorem \ref{Mthm2.intro}  
reduces (our stronger form of) the full resolution of 
Kaledin's conjecture to further generalization of 
of the differential geometric 
Donaldson-Sun theory \cite{DSII}, especially on 
Conjecture \ref{conj:genDS} (as there are recent partial progress 
mentioned above). 

Compared with existing results in differential geometry, the above results 
can be also seen as variants of a famous 
result of Hein-Sun 
\cite{HeinSun} about the local asymptotics of 
(singular) Ricci-flat K\"ahler metric on smoothable 
Calabi-Yau varieties. Indeed, our theorems 
imply the following statements as a differential geometric 
version as consequences of the later developments by 
\cite{ChiuSzek, Zha24}. 
We write an easier version here and Corollary 
\ref{cor:HSZha} later discusses stronger statements. 

\begin{Thm}[cf., Corollary \ref{cor:HSZha}]
In the setup of Theorem \ref{Mthm.intro}, 
if $x\in X$ has only isolated singularity, $g_C$ is 
polynomially close to $g_X$ in the  following sense: 
there is a local biholomorphism 
$\Psi\colon U_C \to X, 0\mapsto x$  where 
$U_C \subset C$ is an open  neighborhood of $0\in C$ 
and a 
positive real number $\delta$ 
which satisfy the following: 
\begin{align}
|\Psi ^*g_X -g_C|_{g_C}=O(r^{\delta}).
\end{align} Here, $r$ is the 
distance function on $C$ from $0\in C$ with respect to 
$g_C$. 
\end{Thm}

Lastly, we discuss again the 
obstructions to generalizing our results to arbitrary  
symplectic singularities. 
Aside from that we fully use Donaldson-Sun theory 
at the {\it metric level}, 
we have one more technical obstruction as follows. 
Note that for general $x\in X$, a priori 
there may be much flexibility of $\sigma_X$ due to the lack of 
(singular) Darboux type theorem. 
In the setup of Theorem \ref{Mthm.intro}, 
we have a unique (parallel)  $\sigma_X$ 
which extends to whole $\overline{X}$ (up to rescale) 
and our proof benefits from that particular property. 
If we simply work on germ of $x\in X$, 
we can not use 
Bochner-Weitzenb\"ock type theorem (cf., \cite[Theorem A]{Henri.etc}) 
to ensure the 
parallelness of general $\sigma_X$ with respect to $g_X$. These reexplain the necessity of  Conjecture \ref{conj:genDS}, which is assumed 
in Theorem \ref{Mthm2.intro}. 

\begin{Rem}[On holomorphic contact geometry]\label{LSconj}
Our existence results for the hyperK\"ahler cone metric on the germs of 
symplectic singularities also apply to the symplectification i.e., 
algebro-geometric cone of contact 
varieties, albeit conditionally i.e., modulo Conjecture \ref{conj:genDS} or 
under the assumption of Theorem \ref{Mthm.intro}. 

The resulting implications
can be viewed as some 
variants of the LeBrun-Salamon conjecture (\cite{Le, 
LeS}) that connects with the quarternionic K\"ahler geometry via twistor theory. 
On the other hand, our results and 
methods remain applicable even to (a priori) {\it singular} contact 
varieties in the sense of e.g., \cite{Namf, Smi}. 

Conversely, 
for an arbitrary symplectic singularity 
which satisfies Conjecture \ref{conj:genDS}, 
the quotient $(C\setminus 0)/\G_m$ 
by the $\G_m$-action 
is the one mentioned in the previous 
subsection \S \ref{sec:outline},  which 
corresponds to some $\xi'$ 
(details of its choice is in Theorem  
\ref{thm:extendtoW2} cf., also 
Theorem \ref{thm:extendtoW}). 
Then, the symplectic form $\sigma_W$ 
is $\G_m$-homogeneous 
by Theorems \ref{thm:extendtoW2} 
so that 
the quotient naturally underlies a 
klt log $\Q$-Fano pair of the form 
$(F,\Delta=\sum_{i}(1-\frac{1}{m_i})\Delta_i)$ 
for a normal projective variety $F$, 
prime divisors $\Delta_i$ and $m_i\in \Z_{>0}$. 
Thus, notably, the klt log $\Q$-Fano pair 
comes with the 
{\it singular contact structure} or 
{\it contact orbifold structure} 
(cf., \cite[\S 2]{Le}, \cite[C.16]{Buc}, \cite[4.4.1]{Name}). 
If $r=1$ (cf., Conjecture \ref{r1conj}, Theorem \ref{r1prop}) holds, 
i.e., $T\simeq \G_m$, 
the log $\Q$-Fano pair furthermore admits 
K\"ahler-Einstein metric with conical singularity and is 
log K-polystable 
(in the sense of \cite{Don11, OS}). 

Note that
it provides a new 
systematic way of constructing (non-homogeneous) 
log K-polystable log Fano pairs, 
without estimating nor calculating invariants for K-stability. 
We collect some of these direct consequences of our main arguments in 
Theorem \ref{r1prop} at the end of this paper, but 
we also hope to discuss further applications in this direction on a different 
occasion. 
\end{Rem}

\subsection{Organization of this paper}
This paper is organized as follows. Our proof of main theorems 
combine arguments of all sections. 
In section \ref{sec:2}, we review the 
Donaldson-Sun theory after \cite{DSII} 
and related later algebro-geometric 
developments in some details. We also prove 
some preparatory Lemma \ref{XtoW.C} to provide 
degeneration 
theoretic viewpoint on it, and 
set up necessary notation. 
Section \ref{sec:PD} is also of preparatory nature, in which 
we analyze some special kind of 
Poisson deformation using the theory of universal Poisson deformation 
by the first author and prove formal local rigidity. 
In section \ref{sec:XW}, by using the preparations, we prove 
$X$ and $W$ have isomorphic analytic germs, which in particular 
implies Kaledin's conjecture. 
The proof combines differential geometric arguments and 
various Diophantine approximations. 
In section \ref{sec:WC}, we prove $W=C$ in our situation. 
These arguments rely 
on singular hyperK\"ahler metrics. 

Finally, section \ref{sec:sum} culminates all the arguments 
to complete the proof of the main theorems. 
For that, we also prove the equivalence of 
existence of symplectic resolution and smoothablity for 
polarized projective symplectic varieties. In Example 
\ref{O'Grady10}, we apply Theorem \ref{Mthm.intro} to O'Grady's 10-folds. 
In \S \ref{subsec:HKQ} 
we discuss hyperK\"ahler quotients as possible source of 
examples 
to which Theorem \ref{Mthm2.intro} and hence its 
corollaries apply. In the last subsection, 
we collect other consequences of our main arguments e.g., 
quasi-regularity, canonical $SU(2)$-action, construction of contact 
singular K\"ahler-Einstein varieties. 


\section{Review and preparation of 
Donaldson-Sun conification}\label{sec:2}

Now we review the theory of Donaldson-Sun \cite{DSII}, which gives some canonical 
modifications of $x\in X$ to the local metric tangent cone of 
singular K\"ahler metrics, which they prove to be unique (cf., also \cite{ColdingMinicozzi}). 

\subsection{The original theory of Donaldson-Sun}
Here is a somewhat simplified 
summary of the original theory of Donaldson-Sun 
\cite{DSII}. 

\begin{Thm}[\cite{DSII}]\label{DSII.thm}
Suppose that $x\in \overline{X}$ is a complex $n$-dimensional projective log terminal variety with 
$K_{\overline{X}}\equiv 0$, which is given a (weak) Ricci-flat K\"ahler metric $g_X$
as the non-collapsing polarized limit space (\cite{DS}) of some polarized smooth Ricci-flat projective varieties. We take an open affine neighborhood of $x$ as $X\subset \overline{X}$. 

Then, the local metric tangent cone i.e., 
the pointed Gromov-Hausdorff limit of $(x\in X,c^2 g_X)$ for $c\to \infty$ 
is a (singular) Ricci-flat K\"ahler cone $C_x(X)$ which has a description 
in terms of $2$-step degeneration 
$(x\in X)\rightsquigarrow (0\in W)\rightsquigarrow (0\in C_x(X)=:C)$.  
\end{Thm}

Note that in this case, klt singularity 
$x\in X$ admits a (nice) Ricci-flat K\"ahler 
metric, hence ``stable" enough from the perspective of 
Yau-Tian-Donaldson correspondence. 
The main point of the above process in Theorem \ref{DSII.thm} is to 
convert such $(x\in X,g_X)$ 
to (tangent) {\it cone} in differential geometric sense. 
For finer details, we prepare the following notation. 
\begin{ntt}\label{ntt1}
We often abbreviate $C_x(X)$ simply as $C$. 
Note that its smooth locus 
$C_x(X)^{\rm sm}$ is a metric cone
in the sense it can be written as 
$(S\times \R_{>0},r^2 g_S+(dr)^2=:g_C)$ with some $2n-1$-dimensional Riemannian (Sasakian)  manifold $(S,g_S)$ and 
the coordinate $r$ of $\R$-direction, 
has the Reeb vector field $\xi=Jr\partial_r$. 
Further, this $C_x(X)$ has an embedding by holomorphic functions $f_i (i=1,\cdots,l)$ 
which are homogeneous of degree $w_i$ with respect to the natural $\R_{>0}$-action. 
After this embedding, $\xi$ can be written as 
${\rm Re}(\sqrt{-1}\sum_{i=1,\cdots,l}w_i z_i \partial_{z_i})$. 
By this reason, we often identify $\xi$ with the vector $(w_1,\cdots,w_l)$, which we also write $w(\xi)$ for distinction. 

As \cite[(around) Lemma 2.17]{DSII} explains, 
this Reeb vector field is a holomorphic Killing field on 
$S$ and $C_x(X)^{\rm sm}$ and 
generates a subgroup in the isometry group of $C_x(X)$, 
and its closure can be complexified into a (complex)  algebraic torus $T^{\rm an}:=T(\C):=N\otimes_{\Z} \C^*$ for some lattice $N$. 
We write the dual lattice of $N$ as $M$. 
We set $T:=N\otimes \G_m$, 
which means ${\rm Spec}\C[M]$ with the group ring 
$\C[M]$, 
as an algebraic torus over $\C$. 
Its cocharactor lattice is $N$. 
We often do not distinguish $T$ and $T^{\rm an}=T(\C)$ 
if there is no confusion. 
We denote its rank by $r(\xi)$. 
It is nothing but the rational rank of $\sum_{i=1}^l \Q w_i$ 
(we also sometimes abbreviate it as $r$ if there is no confusion.) 
This fact easily follows from basic linear algebra or 
the continuous version of the classical Kronecker-Weyl 
equidistribution theorem. 

Now, we define $G_\xi$ to be the closed subgroup of 
$GL(l,\C)$ as the commutator of the torus $T$ (or $\xi$) which is reductive. 
The difference of two elements 
is defined as an element of ${\rm End}_{\C}(\C^l)=\mathfrak{gl}(l,\C)$. 
We also set a diagonal matrix $$\Lambda:={\rm diag}(\sqrt{2}^{w_1},\cdots,\sqrt{2}^{w_l})\in G_\xi.$$ 
\end{ntt}

For our main arguments, 
we need much more precise details of what 
\cite{DSII} proves, 
which we review. It is a little lengthy, but 
we need all the details for our application in this paper. 

\begin{Thm}[Local conification \cite{DSII}]\label{DS.maps}
For the above set-up $x\in \overline{X}$ together with $g$, 
there is the following set of data: 
\begin{itemize}
\item 
an (algebraic) re-embedding of 
the germ of $x$ into $\C^l$ which we denote as $\Phi=\Phi_0\colon X\hookrightarrow \C^l$, 
\item 
positive real weights for each coordinates 
$w_1,\cdots,w_l\in \R_{>0}$ (and the corresponding 
$\Lambda$ and $G_\xi$ as Notation \ref{ntt1}), 
\item a 
sequence of 
    $\Lambda_j=\Lambda\cdot E_j\in G_\xi$ 
with $G_\xi\ni E_j\to {\rm Id}$, 
\end{itemize}
such that the following hold 
(in this case, we call the above set of data 
{\it Donaldson-Sun degeneration data} 
\footnote{the term ``degeneration data" is modeled 
after 
Faltings-Chai \cite[II \S0, III \S2]{FaltingsChai}, though 
there are certain substantial differences in the setups.}
of $(x\in X, g_X)$): 
\begin{enumerate}
    \item \label{DSi} The limit 
    $W=\lim_{j\to \infty}\Lambda^j (X)\subset \C^l$ as 
    Hausdorff convergence (see \cite[\S 3.2]{DSII} in particular) 
    is a normal affine variety with a natural 
    good $T$-action 
    (together with a positive vector field 
    ${\rm Re}(\sqrt{-1}\sum_{i=1,\cdots,n}w_i z_i\partial_{z_i})|_{W^{\rm sm}}$, 
    actually a K-semistable $\Q$-Fano cone 
    in the sense of \cite{CS, CS2}, see Theorem \ref{AGlc} and \cite[\S 2]{Od24a}). 
    \item \label{DSii}
    (See \cite[p.346 \& 3.14]{DSII} in particular) 
    For $\Phi_i:=(\Lambda_i\circ \Lambda_{i-1}\circ \cdots \Lambda_1 \circ \Phi)$ 
    for $i\ge 1$, there is a limit 
    \begin{align*} 
    C&=\lim_{j\to \infty}(\Lambda_i\circ \Lambda_{i-1}\circ \cdots \Lambda_1 \circ \Phi)(X)\subset \C^l\\ 
    &=\lim_{i\to \infty}(E_i\circ E_{i-1}\circ \cdots \circ E_1)(W)\subset \C^l
    \end{align*}
    both as 
    Hausdorff convergence. 
    \footnote{Note that $E_i\circ E_{i-1}\circ \cdots \circ E_1$ does not necessarily converge in $G_\xi$, 
    as it indeed does not if $W\not\simeq C$.} 
    Further, $C$ is again a normal affine cone (a K-polystable $\Q$-Fano cone, see Theorem \ref{AGlc}) 
    with the natural 
    ${\rm Re}(\sqrt{-1}\sum_{i=1,\cdots,n}w_i z_i\partial_{z_i})|_{C^{\rm sm}}$, has a 
    (weak) Ricci-flat K\"ahler cone metric $g_C$,  
    and $(0\in C,g_C)$ realizes the unique metric tangent cone of 
    $(x\in X,g_C)$. 
    
    We set $X_i:=\Phi_i(X)$ and $W_i:=\lim_{j\to \infty}\Lambda^j X_i
=(E_i\circ E_{i-1}\circ \cdots E_1)(W)$ for $i,j=1,2,\cdots$. 
In \S \ref{sec:WC}, we analyze $\Lambda^j (X_i)$ for a priori different 
$i$ and $j$s, which justifies the usefulness of two subindices as 
later convenience. 

    \item \label{DSiii} The above convergence $X\rightsquigarrow C$ realizes the 
    (polarized) limit space in the sense of \cite[p.330]{DSII} (cf., also \cite{DS}) 
    and 
    in particular, it is a Cheeger-Gromov convergence at the regular locus i.e., for any 
    compact subset $K\subset C^{\rm sm}$, 
    there is an open neighborhood of it $(K\subset) U_C\subset C^{\rm sm}$, there is a sequence of diffeomorphisms  $\Psi_i\colon U_C \hookrightarrow  \Phi_i(X^{\rm sm}) (i=1,2,\cdots)$ onto their images 
     such that the following holds. Here, $X^{\rm sm}$ is the smooth locus of $X$. 
    \begin{enumerate}
    \item \label{DSiiia} 
    $\Psi_i\to {\rm Id}$ for $i\to \infty$ as $C^\infty$-maps and 
    \item \label{DSiiib} 
    \begin{itemize}
    \item $2^i \Psi_i^*((\Phi_i^{-1})^* g)\to g_C$, 
    \item $2^i \Psi_i^*((\Phi_i^{-1})^*\omega_X)\to \omega_C$, 
    \item $\Psi_i^*((\Phi_i^{-1})^*J_X)\to J_C $ 
    \end{itemize}
    on 
    $U_C$ for $i\to \infty$. Here, $\omega_X$ (resp., $\omega_C$)  
    is the K\"ahler form for $g$ on $X^{\rm sm}$ (resp., $g_C$ on 
    $C^{\rm sm}$) and $J_X$ (resp., $J_C$) is the complex structure on $X^{\rm sm}$ (resp., $C^{\rm sm}$). 
    \end{enumerate}
\end{enumerate}
\end{Thm}

\begin{Rem}\label{balltoaffine}
To be precise, \cite{DSII} focuses on local embeddings of smaller 
analytic neighborhoods of $x\in X$ to $\C^l$ rather than the above $\Phi_i$, 
e.g., $B_i$ in their notations, 
but since we prefer to 
work on more algebraic categories, we use slight generalization 
as above affine version for our convenience. 
Clearly, this slight extension is non-substantial and 
straightforward from their work in {\it loc.cit} 
by taking $J_k$ in its Proposition 3.14 inside $\Gamma(X,\mathcal{O}_X)$ 
and $k_0$ after that (p.351, before 3.15) large enough 
so that $P$ gives global embedding i.e., affine embedding of 
$X$ to $\C^l$. 
Construction of such $J_k$ follows immediately once we replace each $I_k$ 
of their (3.3) by $I_k\cap \Gamma(X,\mathcal{O}_X)$ and do the same arguments afterwards. 
\end{Rem}

Throughout the paper, 
we use the above notation of the Donaldson-Sun degeneration data 
$$x\in X\overset{\Phi=\Phi_0}{\hookrightarrow} \C^l, 
w_1,\cdots,w_l, \Lambda\in G_\xi, E_j, 
\Lambda_i=\Lambda\cdot E_i\in G_\xi,$$ 
and the resulting 
$$W,X_i, C=C_x(X),W_i,\Phi_i,\Psi_i.$$ 

\subsection{Algebro-geometric version}\label{subsec:AGlc}
As the original Donaldson-Sun \cite{DSII} conjectured (see {\it loc.cit} 3.22), this 
conification process 
$$(x\in X,g_X,J) \rightsquigarrow (0\in W)\rightsquigarrow (0\in 
C_x(X)=:C)$$ 
is actually independent of the metric $g_X$ and 
determined locally only by the analytic (formal) germ of $x\in X$ 
in the case of the setup of Theorem \ref{DSII.thm} 
as confirmed by \cite{LWX}. The proof combines the theory of 
K-stability of Fano cones \cite{CS}, that 
of local normalized volume by \cite{Li}, and related works. 
Some intermediate developments can be reviewed in e.g., \cite{LLX}, 
\cite[\S 2]{Od24a}. It is 
called stable degeneration
\footnote{In \cite{Od24a,Od24c}, 
the second author proposes another name ``(algebraic local) conification"   as we regard $x\in X$ as already stable object, reflecting the existence of local K\"ahler-Einstein metrics.} 
in e.g., \cite{LLX}. The following is its  brief summary. 

\begin{Thm}[cf., \cite{Li, Blum, CS, LX, LWX, Xu, XZ}]\label{AGlc}
For any klt variety and its closed point $x\in X$, 
there exists a unique valuation $v_X$ of $\mathcal{O}_{X,x}$ with the center $x$ 
(\cite{Blum}), which minimizes the 
normalized volume $\widehat{\rm vol}(-)$ of \cite{Li}. 
It 
is quasi-monomial (\cite{Xu}) 
and 
$${\rm gr}_{v_X} \mathcal{O}_{X,x}:=\oplus_{a\in \R_{\ge 0}}\{f\in \mathcal{O}_{X,x}\mid v_X(f)\ge a\}/\{f\in \mathcal{O}_{X,x}\mid v_X(f)> a\}$$ 
is a finite type $\C$-algebra which gives a 
K-semistable (in the sense of \cite{CS}) 
log terminal singularity $W:={\rm Spec}({\rm gr}_{v_X} \mathcal{O}_{X,x})$ with 
an algebraic torus $T=N\otimes \G_m$ action with the lattice $N\simeq \Z^r$ (\cite{LX, XZ}). 
Here, $N$ is the dual lattice of the groupification of the 
image monoid of $v_X$. Further, the 
K-semistable Fano cone $T\curvearrowright 
W$ degenerates to a unique K-polystable Fano cone 
$T\curvearrowright C$ (\cite{LWX}) as a $T$-equivariant faithfully flat 
(see \cite[\S 2]{Od24a}) affine test configuration. 

In the setup of Theorem \ref{DSII.thm} proved by Donaldson-Sun 
\cite{DSII}, the constructions $X\rightsquigarrow W\rightsquigarrow C$ 
coincide with that in 
Theorem \ref{DSII.thm}. 
\end{Thm}

Moreover, $v$ is nothing but 
the function $d(-)$ defined in \cite[(3.1)]{DSII} and 
its properties as above follow from Theorem \ref{DSII.thm}, 
\ref{DS.maps} then (cf., also \cite[Appendix C]{HS}). 
An important technical point for us is that, for general local  
singular 
Ricci-flat K\"ahler metrics, we do {\it not} 
know (yet) if this process comes with 
Donaldson-Sun degeneration data (see Conjecture \ref{conj:genDS}). 
For this reason, Theorem \ref{DSII.thm} gives more 
information especially on the metrics, 
which we use crucially to fit to the theory of Poisson deformation 
later. 

For the first step degeneration 
$X\rightsquigarrow W$ of general Theorem \ref{AGlc}, 
we also prepare the following description 
in terms of families and weighted blow ups, refining some discussions  
of \cite{LX, Od24b}. Here, we do {\it not} assume this process 
$x\in X\rightsquigarrow W\rightsquigarrow C$ has a 
Donaldson-Sun degeneration data of the previous section (Theorem \ref{DS.maps}). 

\begin{Lem}\label{XtoW.C}
\begin{enumerate}

\item (cf., \cite[\S 3]{LX}, \cite[Theorem 2.12]{Od24b})
There is a closed (algebraic) 
embedding $X\hookrightarrow \A^l$ 
with the coordinates $z_1,\cdots,z_l$ and 
consider the corresponding embedding 
$X\times \A^1_t\hookrightarrow \A^l\times \A^1_t$. 
Using this, the degeneration $X\rightsquigarrow W$ is rewritten as an 
affine faithfully-flat family 
$\pi_\sigma\colon \mathcal{X}_\sigma\to U_\sigma$ 
over an affine toric variety $U_\sigma$ 
for a certain rational polyhedral cone $\sigma\subset N_{\R}=N\otimes \R$ with the lattice $N$.  
The fiber over the torus invariant point $p_\sigma\in U_\sigma$ is $W$, on which $T=N\otimes \mathbb{G}_m$ acts, 
and the fibers over torus $N\otimes \G_m$ are $X$. 
We call this type of degeneration {\it generalized test configuration} in \cite[\S 2]{Od24b} 
and the above particular one is called scale up deformation in {\it loc.cit}.

\item \label{W.wtedblowup}

We take any small 
enough cone $\sigma\subset N\otimes \R$, 
and consider any $(\vec{0}\neq)\xi'=(w'_1,\cdots,w'_l)\in N\cap \sigma$ and the 
associated toric morphism $f_{\xi'}\colon \A^1\to U_\sigma$ 
with the natural inclusion $\Z_{\ge 0} \xi' \to \sigma$. 
If we take the pullback of the family, then we obtain an affine test configuration $\mathcal{X}_{\xi'}$ 
of $X$ with the central fiber $W$ 
(but with a $\C^*$-action which depends on $\xi'$). 
In particular, the restriction of 
$\X_{\sigma} \to U_{\sigma}$ 
to the toric boundary of $U_\sigma$ is a $W$-fiber bundle. 

$\X_{\xi'}$ is a Zariski open subset of (or affine version of) 
the weighted blow up of $X\times \A^1$ 
with respect to weights $(w_1,\cdots,w_l,1)$ 
for the coordinates $z_1,\cdots,z_l,t$ for some positive 
integers $w'_1,\cdots,w'_l$. 
Also, $K_{\X_{\xi'}}$ is $\Q$-Cartier. 
In particular, this is a scale up test configuration 
in the sense of \cite[\S 2]{Od24b}. 

\item \label{C.wtedblowup}

There is another affine faithfully-flat 
$\G_m$-equivariant degenerating family 
$\X_C\to \A^1$ of $X$, whose central fiber is 
$C=C_x(X)$ with the action of $\G_m$ as a subtorus of $T$ 
and is a scale up test configuration again.  
(However, we should remark 
rightaway that if $W\neq C$, then the 
obtained $\G_m$-action on the central fiber is 
not the Reeb vector field nor even inside $T$). 

\end{enumerate}
\end{Lem}
Some parts of the following proof will be also used in later sections. 
\begin{proof}
The main idea of the following discussions are already contained in 
\cite{LX, Od24a, Od24c} but we recall and complete for the convenience. 
Take a homogeneous generator system $\overline{z}_i (i=1,\cdots,l)$ 
of 
${\rm gr}_{v_X}(\mathcal{O}_{X,x})$ of weights $w_i>0$ and 
lift them to $z_i\in \Gamma(\mathcal{O}_X)$. 
Note that ${\rm gr}_{v_X}(\mathcal{O}_{X,x})$ has a canonical 
action of the algebraic torus $N\otimes \G_m$ where 
$N$ is the dual of the groupification of the image of $v_X$. 
Taking enough 
$z_i$s we can and do assume this gives an embedding 
$X\hookrightarrow \A^l$ and this $\A^l$ 
is naturally acted by $T=N\otimes \G_m$ (see \cite[\S 2.3]{DSII}) 
and that $W$ is also embedded to $\A^l$ through those $\overline{z}_i$s. 
For \eqref{C.wtedblowup} and our later use, 
we can and do assume that $W\rightsquigarrow C=C_x(X)$ is also realized 
as a test configuration inside $\A^l\times \A^1$. 
Since the coordinates of $\A^l$ are $T$-homogeneous, 
there is a natural homomorphism $w\colon N\to \Z^l$, 
which is injective because of the effectivity of the $T$-action 
on $X$ (and hence on $\A^l$). It also 
extends to 
\begin{align}\label{def:w}
w\colon N_{\R}\hookrightarrow \R^l.
\end{align}

We take its 
universal Gr\"obner basis of the defining ideal of 
$X\subset \A^l$ 
as $\{F_j\}_{1\le j\le m}$ (\cite[\S 5.1]{BCRV}). 
Then, $\Gamma(\O_W)$ can be written as 
$\C[z_1,\cdots,z_l]/\langle \{{\rm in}_{w_i}(F_j) \}_j \rangle$, 
where ${\rm in}_{w_i}(F_j)$ means the 
initial term of $F_j$ with respect to the weights $w_i$s. 
We denote the weight vector $(w_1,\cdots,w_l)$ as $w(\xi)$ or 
simply as $\xi$ since it naturally corresponds to the 
Reeb vector field $\xi$ via the map $w$ (cf., Notation \ref{ntt1} and the item 
\eqref{def:w} above). 
Set some rational polyhedral cone $\sigma$ of $N_{\R}$ 
which contains $\xi$ so that 
$(w_1,\cdots,w_l)\in w(\sigma)\subset w(N_{\R})$ 
and consider the $T$-equivariant 
morphism $\A^l \times T 
 \to T$. Here the algebraic torus 
 $T$ acts on $\A^l$ and, hence naturally acts on $\A^l \times T$. Now we consider the subvariety $T(X \times \{1\}) \subset \A^l \times T$ and take its closure 
 $\overline{T(X \times \{1\})}$ 
 inside $U_\sigma\times \A^l$ and denote it by 
 $\X_{\sigma}$. By definition we have an inclusion map $\X_{\sigma} \to \A^l \times U_\sigma$. 

Now we prove \eqref{W.wtedblowup}. Take small enough $\sigma$ which still contains 
$\xi$ so that the ${\rm in}_{\xi'}(F_j)$s do not 
change. 
One can take a toric morphism $f_{\xi'}\colon \A^1\to U_\sigma$ 
with $\xi' \in N\cap \sigma$. 
Then the affine test configuration $\mathcal{X}_{\xi'}$ 
of $X$ has the central fiber $W$ 
(but with a $\C^*$-action which depends on $\xi'$). 
In particular, the restriction of 
$\X_{\sigma}\to U_{\sigma}$ 
to the toric boundary of $U_\sigma$ is a $W$-fiber bundle. 
Hence we obtain the claim of the first paragraph. 

For simplicity, we denote $\mathcal{X}_{\xi'}$ as $\mathcal{X}$ during this proof. 
Note that this is naturally the closure of 
$\overline{\{(t^{w'_1},\cdots,t^{w'_l})\cdot (X\times 1)\mid t\in \C^*\}}$ 
in $\A^l_{z_1,\cdots,z_l}\times \A^1_t$. 
Hence, it is an open affine subset of the weighted blow up of 
$X\times \A^1$ with the weights 
$w'_1,\cdots,w'_l,1$. This completes the proof of the first assertion of 
(ii) and it is some easier analogue to \cite[\S 2.4, Lemma 2.25]{Od24c}. 
The last statement of $\Q$-Gorensteinness of the family $\X_{\xi'}$ 
(i.e., $\Q$-Cartierness of $-K_{X\times (\A^1\setminus \{0\})}$) 
follows from the 
the above description as an 
affine weighted blow up of 
$X\times \A^1$ together with the Cartierness of the exceptional divisor 
(cf., similar arguments in \cite[\S 2, Lemma 2.4]{OSS}). 

To show \eqref{C.wtedblowup}, 
we use \eqref{W.wtedblowup}  which implies that 
for certain 
large enough embedding $X\hookrightarrow \A^l$, 
there is a one parameter subgroup 
$\mu\colon \G_m\to T\subset {\rm GL}(l)$ 
which induces the test configuration $\X\to \A^1$. 

On the other hand, recall from the arguments 
\cite[p.354]{DSII} using an analytic slice theorem 
(also later reproved 
algebraically by \cite{LWX}) that  
there is another one parameter subgroup $\lambda$ 
of the commutator of $T=N\otimes \G_m$ 
which induces a test configuration $\mathcal{T}$ of $W$ 
degenerating to $C=C_x(X)$. 

Now we glue these $\X$ and $\mathcal{T}$, 
as in \cite[3.1]{LWX} (also \cite[proof of 4.5]{grav}) 
in the sense that we consider the embedding 
$X\subset \A^l$ and act $\G_m$ by $\lambda \cdot \mu^m$ with 
$m\in \Z_{>0}$ 
as they commute in ${\rm GL}(l)$. 
As in the same arguments as {\it loc.cit}, 
it is a test configuration of $X$ degenerating to 
$C$ for $m\gg 1$ which we denote as $\X_C\to \P^1$. 
Since $\X_C\to \P^1$ is a deformation of 
$\mathcal{X}\cup \mathcal{T}\to \A^1\cup \A^1$ (cf., {\it op.cit}) 
which can be anti-pluri-canonically polarized by the 
$\Q$-Cartierness of $-K_{\X_{\xi'}}$ as proved in 
\eqref{W.wtedblowup}, it follows that $\X_C\to \A^1$ is again a 
$\Q$-Gorenstein family (cf., \cite[\S 2, Lemma 2.4]{OSS}). 
Moreover, it easily follows from the construction that 
this is again a scale up test configuration. 

\end{proof}

We often use this construction in the above proof, 
throughout the paper. Here are some other conventional remarks. 
\begin{itemize}
    \item
We denote the weighted blow up obtained in 
\eqref{W.wtedblowup} 
as $b\colon \overline{\X}\to X\times \A^1$ 
and its restriction to an open subset $\X$ as $b^o$, 
and the projection $\overline{X}\to\A^1$ 
(resp., $\X\to \A^1$) 
as $\overline{\pi}$ (resp., $\pi$). 
\item 
Also, if there is no fear of confusion, we sometimes (but rarely) 
identify $\xi$ with $w(\xi)\in \R^l$ and write the latter simply as $\xi$. 
So is the case for an approximation of $\xi'\in N\otimes \R$ and its image $w(\xi')\in \R^l$. 
\end{itemize}

\subsection{Real analytic or sequencial version}
Motivated by the materials of 
Theorem \ref{DS.maps} proven in \cite{DSII} 
and Lemma \ref{XtoW.C}, 
we introduce the following terminlogy to fit real analytic 
degenerations or degenerating sequences to the family 
$\mathcal{X}_{\sigma}\to U_{\sigma}$. 
The point is that, if $\xi$ is rational, we can consider the 
family 
$\X_{\xi}\to \A^1$ and regard it as whole $\X$ but in the case $r(\xi)>1$, 
we want at least some real analogue. 

\begin{ntt}\label{ntt2}
We somewhat follow \cite[\S 2]{Od24a} here. 
For each $\tau\in \R_{>0}$ and 
positive vector field $\xi\in N_{\R}$  
$\Lambda_{\tau}=\Lambda_{\tau}(\xi)\in 
{\rm GL}(l,\R)$ is the diagonal matrix 
$({\rm diag}(\tau^{w_1},\cdots,\tau^{w_l})).$ 
Then, as a subvariety of $\C^l$, we define $$X_\tau:=\Lambda_{\tau}^{-1}(X)\subset \C^l.$$

\end{ntt}

Take the neighborhood rational 
regular polyhedral cone $\sigma\ni \xi$ 
to apply the proof of Lemma \ref{XtoW.C}, 
with the generator of $\sigma\cap N$ as 
$\vec{v}^{(i)}=(w_1^{(i)},\cdots,w_{r(\xi)}^{(i)})$  
for $j=1,\cdots,r(\xi))$, and write 
$\xi=\sum_{1\le i\le r(\xi)}c_i \vec{v}^{(i)}$ 
with $c_i\in \R$. 

Then, we can identify 
$U_\sigma$ as $\C^{r(\xi)}$ by using $\vec{v}^{(j)}$s, 
and  fits $X_\tau\subset \C^l$ 
 into $\pi_\sigma\colon 
\mathcal{X}_{\sigma}\to U_{\sigma}$ 
by the identification 
\begin{align}\label{Xtau}
X_\tau=\pi_\sigma^{-1}
(\tau^{c_1},\cdots,\tau^{c_{r(\xi)}})\subset \C^l.
\end{align}

Recall that $X_{1/\sqrt{2}^{i}}$  is somewhat close to $\Phi_i(X)$ 
in Theorem \ref{DS.maps} for intermediate $i$s, 
but a priori different 
due to the possible gap between $E_i$ and ${\rm Id}$. The former converges to $W$, while the latter converges to $C=C_x(X)$ 
as Hausdorff convergence. 



\section{Scale up Poisson deformations}\label{sec:PD}

In this section, we shall make an algebraic  
study of Poisson deformations of conical symplectic varieties. In  particular, we focus on 
a certain type of Poisson deformations which we call a 
{\em scale up} Poisson deformation, that appears as a 
$\C^*$-equivariant isotrivial degeneration 
(so-called test configuration, after Mumford and Donaldson). 

Let $W$ be an algebraic symplectic variety. By definition $W^{\rm sm}$ admits an algebraic 
symplectic form $\sigma_W$, which we fix. 
Then $\sigma_W$ identifies 
the holomorphic tangent sheaf 
$\Theta_{W^{\rm sm}}$ with 
the sheaf of holomorphic $1$-form 
$\Omega^1_{W^{\rm sm}}$; hence $\wedge^2 \Theta_{W^{\rm sm}}$ with $\Omega^2_{W^{\rm sm}}$. We have a $2$-vector 
$\theta_W$ on $W^{\rm sm}$ corresponding to $\sigma_W$. This $2$-vector $\theta_W$ is called a {\it Poisson $2$-vector}. By using the 
Poisson $2$-vector, we define a  Poisson bracket $$\{\:, \:\}_{W^{\rm sm}} \colon \mathcal{O}_{W^{\rm sm}} \times \mathcal{O}_{W^{\rm sm}} \to 
\mathcal{O}_{W^{\rm sm}}, \:\:\: (f, g) \to \theta_W(df \wedge dg).$$ Note that the $d$-closedness of $\sigma_W$ is equivalent to the Jacobi 
identity of the bracket $\{\:, \:\}_{W^{\rm sm}}$. Since $W$ is normal, the bracket uniquely extends to a Poisson 
bracket $\{\:, \:\}_W$ on $W$. Conversely, if we are given a Poisson bracket on a normal variety $W$ which is non-degenerate 
on $W^{\rm sm}$ (that is, the corresponding Poisson $2$-vector $\theta_W$ on $W^{\rm sm}$ is non-degenerate), then $W^{\rm sm}$ admits 
a holomorphic symplectic 2-form $\sigma_W$.  

Let $f: \mathcal{X} \to S$ be a morphism of algebraic schemes over $\C$. If we are given an $\mathcal{O}_S$-linear Poisson bracket $\{\:, \:\}_{\mathcal{X}} : \mathcal{O}_{\mathcal X} \times \mathcal{O}_{\mathcal X} \to \mathcal{O}_{\mathcal X}$, then 
$(\mathcal{X}, \{\:, \:\}_{\mathcal{X}})$ is called a Poisson scheme over $S$. Let $0 \in S$ be a closed point and assume that 
$f$ is a flat surjective morphism. A Poisson scheme $(\mathcal{X}, \{\:, \:\}_{\mathcal X})$ over $S$ is called a 
{\it Poisson deformation} of $(W, \{\:, \:\}_W)$ if there is a Poisson isomorphism $$\phi\colon (W, \{\:, \:\}_W) \cong (f^{-1}(0), \{\:, \:\}_{\mathcal X}\vert_{f^{-1}(0)}).$$ More precisely, a Poisson deformation is a pair $(\mathcal{X}, \phi)$ of the Poisson scheme $\mathcal{X}$ over $S$ and the Poisson isomorphism $\phi$. Two Poisson deformations $(\mathcal{X}, \phi)$ and $(\mathcal{X}', \phi')$ of $W$ over the same base $S$ is called equivalent if there is a Poisson $S$-isomorphism $\Psi\colon (\mathcal{X}, \{\:, \:\}_{\mathcal X}) \cong (\mathcal{X}', \{\:, \:\}_{\mathcal{X}'})$ such that $\Psi \circ \phi = \phi'$.  
\vspace{0.2cm}

Now let us consider a conical symplectic variety $(W, \sigma_W)$ with the origin $0_W \in W$. Denote the weight ${\rm wt}(\sigma_W)$ as $l$ which is positive; in other words, $\{\:, \:\}_W$ has weight $-l$. 
\begin{defn}[Scale-up Poisson deformation]\label{def:SPD}
Let  
$f\colon (\mathcal{X}, \{\:, \:\}_{\mathcal X}) \to \A^1$ be a $\mathbf{C}^*$-equivariant Poisson deformation of $(W, \{\:, \:\}_W)$
such that $\mathcal{X}$ is affine and  
\begin{enumerate}[label=\arabic*)]\label{SUPD}
\item  $\A^1$ has a negative weight, i.e. there is a positive integer $w$ and $\G_m$ acts on $\A^1 = \mathrm{Spec}\: \C[t]$ 
so that 
$t \mapsto \lambda^{-w}t$  for $\lambda \in \G_m(\C)=\C^*$. 

\item  $\{\:, \:\}_{\mathcal X}$ has weight $-l$. 

\item  \label{SUPD.pos} There is a $\G_m$-invariant section $\Gamma \subset \mathcal{X}$ of $f$ 
such that $\Gamma \cap f^{-1}(0) = 0_W$ and every 
$\G_m$-orbit of $\mathcal{X}$ whose closure contains $0_W$ is $\Gamma -\{0_W\}$ or a $\G_m$-orbit in $f^{-1}(0)$.
\end{enumerate}
\vspace{0.2cm}

In this article, such a Poisson deformation is called a {\em scale-up Poisson deformation} after \cite[\S 2]{Od24b}, as 
the degeneration $f^{-1}(t)\rightsquigarrow f^{-1}(0)$ for $t\to 0$ 
is obtained by scaling up $\C^*$-action. 
\end{defn}
In the remainder, we restrict ourselves to a scale-up Poisson deformation of a conical symplectic variety. 

A typical example of a scale-up Poisson deformation is constructed as follows. 
Let us consider the trivial Poisson deformation ${\rm pr}_2: W \times \A^1 \to \A^1$ of $W$. We introduce a 
$\G_m$-action on $W \times \A^1$ by $(x, t) \to (\lambda\cdot x, \lambda^{-w}t)$, $\lambda \in \G_m$. 
Here $\cdot$ denotes the $\G_m$-action on $W$.        
Then ${\rm pr}_2$ is a $\G_m$-equivariant Poisson deformation of $W$ satisfying the conditions 1), 2) and 3). We can take 
the $\G_m$-invariant section $\Gamma$ of ${\rm pr}_2$ as an obvious choice 
$0_W \times \A^1 \subset W \times \A^1$.

Given an arbitrary $\mathcal{X}$ with the properties 1), 2) and 3), we compare $f\colon \mathcal{X} \to \A^1$ 
with ${\rm pr}_2\colon W \times \A^1 \to \A^1$. 
As the following example shows, we caution that 
they are {\it not} globally isomorphic in general. 

\begin{ex} 
Put $W := \mathrm{Spec}\: \C[x_1, x_2]$, 
$\sigma_W := dx_1 \wedge dx_2$, 
with the weights 
${\rm wt}(x_1) = {\rm wt}(x_2) = 1$ 
and ${\rm wt}(t) = -1$. Define $\mathcal{X} := \mathrm{Spec}\: \C[x_1, x_2, t, \frac{1}{x_1t - 1}]$ and regard it as a Zariski open subset 
of $W \times \A^1$. Then both $\mathcal{X}$ and $W \times \A^1$ are scale-up Poisson deformations of $W$, but 
they have different general fibers. 
\end{ex}

Nevertheless, the following formal local triviality theorem holds. 
\vspace{0.2cm}

\begin{Thm}\label{formallocaltriv}
For any scale-up Poisson deformation 
$f\colon \mathcal{X}\to \A^1$ 
of a conical symplectic variety $W$ (see Definition 
\ref{def:SPD}), 
 let $({\mathcal X})^{\hat{}}_{\Gamma}$ be the formal completion of $\mathcal{X}$ along $\Gamma$. 
 
Then, there is a (non-canonical) $\C^*$-equivariant 
isomorphism 
$$({\mathcal X})^{\hat{}}_{\Gamma} \cong (W \times 
\A^1)^{\hat{}}_{0_W \times \A^1}$$ as formal Poisson 
schemes over $\A^1$. Here the Poisson bracket of the right hand side is induced from the Poisson bracket $\{\:, \:\}_W$ on $W$ and the trivial Poisson bracket on $\A^1$.
\end{Thm}
\vspace{0.2cm}

We actually prove a stronger Theorem \ref{formallocaltriv2}. 
By the isomorphism $f\vert_{\Gamma}\colon \Gamma \to \A^1$, we identify a (closed) point $s \in \A^1$ with a 
point of  $\Gamma$, which we denote by $\Gamma_s$. We put  $\mathcal{X}_s := f^{-1}(s)$.   
Note that $\Gamma_s \in \mathcal{X}_s$. Then we can consider the 
formal completion $(\mathcal{X}_s)^{\hat{}}_{\Gamma_s}$ of $\mathcal{X}_s$ at $\Gamma_s$. \vspace{0.2cm}

\begin{Cor}\label{Cor:noscaleup}
 For any $s \in \A^1$, we have an 
 isomorphism $$(\mathcal{X}_s)^{\hat{}}_{\Gamma_s} \cong (W)^{\hat{}}_{0_W}$$ 
of formal Poisson schemes. 
\end{Cor}

 \begin{proof}[proof of Corllary \ref{Cor:noscaleup} (assuming Theorem \ref{formallocaltriv})]$(\mathcal{X}_s)^{\hat{}}_{\Gamma_s}$ is the fiber of the morphism $({\mathcal X})^{\hat{}}_{\Gamma} \to \Gamma$ over 
$\Gamma_s \in \Gamma$. $(W)^{\hat{}}_{0_W}$ is the fiber of the morphism $(W \times \A^1)^{\hat{}}_{0_W \times \A^1} 
\to \A^1$ over $s \in \A^1$. Therefore the statement follows from the isomorphism in Theorem \ref{formallocaltriv}. 
\end{proof}

Let us consider the commutative diagram 
\begin{equation} 
\begin{CD} 
(\mathcal{X})^{\hat{}}_{\Gamma} @>>> (\mathcal{X})^{\hat{}}_{\Gamma \cup W} @<<< (\mathcal{X})^{\hat{}}_W \\  
@VVV @VVV  @VVV \\  
\A^1 @>{id}>> \A^1 @<<< \mathrm{Spf}\: \mathbf{C}[[t]]
\end{CD} 
\end{equation}   

The actual stronger theorem than Theorem \ref{formallocaltriv}, which we shall prove in the main 
arguments of this section, is the following: 

\begin{Thm}\label{formallocaltriv2}
For any scale-up Poisson deformation 
$f\colon \mathcal{X}\to \A^1$ 
of a conical symplectic variety $W$ (see Definition 
\ref{def:SPD}), 
there is a (non-canonical) $\C^*$-equivariant isomorphism 
\begin{align}
(\mathcal{X})^{\hat{}}_{\Gamma \cup W} \cong 
(W \times \A^1)^{\hat{}}_{(0_W \times \A^1)\: \cup \: W}
\end{align}
of formal Poisson schemes. 
\end{Thm}
Clearly, if this were proved, we get an isomorphism 
$$(\mathcal{X})^{\hat{}}_{\Gamma} \cong 
(W \times \A^1)^{\hat{}}_{0_W \times \A^1},$$
i.e., Theorem \ref{formal.isom} holds. 
To prove Theorem \ref{formallocaltriv2}, 
we define $\mathcal{R}$ to be the coordinate ring 
$\Gamma(\mathcal{X},\mathcal{O}_{\mathcal{X}})$
of $\mathcal{X}$ and prepare 
definitions of its two completions 
$\hat{\mathcal R}$, $\hat{R}$ 
as well as their related rings. 

\begin{defn}[Two completions of $\mathcal{R}$]
\begin{enumerate}
\item 
We let $I \subset \mathcal{R}$ be the defining ideal of $\Gamma \cup  W \: \subset \: \mathcal{X}$ 
and we write the $I$-adic completion as 
\begin{align*}
\hat{\mathcal R} := \varprojlim\: \mathcal{R}/I^n.
\end{align*}

\item 
On the other hand, 
we define $$\hat{R} := \varprojlim \mathcal{R}/(t^{n+1}).$$ 
(A subring $R$ will be defined shortly in Definition \ref{R'R}  
\eqref{R'Rb}). 

\end{enumerate}
\end{defn}

For the latter, 
if we put $S_n := \mathrm{Spec}\: \C [t]/(t^{n+1})$ and define $\mathcal{X}_n := \mathcal{X} \times_{\A^1}S_n$, note that $\mathcal{R}/(t^{n+1})$ is 
the coordinate ring of $\mathcal{X}_n$. 
From the above definitions, there is a natural homomorphism 
$$\hat{\mathcal R} \to \hat{R}.$$ 
Moreover, since $\hat{R}$ is the completion of $\hat{\mathcal R}$ by the ideal $t\hat{\mathcal R}$, the map 
$\hat{\mathcal R} \to \hat{R}$ is an inclusion by the Krull's intersection theorem. \vspace{0.2cm}

By the above definition of $\hat{R}$, $\G_m$ acts on it. 
Using the fact, now we define a few more rings. 
\begin{defn}\label{R'R}
\begin{enumerate}[label=(\alph*)]
    \item 

Let $R' \subset \hat{R}$ be the $\C$-subalgebra generated by 
all $\mathbb{G}_m$-semi-invariant elements of $\hat{R}$. Put $I' := I\hat{R}  \cap R'$. 

\item \label{R'Rb}
We define $R$ to be the $I'$-adic  completion of $R'$. 
\end{enumerate}    
\end{defn}

The ring  $R'$ is characterized as following Lemma \ref{Lem1}. 
To prepare the statements, let 
$$\pi\colon \mathcal{X} \to \mathcal{Y} := \mathcal{X}//\G_m$$ be the GIT quotient map and let 
$\C[\mathcal{Y}]$ be the coordinate ring 
$\Gamma(\mathcal{Y},\mathcal{O}_{\mathcal{Y}})$  
of $\mathcal{Y}$. We take the completion $\C[[\mathcal{Y}]]$ of 
$\C[\mathcal{Y}]$ by the maximal ideal 
$\mathfrak{m}_{\mathcal{Y},\pi(0)}$ corresponding to $\pi (0)$. 
Then, the following holds. 

\begin{Lem}\label{Lem1}
The ring $R'$ coincides with the image of the natural map 
$$\C[[\mathcal{Y}]] \otimes_{\C[\mathcal{Y}]} \mathcal{R} \to \hat{R}.$$
\end{Lem}

\begin{proof} Since any element of 
$\C[[\mathcal{Y}]]$ is $\G_m$-invariant and 
$\mathcal{R}$ is generated by $\G_m$-semi-invariant elements as 
a $\C$-algebra, 
it is clear that the image is contained in $R'$. 
Thus, it suffices to prove that $R'$ is contained in the image. 
We first show that the image of the map $\C[[\mathcal{Y}]] \to \hat{R}$ coincides with $\hat{R}^{\G_m}$, 
the $\G_m$-invariant subring 
of $\hat{R}$. 
Let $x_1, \cdots,  x_n$ be homogeneous elements of $\mathcal{R}$ which gives minimal generators of the $\C$-algebra 
$\C[W] = \mathcal{R}/t\mathcal{R}$, 
the coordinate ring of $W$. 
By assumption, the weights $w_i := {\rm wt}(x_i)$ are all positive integers. 
An element $g \in \hat{R}$ is written (not uniquely) as $$g = \sum_{b \geq 0}f_b(x_1, \cdots , x_n)t^b$$ with polynomials $f_b$. 
When $g$ is a $\G_m$-invariant element of $\hat{R}$, we have ${\rm wt}(f_b(x_1, \cdots, x_n)) = bw $ for each $b$. 
Recall that $w$ is 
minus the weight of $t$. 
Then each monomial factor of $g$ has a form $(\mathrm{const})\cdot x_1^{a_1}\cdot\cdot\cdot x_n^{a_n}t^b$ with 
$a_1w_1 + \cdots  + a_nw_n = bw$. This monomial is an element of $\C[\mathcal{Y}]$; hence $g$ comes from $\C[[\mathcal{Y}]]$. 
        
Consider a $\G_m$ semi-invariant element $g$ of $\hat{R}$ 
with weight $m$. We claim that there are positive constants $C_1,\cdots, C_n, D$ depending only on  
$w_1, \cdots, w_n, w$ and $m$ such that $g$ can be written as 
$$g = \sum_{0 \leq a_1 < C_1,  \cdots ,\:\: 0 \leq a_n < C_n,\:\:  0 \leq b < D, \:\: \sum a_iw_i - bw = m}x_1^{a_1}\cdot\cdot\cdot x_n^{a_n}t^b\cdot h_{\vec{a},b}, 
$$
with each $h_{\vec{a},b}\in \hat{R}^{\G_m}$. 
If the claim holds, then we see that $g$ is in the image of the map $\C [[\mathcal{Y}]]\otimes_{\C[\mathcal{Y}]}\mathcal{R} 
\to \hat{R}$ because $x_1^{a_1}\cdot\cdot\cdot x_n^{a_n}t^b \in \mathcal{R}$ and the invariant elements 
$h_{\vec{a},b}$
come from $\C[[\mathcal{Y}]]$. 

To prove the claim, we may assume that $g$ can be written as a monomial $x_1^{a_1}\cdot\cdot\cdot x^{a_n}t^b$ with $a_1w_1 + \cdots  + a_nw_n -bw = m$. We define 
\begin{align}
C_i &:= \mathrm{max}\bigl\{w, w + \frac{m}{w_i}\bigr\}\:\: (i = 1, \cdots , n) \text{ and } \\
D &:= \mathrm{max}\bigl\{w_1 + \cdots  + w_n - \frac{m}{w}, \: w_1, \cdots , w_n\bigr\}.
\end{align}
If $a_i \geq C_i$ for some $i$, then we have
$$b = \frac{a_1w_1 + \cdots  + a_nw_n -m}{w} \geq w_i.$$ Since $a_i \geq w$, the monomial $x_1^{a_1}\cdot\cdot\cdot x^{a_n}t^b$ is divisible by the invariant monomial 
$x_i^wt^{w_i}$.  On the other hand, if $b \geq D$, then $a_{i_0} \geq w$ for some $i_0$. In fact, suppose to the contrary that 
$a_i < w$ for all $i$. Then $$bw + m = a_1w_1 + \cdots  + a_nw_n < w(w_1 + \cdots  + w_n),$$ which contradicts that 
$$b \geq w_1 + \cdots  + w_n - \frac{m}{w}.$$ 
Now we have $a_{i_0} \geq w$ and $b \geq w_{i_0}$. Then the monomial $x_1^{a_1}\cdot\cdot\cdot x^{a_n}t^b$ is divided by the 
invariant monomial $x_{i_0}^wt^{w_{i_0}}$.  This shows the claim.  \end{proof}


\vspace{0.2cm}


We can also characterize $R$ as follows. 

\begin{Lem}\label{Lem2}  
As subrings of $\hat{R}$, we have $R = \hat{\mathcal R}$. In other words, 
$$R = H^0(\mathcal{O}_{(\mathcal{X})^{\hat{}}_{\Gamma \cup W}}).$$    
\end{Lem}

\begin{proof}[proof of Lemma \ref{Lem2}] Let us consider the map 
$$\C[[\mathcal{Y}]] \otimes_{\C[\mathcal{Y}]} \mathcal{R} \to \hat{R}$$
discussed in Lemma \ref{Lem1}. 
Define $\C[[\mathcal{Y}]] \hat{\otimes}_{\C[\mathcal{Y}]} 
\mathcal{R}$ to be the completion of $\C[[\mathcal{Y}]] \otimes_{\C[\mathcal{Y}]} \mathcal{R}$ by 
$\C[[\mathcal{Y}]] \otimes_{\C[\mathcal{Y}]} I$. 
Note that $\mathfrak{m}_{\mathcal{Y},\pi(0)}\mathcal{R} \subset I$ by the property 3) of the definition of a scale-up Poisson deformation. Then we have   
$$\C[[\mathcal{Y}]] \otimes_{\C[\mathcal{Y}]} \mathcal{R}/I^n  
= \C[[\mathcal{Y}]]/\mathfrak{m}_{\mathcal{Y},\pi(0)}^n\C[[\mathcal{Y}]] \otimes_{\C[\mathcal{Y}]/\mathfrak{m}_{\mathcal{Y},\pi(0)}^n}\mathcal{R}/I^n 
= \mathcal{R}/I^n \:\: (n \geq 1)$$ because $\C[\mathcal{Y}]/\mathfrak{m}_{\mathcal{Y},\pi(0)}^n = \C[[\mathcal{Y}]]/\mathfrak{m}_{\mathcal{Y},\pi(0)}^n \C[[\mathcal{Y}]]$.
Therefore we have 
$$\C[[\mathcal{Y}]] \hat{\otimes}_{\C[\mathcal{Y}]} 
\mathcal{R} = \hat{\mathcal R}.$$ Moreover, the map $$\C[[\mathcal{Y}]] \otimes_{\C[\mathcal{Y}]} 
\mathcal{R} \to \C [[\mathcal{Y}]] \hat{\otimes}_{\C[\mathcal{Y}]} 
\mathcal{R}$$ factors through $R'$ and we have a sequence of rings 
$$\C[[\mathcal{Y}]] \otimes_{\C[\mathcal{Y}]} 
\mathcal{R} \to R' \subset \hat{\mathcal R} \subset \hat{R}.$$ 
$\hat{\mathcal R}$ is the completion of $R'$ by $IR'$. 
Hence, the proof of Lemma \ref{Lem2} 
is reduced to prove the following claim. 
\begin{Claim}\label{IR'I'}
In the above setup, we have $IR' = I'$.     
\end{Claim}
\begin{proof}[proof of Claim \ref{IR'I'}]
In order to prove this, we first claim that $R'/I' = \mathcal{R}/I$. Consider the commutative daigram with exact rows 
\begin{equation} 
\begin{CD} 
0 @>>> \C[[\mathcal{Y}]] \otimes_{\C[\mathcal{Y}]} I @>>> 
\C[[\mathcal{Y}]] \otimes_{\C[\mathcal{Y}]}\mathcal{R} @>>> \mathcal{R}/I @>>> 0 \\
@.  @VVV @VVV  @VVV \\  
0 @>>> IR' @>>> R' @>>> R'/IR' @>>> 0 
\end{CD} 
\end{equation}
Since the middle vertical map is surjective, the third vertical map $\mathcal{R}/I \to R'/IR'$ is surjective. 
Composing this map with the surjection $R'/IR' \to R'/I'$, we have a surjection $\mathcal{R}/I \to R'/I'$. 
Note that the coordinate ring $\mathcal{R}/I$ of the reduced scheme $\Gamma \cup W$ is given by the kernel 
of the map $$\C[t] \oplus \C[W] \to \C, \:\:\: (g, h) \to g(0) - h(0).$$ In particular, we have an inclusion  
$$\mathcal{R}/I \subset \C[t] \oplus \C[W].$$ 
Similarly we have an inclusion $\hat{R}/I\hat{R} \subset \C[[t]] \oplus \C[W].$ 
By the definition of $I'$, we have an injection $R'/I' \subset \hat{R}/I\hat{R}$. Hence we get an inclusion 
$$R'/I' \subset \C[[t]] \oplus \C[W].$$
There is a commutative diagram 
\begin{equation} 
\begin{CD} 
\mathcal{R}/I @>>> R'/I' \\ 
@VVV @VVV \\  
\C[t] \oplus \C[W] @>>> \C[[t]] \oplus \C[W] 
\end{CD} 
\end{equation}
By the diagram, the map $\mathcal{R}/I \to R'/I'$ is an injection. This means that $\mathcal{R}/I = R'/I'$. Then the existence of 
the surjection $\mathcal{R}/I \to R'/IR'$ implies that $IR' = I'$ i.e., the Claim \ref{IR'I'} holds. \end{proof}
Hence, Lemma \ref{Lem2} holds as it follows from 
Claim \ref{IR'I'} (by our discussions above). 
\end{proof}

Here let us briefly review the universal Poisson deformation of a conical symplectic variety $W$. Let $(\mathrm{Art})_{\C}$ be the category of local Artinian $\C$-algebras with residue field $\C$ and let $(\mathrm{Sets})$ be the category of sets. We define the Poisson 
deformation functor  $$\mathrm{PD}_W: (\mathrm{Art})_{\C} \to (\mathrm{Sets})$$ 
by letting $\mathrm{PD}_W(A)$  be equivalence classes of Poisson deformations of $W$ over 
$\mathrm{Spec}(A)$. 

Let $\tilde{W} \to W$ be a $\mathbb{Q}$-factorial terminalization of $W$ and put 
$d := \dim H^2(\tilde{W}, \mathbb{C})$. 
By \cite[\S 5]{Namb} (cf.also \cite[Theorem (2.1)]{Namg}), we have the universal Poisson deformation $f^{\rm univ}: \mathcal{X}^{\rm univ} \to \A^d$ of $W = (f^{\rm univ})^{-1}(0)$ with 
the following properties: \vspace{0.2cm}

\begin{enumerate}
\item There are good $\G_m$-actions respectively on $\mathcal{X}$ and $\A^d$ and $f^{\mathrm{univ}}$ is 
$\G_m$-equivariant. Here ``good'' means that $\G_m$ acts respectively on the cotangent space 
$\mathfrak{m}_{{\mathcal X}^{\rm univ}}/\mathfrak{m}^2_{{\mathcal X}^{\rm univ}}$ at the origin $0 \in \mathcal{X}^{\rm univ}$ and 
the cotangent space $\mathfrak{m}_{\A^d}/\mathfrak{m}^2_{\A^d}$ at the origin $0 \in \A^d$ with only {\em positive} weights. 

\item The Poisson bracket $\{\:, \:\}_{\mathcal{X}^{\rm univ}}$ has weight $-l$ with $l := {\rm wt}(\sigma_W)$. 

\item Let $\mathcal{X}' \to S$ be a Poisson deformation of $W$ with a local Artinian base $S$. Then there is a unique map 
$\varphi: S \to \A^d$ with $\varphi (0) = 0$ such that the induced Poisson deformation 
$\mathcal{X}^{\rm univ}\times_{\A^d}S \to S$ is equivalent to $\mathcal{X}' \to S$.  
\end{enumerate}

Moreover, by \cite[Corollary (2.4)]{Namg} (which is based on the argument of \cite{Rim}), $f^{\rm univ}$ is the universal $\G_m$-equivariant Poisson deformation of $W$.   
Namely, we have: \vspace{0.2cm}

(iii)' Let $\mathcal{X}' \to S$ be a $\G_m$-equivariant Poisson deformation of $W$ with a local Artinian base $S$ such that $\{\:, \:\}_{\mathcal{X}'}$ has weight $-l$.  
Then there is a unique $\G_m$-equivariant map $\varphi: S \to \A^d$ with $\varphi (0) = 0$ such that the induced Poisson deformation $\mathcal{X}^{\rm univ}\times_{\A^d}S \to S$ is equivalent to $\mathcal{X}' \to S$ as Poisson deformations 
of $W$ with $\G_m$-actions. Here the $\G_m$-action on the left hand side is induced from the 
$\G_m$-action an $\mathcal{X}^{\rm univ}$ and the $\G_m$-action on $S$. 

Using the above, we first observe the following Lemma, 
as the starting point of the proof of Theorem \ref{formallocaltriv2}. 

\begin{Lem}\label{Lem3}
There is a (non-canonical) 
$\G_m$-equivariant isomorphism of inductive systems of Poisson schemes $$\{\mathcal{X}_n\} \cong \{W \times S_n\}.$$ 
(Recall that $S_n = \mathrm{Spec}\: \C [t]/(t^{n+1})$). 
The $\G_m$-action on the left hand side is induced from the $\G_m$-action on $\mathcal{X}$ and the 
$\G_m$-action on $W \times S_n$ is given so that 
$\lambda\colon (x, t) \mapsto (\lambda\cdot x, \lambda^{-w}t)$ for $\lambda \in \G_m(\C)$. 
\end{Lem}

\begin{proof}[proof of Lemma \ref{Lem3}] 
As recalled above, 
let $f^{\rm univ}\colon \mathcal{X}^{\rm univ} \to \A^d$ be the universal Poisson deformation of $W$ of \cite{Namb}. 
This $f^{\rm univ}$ is $\G_m$-equivariant and the $\G_m$-action on $\A^d$ fixes the origin $0 \in \A^d$ 
and has only positive weights. Note that $f^{\rm univ}$ is the universal $\G_m$-equivariant Poisson deformation of $W$. 
We apply this to our formal Poisson deformation $\{\mathcal{X}_n\}$ of $W$. For each $n$, there is a map $\varphi_n\colon S_n \to 
\A^d$ so that $\varphi_n$ coincides with the composite $S_n \subset S_{n+1} \stackrel{\varphi_{n+1}}\longrightarrow \A^d$ and 
$\mathcal{X}_n \cong \mathcal{X}^{\rm univ}\times_{\A^d}S_n$.    
Since the $\G_m$-action on the base $S_n$ has a negative weight, we see that $\varphi_n$ is the constant map; in other words, 
$\varphi_n \colon S_n \to \A^d$ factorizes as $S_n \to \{0\} \in \A^d$. This implies Lemma \ref{Lem3}. \end{proof}

Now, we are ready to prove Theorem \ref{formallocaltriv2} 
by using our prepared materials and lemmas above. 

\begin{proof}[proof of Theorem \ref{formallocaltriv2}]
Let us compare $\X$ with the trivial 
Poisson deformation ${\rm pr}_2\colon W \times \A^1 \to \A^1$ of $W$ introduced at the beginning.  
For this Poisson deformation, we define 
$\mathcal{R}$, $I$, $\hat{R}$, $R'$, $R$ and $\hat{\mathcal R}$ in the same way as above.  
To distinguish them from those obtained from $\mathcal{X}$, we denote them by 
$\mathcal{R}_{W \times \A^1}$, $I_{W \times \A^1}$, $\hat{R}_{W \times \A^1}$, $R'_{W \times \A^1}$, $R_{W \times \A^1}$ and $\hat{\mathcal R}_{W \times \A^1}$. 
Then $\mathcal{R}$ does not necessarily coincide with $\mathcal{R}_{W \times \A^1}$.  
However, by Lemma \ref{Lem3}, we firstly see that $$\hat{R} = \hat{R}_{W \times \A^1}$$ $\G_m$-equivariantly so that 
$$R' = R'_{W \times \A^1}.$$ 
Moreover, $I\hat{R} = I_{W \times \A^1}\hat{R}_{W \times \A^1}$. In fact, there are surjections 
$p_{\Gamma}: \hat{R} \to \C[[t]]$ and $p_{\{0\} \times \A^1}: \hat{R}_{W \times \A^1} \to \C[[t]]$ corresponding to the sections $\Gamma$ and $\{0\} \times \A^1$. Note that $\Gamma - \{0\}$ (resp. $\{0\} \times (\A^1 - \{0\})$) is a unique $\G_m$-orbit in 
$\mathcal{X}$ (resp. $W \times \A^1$) whose closure contains $0 \in \mathcal{X}$ (resp. $(0,0) \in W \times \A^1$) and which is not contained in the central fiber $W$.  Therefore, 
by the isomorphism $\hat{R} \cong \hat{R}_{W \times \A^1}$, we can identify $\mathrm{Ker}(p_{\Gamma})$ with $\mathrm{Ker}(p_{\{0\} \times \A^1})$. We then have  
$$I\hat{R} = (t) \cap \mathrm{Ker}(p_{\Gamma}), \:\:\:  
I_{W \times \A^1}\hat{R}_{W \times \A^1} = (t) \cap \mathrm{Ker}(p_{\{0\} \times \A^1}).$$  
Hence, $I\hat{R} = I_{W \times \A^1}\hat{R}_{W \times \A^1}$. This means that $$I\hat{R} \cap R' = 
I_{W \times \A^1}\hat{R}_{W \times \A^1} \cap R'_{W \times \A^1},$$ and   
$R = R_{W \times \A^1}$.  Next, by Lemma \ref{Lem2}, we have  $$\hat{\mathcal R}  = \hat{\mathcal R}_{W \times \A^1}.$$ 
Therefore we have   
$$(\mathcal{X})^{\hat{}}_{\Gamma \cup W} \cong 
(W \times \A^1)^{\hat{}}_{(0_W \times \A^1)\: \cup \: W}.$$
This completes the proof of Theorem \ref{formallocaltriv2}.   
\end{proof}

\begin{Rem}
Our Theorem \ref{formallocaltriv} and Corollary \ref{Cor:noscaleup} morally show that the symplectic variety limits to the 
conical symplectic variety only in the direction of scale-down degeneration, which often appears as the algebro-geometric realization of the (metric)
tangent cone {\it at infinity} of the complete Ricci-flat K\"ahler metric
of Euclidean volume growth (see \cite{CH, SZ, Od24b, Od24c} for the detailed meaning). 
From this perspective, a differential geometric work of Bielawski-Foscolo \cite{BF} through twistor methods `a la Penrose and Hitchin
can be seen vaguely as a differential geometric analogue of our
claims. 
However, their metrics are 
not complete in general, which makes it difficult to connect with our 
work. 
We thank L.Foscolo for the discussion on this issue, and we 
hope to come back to discuss this issue in the future. 
\end{Rem}



\section{Comparison of $X$ and $W$}\label{sec:XW}

\subsection{Outline of the arguments in this section}
In this section, 
we show that the germ $x\in X$ of symplectic singularity and the cone $0\in W$ 
(see Theorems \ref{DSII.thm}, \ref{DS.maps}, \ref{AGlc}) 
have isomorphic 
analytic germs,  
in the setup of Theorem \ref{Mthm.intro} and \ref{Mthm2.intro}, although in a priori slightly non-canonical manner.

The outline of its proof goes as follows. 
We begin by recalling from Lemma \ref{XtoW.C} (ii) that the process $X\rightsquigarrow W$ is realized as a 
flat family $\mathcal{X}_{\xi'} \to \A^1$ with 
the central fiber $W$ and a general fiber $X$. 
Let us briefly recall the construction of $\mathcal{X}_{\xi'}$. Let $b: \bar{\mathcal{X}} \to X \times \A^1$ be the weighted blow up 
at $x \times \{0\} \in X \times \A^1$ as in 
Lemma \ref{XtoW.C}, (ii). Let $\overline{X \times \{0\}}$ (resp. $\Gamma$) be the proper transform of $X \times \{0\}$ (resp. $x \times \A^1$) by $b$ and let $\bar{W}$ 
be the exceptional divisor of $b$. Then 
$\mathcal{X}_{\xi'} = \bar{\mathcal{X}} - 
\overline{X \times \{0\}}$.  Note that $\Gamma 
\subset \mathcal{X}_{\xi'}$ because  $\Gamma \cap \overline{X \times \{0\}} = \emptyset$. 
There is a natural map $\mathcal{X}_{\xi'} \to \A^1$ and $\Gamma$ gives a section of the map. The central fiber of this map is $W := \bar{W} - \overline{X \times \{0\}}$ and $\Gamma$ intersects $W$ at $0 \in W$.   
Define $\mathcal{X}^{sm}_{\xi'}$ to be the 
open subset of $\mathcal{X}_{\xi'}$ where the 
map is smooth. Note that $X^{sm} \times (\A^1 - 0) \subset \mathcal{X}^{sm}_{\xi'}$. 
Consider the relative symplectic form $p_1^*\sigma_X$ on $X^{sm} \times (\A^1 - 0)$, where $p_1$ is the 1-st projection map. 
The most technical core of this section, which takes up whole 
\S \ref{subsec:ambmet} and \S \ref{subsec:extendtensor}, is that 
with a positive integer $D$ is suitably chosen, 
$t^{-2D}p_1^*\sigma_X$ extends to a relative symplectic form on 
$\mathcal{X}^{sm}_{\xi'}$. For that, we prepare differential geometric 
lemmas and use some careful Diophantine approximation arguments, 
relying on some classical works of Dirichlet and Kronecker. 

Then the map $\mathcal{X}_{\xi'} \to 
\A^1$ can be enhanced as a Poisson deformation of $W$ together with the section $\Gamma$. Moreover, $\G_m$ acts on $X \times \A^1$ by $(y,t) \to (y, \lambda^{-1}t), \:\: \lambda \in \G_m$; then this $\G_m$-action induces a $\G_m$-action on $\mathcal{X}_{\xi'}$. We can see that 
$\G_m$ acts on $W$ fixing $0 \in W$ with only 
positive weights. The Poisson deformation turns out to be a scale-up Poisson deformation of $W$. Then, by 
Corollary \ref{Cor:noscaleup} in the previous section, 
there is a Poisson (or symplectic) isomorphism $(X, x)^{\hat{}}  \cong  (W, 0)^{\hat{}}$ of the formal completions of  symplectic singularities (Corollary \ref{formal.isom}). 

\subsection{Invariance of $\Q$-Gorenstein index} 

This subsection is of supplementary nature and can be skipped if one is 
in a haste and interested only in the setup of Theorem \ref{Mthm.intro} 
and \ref{Mthm2.intro}. 
It is a natural question to ask whether the Donaldson-Sun procedure 
$X\rightsquigarrow W\rightsquigarrow C$ can increase the 
($\Q$-Gorenstein) indices. 
The following arguments is {\it not} logically used in 
(and eventually follows from) 
our proof of Theorem \ref{Mthm.intro} and  \ref{Mthm2.intro} in those 
setup. Nevertheless, 
we include it for convenience and interest in its own. 

\begin{Prop}\label{QG.ind}
For any klt singularity $x\in X$, the first step degeneration 
$0\in W$ 
of algebraic local conification (Theorem \ref{AGlc}) 
as well as the metric tangent cone $0\in C$ both 
have the same $\Q$-Gorenstein ind
ices as the original 
$x\in X$. 
\end{Prop}

 \begin{proof}
Suppose $x\in X$ has the $\Q$-Gorenstein index $m$. 
We first discuss the case of $W$. 
By Lemma \ref{XtoW.C} \eqref{W.wtedblowup}, 
for each $\xi'\in \sigma\cap N$, 
there is a $\Q$-Gorenstein degeneration 
$\X_{\xi'}$ of $X$ to $W$. Hence, 
the $\Q$-Gorenstein index of $0\in W$ 
is $dm$ for some positive integer $d$. 
Then, we define 
$\mathcal{Y}:={\it Spec}_{\mathcal{O}_{\X_{\xi'}}}
(\oplus_{j=0,\cdots,d-1}\mathcal{O}_{\X_{\xi'}}(-mjK_{\X_{\xi'}}))$ and 
denote the associated affine structure (finite) morphism, 
a variant of the 
index $1$ covering, by 
\begin{align}
c\colon 
\mathcal{Y}\twoheadrightarrow \X_{\xi'}, 
\end{align}
with respect to a non-vanishing section of $\mathcal{O}_{\X_{\xi'}}(dmK_{\X_{\xi'}})$, that exists 
after shrinking $x\in X$ and corresponding $\X$ sufficiently 
if necessary. 
Note that $c$ is automatically quasi-\'etale so that 
$K_{\mathcal{Y}}$ is again $\Q$-Cartier. 
We denote the central fiber as 
$0\in W_Y$ and the general fiber as 
$c_Y\colon Y\to X$. Take $c^{-1}(\Gamma)$ 
where $\Gamma\subset \mathcal{X}$ denotes the 
($\G_m$-invariant) vertex section 
which passes through $0\in W=\X_0$ and $x\in X=\X_t$ for $t\neq 0$. 
Note that $c^{-1}(0(=W\cap \Gamma))$ is one point 
while, for $t\neq 0$, $c^{-1}(x=\X_t\cap \Gamma)$ 
is $d$ points by e.g., \cite[2.48(i), Lemma 9.52]{Kol13}. 
Now we consider the finite base change 
of $f\colon \Y\to \A^1$ with respect to 
$c^{-1}(\Gamma)\to \A^1_t$ and denote it by 
$f'\colon \mathcal{Y}'\to C$ for some affine curve $C\simeq c^{-1}(\Gamma)$. 
Note that its central fiber of $f'$ has a 
larger normalized volume than 
$0\in W$ by the finite-degree formula 
(cf., e.g., \cite[Theorem 1.3]{XZ}), 
if $d>1$. 
On the other hand, general fiber of $f'$ 
is \'etale locally 
$x\in X$ so that it has the same local 
normalized volume as $0\in W$. 
So, if $d>1$, it contradicts with the {\it lower} semicontinuity 
of local normalized volume \cite[Theorem 1]{BlumLiu}. 
The case of $C=C_x(X)$ can be proved in the same way since 
we know $-K_{\X_C/\A^1}$ is $\Q$-Cartier by Lemma \ref{XtoW.C} 
\eqref{C.wtedblowup}. 
 \end{proof}

 \begin{Rem}
There are somewhat analogous differential geometric arguments by Spotti-Sun 
in \cite[\S 3]{SS}. 
\end{Rem}

 \begin{Cor}
If $x\in X$ is an arbitrary symplectic singularity, then 
$0\in W$ and $0\in C$ only have canonical Gorenstein singularities.  
 \end{Cor}

 \begin{proof}
This follows from Proposition \ref{QG.ind} 
since any symplectic singularity is Gorenstein. 
 \end{proof}

\subsection{Approximation by ambient cone metric}\label{subsec:ambmet}

In this subsection, we show that our local metric 
$g_X$ is comparable in a rather weak sense to certain ambient explicit 
metric as below. 
This is a differential geometric preparation for the 
extension of symplectic forms to a certain test configuration of $X$ 
(to be denoted by $\mathcal{X}_{\tilde{\xi}'}$ in the 
next subsection) in the next subsection. 
More specifically, Theorem \ref{thm:extendtoW} 
\eqref{est..tensor} (and later \eqref{extension}) 
relies on this subsection. 
We prepare the following setup and notation 
in this subsection. 

\begin{ntt}\label{ntt3}
\begin{enumerate}
\item Let $X$ be a log terminal affine variety with a closed point $x\in X$ and assume that it satisfies Conjecture 
\ref{conj:genDS} for a singular Ricci-flat K\"ahler metric $g_X$ 
(for instance, in the setup of Theorem \ref{Mthm.intro}). 

    \item 
We take a 
singular K\"ahler metric $\omega_\xi$ in the ambient space $\C^l$ 
defined as 
$$\omega_\xi:=\sqrt{-1}\sum_{1\le i\le n}
|z_i|^{\bigl(\frac{2}{w_i}-2\bigr)}
dz_i\wedge d\overline{z_i},$$ 
where $w_i$s are the weights for $C=C_x(X)$ i.e., 
$\xi=(w_1,\cdots,w_l)$. Note $w_i>0$ for any $i$. 
This $\omega_\xi$ is a smooth at least on 
$(\C^*)^l$ and, if $w_i$ are all at least $1$, it gives nothing but the standard model of 
the so-called conical singularity (or edge singularity) 
with cone angle 
$\frac{2\pi}{w_i}$ along $(z_i=0)$, 
in the sense of cf., e.g., \cite{Don11}. 
Note that $\Lambda_\tau^* \omega_\xi=|\tau|^{2}\omega_\xi$, 
where $\Lambda_\tau$ denotes that of Notation \ref{ntt2}. 

\item We denote the corresponding distance function $d_\xi$ to 
$\omega_{\xi}$ and the distance from the origin as 
$d_\xi(\vec{0},-)=r_\xi(-)$. 
Similarly, we denote the distance function $d_X$ to 
$\omega_{X}$ and the distance from the origin as 
$d_X(\vec{0},-)=r_X(-)$. 

\item 
We consider the rescaling action of the real multiplicative 
group $\R_{>0}\ni \tau$ on $\C^l$ 
as $(z_1,\cdots,z_l)\mapsto 
(\tau^{w_1}z_1,\cdots,\tau^{w_l}z_l)$, 
corresponding to the Reeb vector field $\xi$. 
We denote the quotient map 
$(\C^l\setminus 0)\twoheadrightarrow 
(\C^l\setminus 0)/\R_{>0}$ as ${\rm Arg}_{\xi}$. 
\end{enumerate}
\end{ntt}

We fix this notation throughout. 
We use this $\omega_\xi$ to give a rough approximation of 
the local K\"ahler-Einstein
metric by restriction of $\omega_\xi$. 
(There is also an alternative variant of $\omega_\xi$ which is smooth 
outside the origin, given in a more Sasaki geometric context 
\cite[\S 2.2, \S 2.3, Lemma 2.2]{HS}). 
These arguments closely follow 
the methods of \cite{SZ, Zha24}. The proof relies on the Donaldson-Sun 
theory, notably the algebraic realization of 
the local tangent cone $C_x(X)$ of $(x\in X,\omega_X)$ by \cite{DSII} 
as we review in Theorem \ref{DS.maps}. 

\begin{Lem}[{cf., \cite[Proposition 3.5]{SZ}, \cite[Lemma 5.3]{Zha24}}]
\label{SZlem}
Consider both $X$ and $W$ as subspaces of $\C^l$. 
There is an open subset 
$X^o\Subset X^{\rm sm}$ 
whose closure in $X$ 
contains $x$ 
such that for any $\epsilon>0$, 
there are positive constants $C_\epsilon$ and $D_\epsilon$ 
such that 
\begin{align}\label{met.comp}
C_\epsilon^{-1} 
r_\xi^{\epsilon}\omega_\xi|_{X^o}\le &\omega_X|_{X^o}\le 
C_\epsilon r_\xi^{-\epsilon}\omega_\xi|_{X^o}\\
\label{dist.comp}
D_\epsilon^{-1} 
r_\xi^{1+\frac{\epsilon}{2}}|_{X^o}\le 
&r_X|_{X^o}\le 
D_\epsilon r_\xi^{1-\frac{\epsilon}{2}}|_{X^o}. 
\end{align}
More precisely, 
if we take an arbitrary open subset 
$B_0\Subset((\C^*)^l/\R_{>0})\setminus 
{\rm Arg}_{\xi}({\rm Sing}(C)))$, there is $r_0>0$ such 
that 
we can take such $X^o$ which contains 
${\rm Arg}_{\xi}^{-1}(B_0)\cap  \{y\in X\mid 0<d_\xi(x,y)<r_0\}$. 
\end{Lem}

The proof follows the arguments of a variant 
\cite[Proposition 3.5]{SZ}, which was 
for the tangent cone {\it at infinity} of 
complete Ricci-flat K\"ahler metrics with Euclidean 
volume growth. We note that we do not a priori use 
$W=C$ in the following proof, 
but rather only use the metric comparison with the local metric tangent cone $C$ 
together with its algebraic realization due to \cite{DSII}. 

\begin{proof}
First we make some preparation. 
Within the ambient space $\C^l$, 
we denote the annulus 
$\{x\in \C^l\mid 1<d_\xi(0,x)<\frac{3}{2}\}$ 
by $A$. 
On the other hand, take a (large) open subset 
$$B'\Subset 
((\C^*)^l/\R_{>0})\setminus 
{\rm Arg}_{\xi}({\rm Sing}(C))),$$ 
where ${\rm Sing}(C)$ denotes the singular locus of $C$ i.e., $C\setminus C^{\rm sm}$, 
define $A'\Subset 
(A\setminus {\rm Sing}(C))$ as 
$$A':=A\cap {\rm Arg}_{\xi}^{-1}(B').$$ 
Further take subsets as $$A^o\subset U_A  
\Subset A'$$ where $U_A$ is open. 
Note that 
there is $j_0$ such that for any $j\ge j_0$ we have 
$\frac{2\sqrt{2}}{3}d_\xi(0,E_j(x))<d_\xi(0,x)<\frac{3}{2\sqrt{2}}d_\xi(0,E_j(x))$ as $E_j$ in 
Theorem \ref{DS.maps} converges to the identity. 
Therefore, if 
we take $A^o$ only 
slightly smaller than $A'$ for the {\it radial} direction, 
the following holds: 
\begin{align}\label{bigenough}
    \cup_j \Phi_j^{-1}(A^o)\supset 
{\rm Arg}_{\xi}^{-1}(B^o)\cap \{y\in X\mid 0<d_\xi(x,y)<r_0\},
\end{align}
for some small $r_0>0$ 
and $B^o\subset B'$ which again has some 
intermediate open subset $U_B$ of 
$((\C^*)^l/\R_{>0})$ 
as 
$B^o\subset U_B\Subset B'$. 
(Otherwise, for too small $A^o$, $\cup_j \Phi_j^{-1}(A^o)$ would 
contain infinite {\it  horizontal gaps} so that not containing the right hand side 
of \eqref{bigenough}). 
Note that for any $B^o$ with the above condition, we have corresponding large enough $A^o$ with \eqref{bigenough}. 

To prove \eqref{met.comp}, 
we start with the following obvious comparison 
\begin{align}\label{comp.A}
c_1^{-1}\omega_C\le \omega_\xi\le c_1 \omega_C    
\end{align}
for some $c_1>0$ 
on $A'\cap C$, which holds due to 
its relative compactness as we avoid 
singular locus. 
Now, we use Theorem \ref{DS.maps} which is proved 
in \cite{DSII}. We define $X_j:=\Phi_j(X)\subset \C^l$. 

Now we restrict this inequality 
to $X_j\cap A'$ and pull back by the diffeomorphism 
$\Psi_j$ as follows. More precisely, 
we take $(C^{\rm sm}\Supset) U_C\supset A'$ 
and then apply Theorem \ref{DS.maps} to obtain $\Psi_j$ 
and take $A^o$ which satisfies $\Psi_j(C\cap A')\supset X_j\cap A^o$ 
for any $j\ge j_0$ with fixed $j_0$. 
Due to 
    $\Psi_j\to {\rm Id}$ for $j\to \infty$ 
    (Theorem \ref{DS.maps} \eqref{DSiiia}) 
and 
$2^j \Psi_j^*((\Phi_j^{-1})^*\omega_X)\to \omega_C$ for $j\to \infty$ (Theorem \ref{DS.maps} \eqref{DSiiib}), 
since $A^o\Subset A'$, 
we have 
\begin{align}\label{comp.Xj}
    c_2^{-1}\cdot 2^j (\Phi_j^{-1})^*\omega_X
\le \omega_\xi\le c_2\cdot 2^j (\Phi_j^{-1})^*\omega_X
\end{align}
on $X_j\cap A^o$ 
for some $c_2>0$ and any $j\ge j_0$. 

For the same $j(\ge j_0)$, consider the pull back 
of the inequality \eqref{comp.Xj} by 
the embeddings $\Phi_j$ of Theorem \eqref{DS.maps}, 
we obtain 
\begin{align}\label{comp.X}
    c_2^{-1}\cdot 2^j \omega_X
\le \Phi_j^* \omega_\xi\le c_2\cdot 2^j \omega_X
\end{align} 
on $\Phi_j^{-1}(A^o)\subset X$ for any $j\ge j_0$. 

As in \cite[(3.8)]{SZ}, 
for any fixed neighborhood $U_G\Subset G_\xi$ of ${\rm Id}$ 
and distance $d_G$ induced by a Riemannian metric on $G_\xi$, 
there is a positive real constant $c_4$ such that 
\begin{align}\label{3.8}
|g^* \omega_\xi-\omega_\xi|_{\omega_\xi}
\le c_4 d_G(g,{\rm Id}), 
\end{align}
for any $g\in U_G$. 
Further, if $g_i (i=1,2,\cdots) \in G_\xi$ 
converges to ${\rm Id}$, then 
for any $\epsilon>0$, there is $c_\epsilon$ such that 
\begin{align}\label{3.9}
\prod_{i=2}^j(1+d_G(g_i,{\rm Id}))
\le c_{\epsilon}2^{\epsilon j}. 
\end{align}

Combining \eqref{3.8}, \eqref{3.9}, 
it follows that for any $\epsilon>0$, 
\begin{align}\label{Phi1Phij}
c_{\epsilon}^{-1}2^{(1-\epsilon) j}\Phi_0^* \omega_{\xi}
<\Phi_j^* \omega_{\xi}
<c_{\epsilon}2^{(1+\epsilon) j}\Phi_0^* \omega_{\xi}
\end{align}
for any $j$. 
On the other hand, note that there is a constant $c_4>0$ such that 
for any $j$ and $y\in \Phi_j^{-1}(A^o)$ we have 
\begin{align}\label{r.est}
d_{\xi}(x,y)<c_4 \sqrt{2}^{-j}
\end{align}
since $E_j\to {\rm Id}$ for $j\to \infty$. 

Hence, from \eqref{comp.X}, combining together with 
\eqref{Phi1Phij} and \eqref{r.est},  
we obtain the proof of the desired inequality 
\eqref{met.comp} 
on $\cup_j \Phi_j^{-1}(A^o)$ which contains 
the right hand side of \eqref{bigenough}. 
The proof of \eqref{dist.comp} follows the same arguments 
if we apply them to the comparison of distance functions (rather than 
the metric tensors). 
\end{proof}



\subsection{Asymptotic behaviour and extension of holomorphic differential forms}\label{subsec:extendtensor}

In this section, we keep the same setup as 
Notation \ref{ntt3} in the previous subsection,  
and 
consider the limiting behavior of holomorphic forms in the degenerate family of $X$ to $W, C$ as introduced in Lemma \ref{XtoW.C} and apply to certain extensions. 
Later we apply the following with $p=2$ and holomorphic symplectic form as $\sigma_X$. 
\begin{Thm}\label{thm:extendtoW}
As in the previous subsection and 
Notation \ref{ntt3}, 
let $X$ be a normal log terminal affine variety 
with a closed point $x\in X$ and apply Theorem \ref{AGlc}. Suppose $(0\neq)\sigma_X\in H^0(X,(\Omega_X^p)^{**})$ 
satisfies $p|n$ and $(\sigma_X^{\wedge \frac{n}{p}})\in H^0(\mathcal{O}_X(K_X))$ 
is non-vanishing. 
Here, ${}^{**}$ means the double dual to make it 
a reflexive coherent sheaf.  
If so, we call $\sigma_X$ is non-degenerate. 

Further, we also assume that 
Conjecture \ref{conj:genDS} holds for some $g_X$ with respect to 
which $\sigma_X$ is parallel, which already holds under the assumption of Theorem \ref{Mthm.intro} (our main interest is in $p=2$ i.e., algebraic 
symplectic form case). 

In this setup, the following holds. 
We write $X_\tau^o:=\Lambda_{\tau}^{-1}(X^o)\cap 
\{r_\xi\geq 1\}$ for $\tau\in (0,1)$. 
    \begin{enumerate}
    \item \label{est..tensor}
    For any $\epsilon>0$, there is $C_{\epsilon}$ 
    so that 
    \begin{align}\label{est.tensor}
    C_{\epsilon}^{-1}\tau^{\epsilon p}\le \tau^{-p}||\Lambda_\tau^*\sigma_X||_{(\omega_{\xi}|_{X^o_\tau})}
    \le C_{\epsilon}\tau^{-\epsilon p}
     \end{align}
    on $X^o_\tau$ for any $0<\tau\le 1$. 
    \item \label{extension}
    We follow Notation \ref{ntt2} e.g., the cone $\sigma\subset N\otimes \R$ and the definition of 
    $w\colon N\otimes \R\hookrightarrow \R^l$ (\eqref{def:w} in the proof of Lemma \ref{XtoW.C}). 
 For $\xi'\in \sigma\cap (N\otimes \Q)$, we describe $w(\xi')$ as $$\biggl(w'_1=\frac{\tilde{w}_1}{D},\cdots,w'_l=\frac{\tilde{w}'_l}{D}\biggr)\in \Q^l$$ and set $\tilde{\xi}':=D\xi'=(\tilde{w}'_1,\cdots,\tilde{w}'_l)\in \sigma\cap N$, with 
 $\tilde{w}'_i\in \Z \hspace{1.6mm} (i=1,\cdots,l), D\in \Z_{>0}$. 
    We consider $\mathcal{X}_{\tilde{\xi}'}$ as Lemma \ref{XtoW.C} \eqref{W.wtedblowup} with $\xi'$ there replaced by our $\tilde{\xi}'$. 
    
    There is a close enough choice of approximation $\xi'$ of $\xi$, which satisfies 
    that $W=(\mathcal{X}_{\tilde{\xi}'})|_0$ 
    (cf., Lemma \ref{XtoW.C} and \cite{Od24b, Od24c}) 
    and 
    the following extension property $(*)$: 
    \begin{quote}(*)
    $t^{-pD}p_1^*\sigma_X$ on $X\times (\A^1\setminus \{0\})$, 
    where $p_1$ stands for the first projection, 
    extends to $\X_{\tilde{\xi}'}^{\rm sm}$ as 
    a non-vanishing relative $p$-form i.e., 
    a section of $(\Omega_{\X_{\tilde{\xi}'/\A^1}})^{**}$, 
    which restricts to a non-degenerate $p$-form, 
    which we denote by  $\sigma_W(\xi')$ on 
    $W^{\rm sm}$. 
    \end{quote}
    \end{enumerate}
    \end{Thm}
    
    In our proof, $\xi'$ is fairly carefully chosen, 
    solving some Diophantine approximation problem. 
    Note that in the context of general local conification 
    (Theorem \ref{AGlc}), an elementary 
    Diophantine approximation is also used in \cite[\S 2.2]{LX}, 
    but for our purposes, we need more delicate version 
    connecting with estimates of asymptotics of the 
    norms of symplectic forms, 
    with respect to singular metrics. 
    (We actually give an explicit sufficient condition of this approximation, which we call {\it nice approximant}. See Definition \ref{def:nice}.) 
    
    Also note that replacement of  $D, \tilde{\xi}'$ by 
    their positive integer $d$ multiple corresponds 
    to the base change of ${\X}_{\tilde{\xi}'}\to \A^1$ 
    with respect to the degree $d$ ramified map 
    $\A^1\to \A^1$, $t\mapsto t^d$, so that the above condition 
    (*) remains equivalent. This justifies the notation 
    $\sigma_W(\xi')$.

\begin{proof}
\eqref{est..tensor} is a consequence of 
Lemma \ref{SZlem} \eqref{met.comp}. Indeed, 
\begin{align}
\tau^{-p}||\Lambda_\tau^*\sigma_X||_{(\omega_{\xi}|_{X_\tau})}
&=||\Lambda_\tau^* \sigma_X||_{\tau^2 (\omega_\xi|_{X_\tau})}\\ 
&=||\Lambda_\tau^* \sigma_X||_{\Lambda_\tau^* (\omega_\xi|_{X})}\\
&=\Lambda_\tau^* (||\sigma_X||_{(\omega_\xi|_X)}). \label{eqn.norm} 
\end{align}

On the other hand, Lemma \ref{SZlem} \eqref{met.comp} directly implies 
\begin{align}\label{tensor.estimate}
C_\epsilon^{-p} 
r_\xi^{p\epsilon}||\sigma_X||_{(\omega_X|_{X^o})}\le ||\sigma_X||_{(\omega_\xi|_{X^o})}\le 
C_\epsilon^p r_\xi^{-p\epsilon}||\sigma_X||_{(\omega_X|_{X^o})}, 
\end{align}
with the same constants $C_\epsilon$ with Lemma 
\ref{SZlem} \eqref{met.comp}. 

Since $\sigma_X$ is parallel with respect to $g_X$, 
there is a positive constant 
$c_5$ with 
\begin{align}\label{sigmanorm}
c_5=||\sigma_X||_{\omega_X}=\sqrt[\frac{n}{p}]{||\sigma_X^{\frac{n}{p}}||_{\omega_X}}. 
\end{align}
Combining above \eqref{eqn.norm}, \eqref{tensor.estimate} 
and \eqref{sigmanorm}, \eqref{est.tensor} follows.

In fact, by 
\eqref{tensor.estimate} 
and \eqref{sigmanorm} we have 
$$C_\epsilon^{-p}r_{\xi}^{p\epsilon}c_5 
\le ||\sigma_X||_{(\omega_\xi|_{X^o})}\le 
C_\epsilon^p r_\xi^{-p\epsilon}c_5
$$
Since  
${r_{\xi}} \geq \tau$ on 
$\Lambda_{\tau}(X_{\tau}^o)$, we have 
$$C_\epsilon^{-p}\tau^{p\epsilon}c_5 
\le \Lambda_{\tau}^*||\sigma_X||_{(\omega_\xi|_{X^o})}\le 
C_\epsilon^p \tau^{-p\epsilon}c_5. 
$$
After replacing $C_{\epsilon}$ by a suitable constant, we get \eqref{est.tensor}. 

\vspace{3mm}

Now we prove \eqref{extension} by using \eqref{est.tensor}. 
To clarify the idea, first we prove the case when $r(\xi)=1$, when there is no approximation issue. 
In this case, we can assume $\xi=\xi'$ is the generator of 
$\sigma\cap N\simeq \Z_{\ge 0}$ and $\tau$ can be regarded as the coordinate $t$ 
of the base $\A^1$ (only restrict to the positive real numbers). 
If $t^{-p}\Lambda_t^* \sigma_X$, 
defined on $\mathcal{X}\setminus W$ 
has a pole at whole $W\subset \X$, 
it clearly contradicts with 
\eqref{est.tensor}. 
For the open immersion $j\colon \mathcal{X}^{\rm sm}\hookrightarrow\mathcal{X}$, 
$j_* \Omega_{\X^{\rm sm}/\Delta}^p$ is a reflexive 
sheaf. 
Thus, since the central fiber of $\mathcal{X}$ is 
irreducible, 
$t^{-p}\Lambda_t^* \sigma_X$ 
extends to a global section of 
$\Omega^p_{\X^{\rm sm}/\Delta}$ which we denote as 
$\tilde{\sigma}$ and we define 
$\sigma_W:=\tilde{\sigma}|_W$. 
Now we are going to prove this algebraic limit 
$\sigma_W$ is non-vanishing on $W^{\rm sm}$. 
There are two proofs for this and we choose an easier way. 
One is to consider 
$\sigma_X^{\frac{n}{p}}$ and apply \eqref{est.tensor}. 
From the above arguments, $t^{-n}\sigma_X^{\frac{n}{p}}$ 
extends as a relative holomorphic $n$-form on $\X^{\rm sm}$ 
which can not vanish along $W$. Hence, $\tilde{\sigma}$ 
also can not vanish. 

For more general tensors, even without the assumption 
that 
 $\wedge^{\frac{n}{p}} \sigma_X$ give a holomorphic volume form, 
 we can at least prove that extended 
 $\sigma_W$ does not vanish on $W^{\rm sm}$. As it may 
 have 
 potential applications in the future, we keep the arguments for reference. 

We take the a 
priori vanishing locus of $\tilde{\sigma}$ as 
$W'\subset W^{\rm sm}$. 
Clearly $W'$ is $\R_{>0}$-invariant so 
we can take  $(0\neq) w\in W'\cap A$ where $A$ is the 
annulus 
$\{x\in \C^l\mid 1<d_\xi(0,x)<\frac{3}{2}\}$ 
defined in the proof of Lemma \ref{SZlem} \eqref{met.comp}. 
We take a local holomorphic coordinates 
$t, z_1,\cdots,z_n$ around $w\in \X$ with $z_i(w)=0$. 
Then we can and do locally describe 
$$\tilde{\sigma}=\sum_{I\subset \{1,\cdots,n\}, \#I=p}
\bigl(\sum_{j=1,\cdots,n}h_{I,j}(\vec{z},t)z_j+k(\vec{z},t)t\bigr) \wedge_{i\in I} dz_i,$$ 
with some local holomorphic functions 
$h_{I,i}(-,-)$ and $k(-,-)$, due to the vanishing at $w$. 
Now we 
take $A^o$ and $B^o$ in 
the proof of Lemma \ref{SZlem} 
large enough in the sense $w\in A^o$ and ${\rm 
Arg}_\xi(w)\in B^o$. 
Then, from Lemma \ref{SZlem}, we have 
\begin{align}\label{van.contr}
c_{\epsilon}^{-1}\tau^{\epsilon p}
&\le |h_{I,j}|^2 |z_j|^2+|k|^2 |t|^2\\ 
&\le c_6(\sum_{j=1}^n |z_j|^2+|t|^2) 
\end{align}
with some $c_6>0$. 
Note that 
a neighborhood of $w$ in $\X_W(\subset \C^l)\to \mathbb{C}$ is 
a holomorphic submersion. Take its local coordinate system, which
maps $w$ to the origin, and the
corresponding local holomorphic $0$-section $(\C^l \supset)\Delta\to \X_W$. 
Inside the section, we take a sequence $(\cup_\tau \Lambda_\tau^* X^0\ni)p_i\to w\in W\subset \X (i=1,2,\cdots)$ 
with $\tau(p_i)\to 0$ which automatically satisfies  $|z_j(p_i)|=O(\tau(p_i))$. 
Then we obtain the contradiction by applying the inequality \eqref{van.contr} to the sequence. 
We end the second arguments on the non-vanishing of $\sigma_W$ on $W^{\rm sm}$. 

\vspace{2mm}

Now we proceed to the case $r(\xi)>1$ by finding a suitable 
Diophantine approximation of $\xi$. 
Recall that 
for 
$\tau\in \R_{>0}\subset \C$, $X_\tau$ refers to 
${\rm diag}(\tau^{-w_1},\cdots,\tau^{-w_l})\cdot X_1\subset \C^l$. 
Similarly, we denote the fiber at $\tau$ of 
$\X_{\tilde{\xi}'}$ (resp., $\X_{\xi'}$) 
as $X_{\tau}^{\tilde{\xi}'}$ (resp., $X_\tau^{(\xi')}$). 
These are embedded in $\C^l$ and the isomorphism 
$\varphi_\tau\colon 
X_\tau^{(\xi')}\to X_\tau$ is realized by 
a matrix 
$A:={\rm diag}(\tau^{w'_1-w_1},\cdots,\tau^{w'_l-w_l})$. 
We also denote the isomorphism 
$\Lambda_{\tau}^{(\xi')}\colon 
X_\tau^{(\xi')}\to X_1^{(\xi')}$ given by 
${\rm diag}(\tau^{w'_1},\cdots,\tau^{w'_l})$. 
If we compare $\omega_\xi|_{X_{\tau}}$ and $\omega_{\xi}|_{X_\tau^{(\xi')}}$ via $\varphi_\tau$, 
we have 
\begin{align}\label{ggcomp}
\tau^{2\underset{1\le i\le l}\max\bigl\{\frac{|w_i-w'_i|}{w_i}\bigr\}}
\varphi_\tau^* (
\omega_\xi|_{X_\tau})
\le 
\omega_\xi|_{X_\tau^{(\xi')}} 
\le 
\tau^{-2\underset{1\le i\le l}\max\bigl\{\frac{|w_i-w'_i|}{w_i}\bigr\}}
\varphi_\tau^* (
\omega_\xi|_{X_\tau}), 
\end{align}
for $0<\tau\le 1$, 
by the definition of $\omega_\xi$. Thus, 
\begin{align}\label{perturb.tensor}
\tau^{p\underset{1\le i\le l}\max\bigl\{\frac{|w_i-w'_i|}{w_i}\bigr\}}
|(\Lambda_\tau^{(\xi')})^*\sigma_X|_{\varphi_\tau^*(\omega_\xi|_{X_\tau})}
&\le 
|(\Lambda_\tau^{(\xi')})^*\sigma_X|
_{(\omega_\xi|_{X_\tau^{(\xi')}})}\\
&\le  \label{perturb.tensor2}
\tau^{-p\underset{1\le i\le l}\max\bigl\{\frac{|w_i-w'_i|}{w_i}\bigr\}}
|(\Lambda_\tau^{(\xi')})^*\sigma_X|_{\varphi_\tau^*(\omega_\xi|_{X_\tau})}, 
\end{align}
for $0<\tau\le 1$. 

We set $d(\xi,\xi'):=\underset{1\le i\le l}\max\bigl\{\frac{|w_i-w'_i|}{w_i}\bigr\}$. 
Combining \eqref{est.tensor} with \eqref{perturb.tensor} and 
\eqref{perturb.tensor2}, 
for any $\epsilon>0$, there is a positive constant 
$C_{\epsilon}>0$ such that on 
$\varphi_\tau^{-1}(X_\tau^o)\subset X_\tau^{(\xi')}$, we have 

\begin{align}\label{tensor1}
C_{\epsilon}^{-1}\tau^{pD(1+\epsilon)}\cdot \tau^{2pD d(\xi,\xi')}
& \le 
|(\Lambda_{\tau^D}^{(\xi')})^* \sigma_X|_{X^o}|_{(\omega_\xi|_{\varphi_\tau^{-1}(X_\tau^o)})}\\ 
&\label{tensor2}
=|(\Lambda_{\tau}^{(\tilde{\xi'})})^* \sigma_X|_{X^o}|_{(\omega_\xi|_{\varphi_\tau^{-1}(X_\tau^o)})}\\ 
&\label{tensor3}
\le C_{\epsilon}\tau^{pD(1-\epsilon)}\cdot \tau^{-2pD d(\xi,\xi')}. 
\end{align}

Now, we prove the following, from which we show 
the extendability of rescaled $\sigma_X$ to $W$. 
\begin{Claim}\label{claim:dio}
For any positive real number 
$\epsilon'$, 
there are 
$D\in \Z_{>0}$, $\xi'\in N\otimes \Q$ 
and $\epsilon>0$ so that $Dw(\xi') \in \Z^l_{>0}$ and 
$$pD\epsilon + pDd(\xi, \xi') < \epsilon'.$$ 
\end{Claim}
\begin{proof}[proof of Claim \ref{claim:dio}]
We take a sufficiently large positive integer $N(\xi)$, 
depending just on $n,p,w_i$s as we clarify later, 
and 
take a Diophantine approximation $\xi'$ of $\xi$ 
of Dirichlet type 
as 
$w(\xi')=(w'_1=\frac{\tilde{w}'_1}{D},\cdots,w'_l=\frac{\tilde{w}'_l}{D})$ 
with $\tilde{w}'_i, D\in \Z_{>0}$ which satisfies 
\begin{align}\label{diophantine}
|Dw_i-\tilde{w}'_i|<\frac{1}{N(\xi)}. 
\end{align}

The existence of such $D$ and $\tilde{w}'_i$ s, which further satisfies $0<D<N(\xi)^l$ (although we do not need this effective
upper bound), is 
standard after L.~Dirichlet and follows from
the statements in \cite[Theorem 1A, (1.1) 
Chapter II]{Schmid} for example.  
Note that \eqref{diophantine} implies 
\begin{align}\label{diophantine2}
pDd(\xi, \xi')<\frac{p}{N(\xi)\cdot \underset{1\le i\le l}\min\{w_i\}}, 
\end{align}
which can be taken arbitrarily small if we take large enough $N(\xi)$. 
In particular, we can take $N(\xi)$ so that 
\begin{align}\label{N.exp}
\frac{p}{N(\xi)\cdot \underset{1\le i\le l}\min\{w_i\}} < \frac{1}{2}\epsilon'.
\end{align}
For this $N(\xi)$ we take $\epsilon$ so that 
$pD\epsilon<pN(\xi)^l\epsilon < \frac{1}{2}\epsilon'$. 
Then the desired inequality 
\begin{align}\label{diophantine3}
    pD\epsilon + pDd(\xi, \xi') < \epsilon'
\end{align}
holds. 
\end{proof}


We finish the rest of the proof of Theorem \ref{thm:extendtoW} relying on the above Claim \ref{claim:dio}, applied with 
$\epsilon'\le \frac{p}{n}$. 
Using the approximant $\tilde{\xi}'=(\tilde{w}'_1,\cdots,\tilde{w}'_l)=(\frac{\tilde{w}'_1}{D},\cdots,\frac{\tilde{w}'_l}{D})$ which exists 
by Claim \ref{claim:dio}, 
we can apply 
the same arguments as $r(\xi)=1$ case to prove the 
desired extendability assertion. Indeed, 
consider $\tau^{-pD}(\Lambda_{\tau}^{(\tilde{\xi'})})^* \sigma_X$ on 
$\X_{\tilde{\xi}'}\setminus W$ 
extends to whole 
$\X_{\tilde{\xi}'}$ as a family of  
$p$-form because of \eqref{diophantine3} and 
\eqref{tensor1}-\eqref{tensor2}-\eqref{tensor3} 
as far as $\epsilon'\le 1$. 
Further, as in our previous arguments for $r(\xi)=1$ case, 
if we apply the same arguments to its $\frac{n}{p}$-th 
(exterior) power 
$\tau^{-nD}((\Lambda_{\tau}^{(\tilde{\xi'})})^* \sigma_X^{\wedge^{\frac{n}{p}}})$, we conclude 
the restriction of $\tau^{-pD}(\Lambda_{\tau}^{(\tilde{\xi'})})^* \sigma_X$ to the central fiber is also non-degenerate in the sense of the statement of our Theorem \ref{thm:extendtoW}, 
because $\frac{n}{p}\epsilon'\le 1$. 
(This is where we need $\epsilon'\le \frac{p}{n}$.) 
Hence \eqref{extension} follows. 
We complete the proof of Theorem \ref{thm:extendtoW}. 
\end{proof}

Motivated by the above discussion, 
now we define an explicit 
sufficient condition of extendability of 
the family of holomorphic forms 
$\tau^{-pD}(\Lambda_{\tau}^{(\tilde{\xi'})})^* \sigma_X$. 

\begin{defn}[Nice approximation of $\xi$]\label{def:nice}
Recall the map 
$w(-)$ as that of \eqref{def:w} defined in the proof of Lemma \ref{XtoW.C}. 
We fix a large enough positive integer 
$$N(\xi)\ge  \lceil \frac{4n}{\underset{1\le i\le l}\min\{w_i\}} \rceil$$ 
which further satisfies the following condition: 
\begin{quote}(**)
any vector $w(\xi')=(w'_1,\cdots,w'_l)\in w(N\otimes \R)\subset \R^l$ which satisfies 
$|w_i-w'_i|<\frac{1}{D N(\xi)}$ is contained in $w(\sigma)$, 
where $\sigma$ is that of Lemma \ref{XtoW.C}. 
\end{quote}

Then, we call $\xi'\in N\otimes \Q$ with 
$$w(\xi')=(w'_1,\cdots,w'_l)\in N\otimes \Q\subset \Q^l,$$ together with 
expressions $w'_i=\frac{\tilde{w}'_i}{D} (i=1,\cdots,l)$ ($\tilde{w}'_i\in \Z, 
D\in \Z_{>0}$) is a {\it nice approximant} of $\xi$ (or $w(\xi)$), if 
it satisfies the estimates  \eqref{diophantine} 
for the above fixed $N(\xi)$ i.e., 
\begin{align}\label{eqn:nice}
|Dw_i-\tilde{w}'_i|<
\frac{1}{N(\xi)}
\end{align}
for any $1\le i\le l$. 
\end{defn}

We have to be careful that for any nice approximant 
$w(\xi')=(\frac{\tilde{w}'_1}{D},\cdots,\frac{\tilde{w}'_l}{D})$, its different expression 
$(\frac{a\tilde{w}'_1}{aD},\cdots,\frac{a\tilde{w}'_l}{aD})$
for $a\gg 1$ is {\it not} a nice approximant any more, as 
similar things often happen in the theory of Diophantine 
approximation. 
Henceforth, throughout this section, we identify $\xi'$ and 
$w(\xi')$ through the inclusion map $w$. 

\vspace{5mm}

Before going into further analysis, note that 
our arguments so far at least imply the original conjecture of 
Kaledin (\cite{Kaledin, Kaledin 2}) as follows. 

\begin{Cor}[Kaledin's conjecture]\label{formal.isom}
In the setup of Theorem \ref{Mthm.intro} or \ref{Mthm2.intro}, 
there is an isomorphism of Poisson formal schemes: 
$(X, x)^{\hat{}}  \cong  (W, 0)^{\hat{}}$. 
Here, $W\ni 0$ is the first step object of 
Donaldson-Sun theory (see \S \ref{sec:2}) and it is a 
conical symplectic variety. In particular, 
Kaledin's conjecture (\cite[cf., Remark 4.2, \S 4]{Kaledin}, \cite[Conjecture 1.8]{Kaledin 2}), at the formal isomorphism level, 
holds in the setup. 
\end{Cor}

\begin{proof}
Recall that by Lemma \ref{XtoW.C}, there is a scale up test configuration 
$\mathcal{X}_{\tilde{\xi}'}\to \A^1$
of 
$X$ whose special fiber is $W$, for 
any close enough approximation $\xi'=\frac{\tilde{\xi}'}{D}$ 
of $\xi$. Further, by applying 
Theorem \ref{thm:extendtoW} \eqref{extension} with $p=2$, 
among those  approximations $\xi'$, 
if we choose $\xi'=\frac{\xi'}
{D}$ more carefully as a nice approximant in the sense of 
Definition \ref{def:nice}, the obtained $\mathcal{X}_{\tilde{\xi}'}\to \A^1$ is even 
enhanced to a Poisson deformation. Then, we can apply 
Theorem \ref{formallocaltriv} and Corollary \ref{Cor:noscaleup} 
in the previous section to show the desired 
claim. We complete the proof of Corollary \ref{formal.isom}. 
\end{proof}

The fact $p=2$ holds 
is not really used in the above proof, except for the use of \S 
\ref{sec:PD}. 
Henceforth, our discussions will be devoted to improve the 
result (for general $p$), in a different direction. 
Firstly, in Theorem \ref{thm:extendtoW2} below, we refine Theorem \ref{thm:extendtoW} by going through 
more delicate Diophantine approximations, which we apply to 
analyze the asymptotic behaviour of degeneration of $\sigma_X$. 
For that, we use the following usual convention. 
\begin{ntt}\label{ntt4}
For $\vec{v}(=(x_1,\cdots,x_s))\in \R^{s}$, 
we denote $\{\vec{v}\}=(\{x_1\},\cdots,\{x_s\})
\in [0,1)^s$ 
where $\{x_i\}$ denotes the fractional part of $x_i$ i.e., $x_i-[x_i]$ with the Gauss symbol $[-]$ 
(the rounddown).     
\end{ntt}
    \begin{Thm}\label{thm:extendtoW2}
Under the setup of Theorem \ref{thm:extendtoW}, 
we further have: 
\begin{enumerate}[label=(\alph*)]
    \item \label{extension.cone}
    There is a rational polyhedral simplicial 
    cone $\sigma'(\subset w(\sigma)\subset \R^l)$, which contains $\xi$, and whose extremal rays are  
    all of the forms $\R_{\ge 0}\cdot w(\xi'(k))$, where 
    $$w(\xi'(k))=\biggl(\frac{\tilde{w}'_1(k)}{D(k)},\cdots,\frac{\tilde{w}'_l(k)}{D(k)}\biggr)$$ 
    are nice  approximants of $w(\xi)$ (Definition \ref{def:nice}). Note that we do not require 
    $\dim \sigma'$ is $l$. 
    \item \label{canonicity.sigmaW}
    Take any element $(0\neq )\xi' \in \sigma'\cap \Q^l$. For any sufficiently divisible 
    positive integer $D\in \Z_{>0}$ (we set $\tilde{\xi}'=D\xi'\in \Z^l$),  
    the extendability condition of the holomorphic symplectic form 
    (*) of Theorem \ref{thm:extendtoW}   \eqref{extension} 
    to $\sigma_W(\xi')$ on $W$ 
    is satisfied (though we do not claim the niceness in the sense of  
    Definition \ref{def:nice}). 
    
    Furthermore, that $\sigma_W(\xi')$ (see Theorem \ref{thm:extendtoW} \eqref{extension}) actually 
    does not depend on $\xi'$ and it is $T=N\otimes \G_m$-homogeneous i.e., 
    it is an eigensection with respect to the $T$-action. 
    \end{enumerate}
\end{Thm}

\begin{proof}

We give the proof of \eqref{extension.cone} first. 
During this proof, we use the above notation \ref{ntt4}. 
Recall from Notation \ref{ntt1} that $\Q$-rank of 
$\sum_{i=1}^l \Q w_i$ is $r=r(\xi)$. 
Thus, the rational rank of a bigger $\Q$-linear subspace 
$\sum_{i=1}^l \Q w_i+\Q\cdot 1$ is either $r+1$ or $r$, 
which we denote as $s(\xi)+1$ with $s(\xi)$ either $r$ or $r-1$. 
Renumbering the subindices of $w_i$s if necessary, we can and do assume that 
$1, w_1,\cdots,w_{s(\xi)}$ are linearly independent over $\Q$. 
Henceforth, 
if there is no fear of confusion, we sometimes also abbreviate 
$s(\xi)$ as $s$. 
If $s=s(\xi)<l$, 
we take integers $m\in \Z_{>0}, a_{i,j}\in \Z\hspace{1.5mm} 
(0\le i\le s, 1\le j\le l-s)$ such that 
\begin{align}\label{w.rel}
w_{s+j}=\frac{1}{m}\sum_{i=1}^s a_{i,j}w_i+\frac{a_{0,j}}{m}. 
\end{align}
They are unique up to multiple, due to the definition of $s$. 
Motivated by this, we introduce 
an $s$-dimensional {\it affine} $\Q$-linear subspace of $\Q^l$ i.e., 
a translation of an $s$-dimensional $\Q$-linear subspace 
\begin{align}\label{Vxi}
V_\xi:=\{(x_1,\cdots,x_l)\in \Q^l\mid 
x_{s+j}=\frac{1}{m}\sum_{i=1}^s a_{i,j}x_i+\frac{a_{0,j}}{m} \text{ for } 1\le j\le l-s\},
\end{align}
in which we seek for nice approximations of $\xi$. 
In other words, 
$V_\xi$ is nothing but the 
minimal {\it affine} $\Q$-linear subspace of $\Q^l$ which contains 
$w(\xi)=(w_1,\cdots,w_l)$. 
In particular, 
$V_\xi\subset w(N\otimes \Q)=\Q\cdot V_\xi$ holds. 

Since $1, w_1, \cdots , w_s$ are lineraly independent over $\mathbb{Q}$, we can and do apply the well-known density 
\footnote{In some literature, the integer $d$ is often allowed to be both negative or positive, but 
it is straightforward to reduce our version with $d>0$ 
to that case}
of $\{\{d(w_1,\cdots,w_s)\}\mid d\in \Z_{>0}\}\subset [0,1)^s$ (cf., Notation \ref{ntt4}) 
which dates back to Kronecker or Weyl. See \cite{Kronecker}, 
\cite{Weyl}, cf., also \cite[Appendix A]{Hum}. 
We take a sufficiently large positive integer $N'(\xi)>1$ 
such that 
\begin{align}\label{def:N'}
\frac{ms \underset{1\le i\le s,1\le j\le l-s}\max|a_{i,j}|}
{ N'(\xi)}<\frac{1}{N(\xi)}. 
\end{align}
If we consider a small subset 
\begin{align}\label{def:Vk}
\{(y_1,\cdots,y_s)\in [0,1]^s\mid \min\{|y_i|,|1-y_i|\}\le \frac{1}{N'(\xi)} 
\text{ for }1\le \forall i\le s\}
\end{align}
of $[0,1]^s$, it consists of $2^s$ small $s$-dimensional 
cubes (the connected components) 
which we denote by $V_1,\cdots, V_{2^s}$. 
From the density, for any $1\le k\le 2^s$, there is some 
$C(k)\in \Z_{>0}$ such that 
$\{C(k)(w_1,\cdots,w_s)\}\in V_k$. We denote the closest integral vector to 
$C(k)(w_1,\cdots,w_s)$ as $(\tilde{v}'_1(k),\cdots,\tilde{v}'_s(k))$ 
with $\tilde{v}'_i(k)\in \Z$. 
We set $(\tilde{v}'_1(k), \cdots,\tilde{v}'_s(k))$ as $\vec{v}(k)$. 
From our construction, we have 
\begin{align}\label{ckvi}
|C(k)w_i-\tilde{v}'_i(k)|<\frac{1}{N'(\xi)}    
\end{align}
for any $1\le i\le s$, $1\le k\le 2^s$. 
Further, the polyhedral cone 
$\sum_{1\le k\le 2^s} \R_{\ge 0}\vec{v}(k)$ contains $(w_1,\cdots,w_s)$ from our construction 
and the definition of $V_k$s (recall \eqref{def:Vk}). 

We set 
\begin{align*}
\tilde{w}'_i(k)&:=m\tilde{v}'_i(k) \text{ for }1\le i\le s, 1\le k\le 2^s, \\
D(k)&:=mC(k) \hspace{2mm}(\text{for same } k),
\end{align*}and then set the first $s$ components of our desired approximations as 
\begin{align*}
w'_i(k)&:=\frac{\tilde{w}'_i(k)}{D(k)}=\frac{\tilde{v}'_i(k)}{C(k)}\text{ for }1\le i\le s, 1\le k\le 2^s. 
\end{align*}
Recalling \eqref{w.rel} and the definition \eqref{Vxi} of $V_\xi$, if $s\le l$, 
we also define 
\begin{align}\label{wsjwi}
w'_{s+j}(k)&:=\frac{1}{m}\sum_{i=1}^s a_{i,j} w'_i(k)+\frac{a_{0,j}}{m}\hspace{2mm} 
(1\le j\le l-s) \text{ and equivalently,} \\ 
\tilde{w}'_{s+j}(k)&:=\frac{1}{m}\sum_{i=1}^s a_{i,j} \tilde{w}'_i(k)+C(k)a_{0,j}\hspace{2mm} 
(1\le j\le l-s). 
\end{align}
From this definition, 
$\bigl(w'_1(k)=\frac{\tilde{w}'_{1}(k)}{D(k)},\cdots,w'_l(k)=\frac{\tilde{w}'_{l}(k))}{D(k)}\bigr)$ 
lies in $V_{\xi}$. 
Then it follows that 
\begin{align}
|D(k)w_i-\tilde{w}'_i(k)|&=m|C(k)w_i-\tilde{v}'_i(k)|\\ 
                         &<\frac{m}{N'(\xi)} (\text{by }\eqref{ckvi}) \label{wivi.low}\\ 
                         &<\frac{ms \underset{1\le i\le s,1\le j\le l-s}\max|a_{i,j}|}{N'(\xi)}\\
                         &<\frac{1}{N(\xi)} (\text{by }\eqref{def:N'}),\label{former.nice}
\end{align}
for $1\le i\le s$. Using this, for $1\le j\le l-s$, it follows that 
\begin{align}
|D(k)w_{s+j}-\tilde{w}'_{s+j}(k)|&=m|c(k)w_{s+j}-\tilde{v}'_{s+j}(k)|\\ 
                                 &<m \max_{i,j}|a_{i,j}|\sum_{1\le i\le s}|c(k)w_{i}-\tilde{v}'_{i}(k)| (\text{by }\eqref{w.rel}, \eqref{wsjwi})\\ 
                                 &\le m s \max_{i,j}|a_{i,j}|\frac{1}{N'(\xi)}  (\text{by }\eqref{ckvi}) \\
                                 &\le \frac{1}{N(\xi)} (\text{by }\eqref{def:N'}). \label{latter.nice}
\end{align}
Thus, 
$$\biggl(w'_1(k)=\frac{\tilde{w}'_1(k)}{D(k)},\cdots,w'_l(k)=\frac{\tilde{w}'_l(k)}{D(k)}\biggr)$$ satisfies \eqref{eqn:nice}
of Definition \ref{def:nice} 
and further it can be written as $w(\xi'(k))$ 
for some $\xi'(k)\in N\otimes \Q$ because 
$\Q\cdot V_\xi=w(N_\Q)$ follows from the 
definitions (of $r(\xi), s(\xi), V_\xi$), 
and $\xi'(k)\in \sigma$ 
by the condition (**) of $N(\xi)$ in 
Definition \ref{def:nice}. 
By \eqref{former.nice} and \eqref{latter.nice}, 
such $\xi'(k)$ 
are nice approximants of $\xi$ for any $1\le k\le 2^s$. 

Thus, $\sum_{1\le k\le 2^s} \R_{\ge 0}\xi'(k)\subset \R^l$ satisfies 
the condition of \eqref{extension.cone} although it is not simplicial 
if $s>1$. 
To take a simplicial $\sigma'$ as a subcone of 
$\sum_{1\le k\le 2^s} \R_{\ge 0}\xi'(k)$, note that 
there is a subset $S$ of $\{1,\cdots,2^s\}$ of order $r$ 
which satisfies 
$\sum_{k\in S} \R_{\ge 0}\vec{v}(k)\ni \xi$. 
By re-ordering the subindices, we can and do assume that 
$S=\{1,\cdots,r\}$ so that 
$\sum_{1\le k\le r} \R_{\ge 0}\vec{v}(k)\ni (w_1,\cdots,w_s).$ 
Consequently, if we define a simplicial cone as 
\begin{align}\label{sigma'}
\sigma':=\underset{1\le k\le r}\sum \R_{\ge 0}\xi'(k)\subset \R^l, 
\end{align}
it contains $\xi$ and satisfies the desired properties. 
We complete the proof of \eqref{extension.cone}. 
\vspace{2mm}

 Finally we prove \eqref{canonicity.sigmaW} for 
the above $\sigma'$ 
using Theorem \ref{thm:extendtoW} \eqref{extension} and above  \eqref{extension.cone} (of Theorem \ref{thm:extendtoW2}) as follows. 
 
 We take an affine toric variety $U_{\sigma'}$ 
 corresponding to $\sigma'(\subset w(\sigma))$ 
 with respect to the integral structure $w(N')\subset w(N'\otimes \R)$ 
 in \eqref{extension.cone}. 
 Then, 
 as the base change of $\pi_\sigma\colon 
 \mathcal{X}_\sigma\to U_\sigma$ 
 of Lemma \ref{XtoW.C} 
 by the natural toric morphism $U_{\sigma'}\to U_\sigma$, 
 there is a faithfully flat affine family 
 $p_1\colon \mathcal{U}(\subset U_{\sigma'}\times\C^l)\to U_{\sigma'}$ of $X$ over $U_{\sigma'}$ 
 which is $X\times (N'\otimes \G_m)$ over $(N'\otimes \G_m)$ and 
 is trivial $W(\subset \A^l)$-fiber bundle  
 over the toric boundary $\partial U_{\sigma'}:=U_{\sigma'}\setminus (N'\otimes \G_m)$ 
 (cf., also \cite[\S 2, Example 2.9]{Od24b} and also related 
 \cite[\S 2.2]{Od24c}) i.e., $p_1|_{\partial U_{\sigma'}}$ is isomorphic to 
 the natural projection $W\times \partial  U_{\sigma'}\to \partial U_{\sigma'}.$

 We denote the dual lattice of $N$ 
 as $M$ 
 as before. 
 Then, we can and do 
 take an element $\vec{m}\in M\otimes \Q$ which satisfies 
 that $\langle \vec{m}, \xi'(k)\rangle =p$ 
 for $k=1,2,\cdots,r$. 
 Here, the existence of such $\vec{m}$ is thanks to the simpliciality of 
 $\sigma'$. 
 We take 
 a sufficiently divisible 
 $d$ such that 
 \begin{itemize}
 \item $\vec{m}\in \frac{1}{d}M,$ 
 \item $D(k)|d$ for $1\le \forall k\le r$ 
 (henceforth, we write $d\xi'(k)=d(k)\tilde{\xi'}(k)$ with 
 $d(k)\in \Z_{>0}$). 
 \end{itemize}
 Accordingly, we set $N':=d\underset{1\le k\le r}\sum\Z\xi'(k)$ and 
 take the affine toric variety $U_{\sigma'}^{(d)}$ for 
 $(\sigma',w(N'))$. Considering the corresponding finite 
 morphism $U_{\sigma'}^{(d)}\to U_{\sigma'}$, we work over 
 the base change of $p_1\colon \mathcal{U}\to U_{\sigma'}$ to $U_{\sigma'}^{(d)}$ which we denote as 
 $p_1^{(d)}\colon \mathcal{U}^{(d)}\to U_{\sigma'}^{(d)}$. 

We naturally set $T':=N'\otimes \G_m$ and 
 take its charactor $\tau'_{\vec{m}}$ whose exponent is $\vec{m}$ 
 (which we occasionally simply write $\tau'$). 
 From the first condition on $d$, $m$ lies inside the 
 dual of $N'$ and is positive on $\sigma'$ hence 
 gives a regular function on the base $U_{\sigma'}^{(d)}$. 
 Note that for each $1\le k\le s$, 
 and that $\X_{d\xi'(k)} \to \A^1_t$ 
 (resp., $\X_{\tilde{\xi}'(k)} \to \A^1_t$) 
 is a base change of $p_1^{(d)}$ by the toric morphism 
 $\A_t^1\to U_{\sigma'}^{(d)}$ for 
 $\Z_{\ge 0}\hspace{0.6mm} d\xi'(k)\to \sigma'\cap N'$ (resp., $\Z_{\ge 0}\hspace{0.6mm} \tilde{\xi}'(k)\to \sigma'\cap N'$). 

Now we consider the relative holomorphic 
 $p$-form 
 \begin{align}\label{original.form}
 (\tau'_{\vec{m}})^{-1}(p_1^{(d)})^*\sigma_X
 \end{align} on the $p_1^{(d)}$-preimage of the open strata on  the base 
 $T'(\subset U_{\sigma'}^{(d)})$. 

Write $N'=(\R \xi'(k)\cap N)\oplus N''$ 
with a complement sublattice $N''(\subset N')$ of rank $r-1$ 
and set $T''=N''\otimes \G_m (\subset T')$. There is a natural 
$\G_m\times T''$-equivariant morphism 
$\A_t^1\times T''\to U_{\sigma'}^{(d)}$ 
which corresponds to 
 $\Z_{\ge 0}\hspace{0.6mm} d\xi'(k)\times N''\to (\sigma'\cap N')\times N''$ (resp., $(\Z_{\ge 0}\hspace{0.6mm} \tilde{\xi}'(k))\times N'' \to (\sigma'\cap N')\times N''$), 
and the pullback of $\tau'_{\vec{m}}$ to $\X_{d\xi'(k)}\times T''$ (resp., $\X_{\tilde{\xi}'(k)}\times T''$) coincides with $t^{pd}\cdot \tau'_{\vec{m}}|_{T''}$ (resp., $t^{pd(k)}\cdot \tau'_{\vec{m}}|_{T''}$) 
by the construction. 
We denote the relatively ($p_1^{(d)}$-)smooth locus of $\mathcal{U}^{(d)}$ as 
 $\mathcal{U}^{{\rm rsm}, (d)}$ and 
 denote the union of the open strata of $U_{\sigma'}^{(d)}$ 
 and codimension $1$ toric strata as $U_{\sigma'}^{(d),o}$. 

Then, by the niceness of the approximations 
$\xi'(k)$ (Definition \ref{def:nice}) 
and (the proof of) Theorem \ref{thm:extendtoW} 
\eqref{extension}, 
the pullback of 
$(\tau'_{\vec{m}})^{-1}(p_1^{(d)})^*\sigma_X$ to 
$\X_{d\xi'(k)}\times T''$ (resp., $\X_{\tilde{\xi}'(k)}\times T''$) is globally defined, not only a meromorphic section. 
In other words, 
 $ (\tau'_{\vec{m}})^{-1}(p_1^{(d)})^*\sigma_X$
 extends to 
 $(p_1^{(d)})^{-1}U_{\sigma'}^{(d),o}\cap \mathcal{U}^{{\rm rsm}, (d)}$ 
 as a section of 
 $\Omega^2_{\mathcal{U}^{{\rm rsm}, (d)}/U_{\sigma'}^{(d)}}$ 
 which we denote as $\tilde{\sigma}_{\mathcal{U}}$. 
 On the other hand, note that codimension of 
 $U_{\sigma'}^{(d)}\setminus U_{\sigma'}^{(d),o}$ in 
 $U_{\sigma'}^{(d)}$ is $2$. Hence, by the normality 
 of $\mathcal{U}^{(d)}$, the relative $p$-form $\tilde{\sigma}_{\mathcal{U}}$ further 
 extends to a section of 
  $\Omega^2_{\mathcal{U}^{{\rm rsm}, (d)}/U_{\sigma'}^{(d)}}$ over whole 
  $\mathcal{U}^{{\rm rsm}, (d)}$ which we denote by 
  $\overline{\tilde{\sigma}_{\mathcal{U}}}$. Since the codimension of 
  $\mathcal{U}^{(d)}\setminus \mathcal{U}^{{\rm rsm}, (d)}$ in 
  $\mathcal{U}^{(d)}$ is at least $2$, similarly 
 the top exterior power 
 $(\tilde{\sigma}_{\mathcal{U}})^{\wedge (n/2)}$ extends to a global section of 
 $\mathcal{O}_{\mathcal{U}^{(d)}}(K_{\mathcal{U}^{(d)}/U_{\sigma'}^{(d)}})$, 
 which we simply denote by $\overline{\tilde{\sigma}_{\mathcal{U}}}^{\wedge \frac{n}{2}}$. 
 We denote the restriction of 
 $\overline{\tilde{\sigma}_{\mathcal{U}}}$  
 to $W=(p_1^{(d)})^{-1}(p)$ for the unique 
 $N'\otimes \G_m$-invariant point $p\in U_{\sigma'}^{(d)}$, 
 as $\sigma_W$. 
 From the construction, $\overline{\tilde{\sigma}_{\mathcal{U}}}^{\wedge \frac{n}{2}}$ does not vanish along any divisor in $\mathcal{U}^{(d)}$, 
 it is non-vanishing globally on $\mathcal{U}^{{\rm rsm}, (d)}$. 
 Hence, $\sigma_W$ is non-degenerate in the sense of the statements of Theorem \ref{thm:extendtoW}. 

 We again consider $T'$-action on $\mathcal{U}^{(d)}$. 
 Since $\overline{\tilde{\sigma}_{\mathcal{U}}}$ is the extension of 
 \eqref{original.form} and that for any $t'\in N'\otimes \C^* \colon 
 \mathcal{U}^{(d)}\to \mathcal{U}^{(d)}$, we have 
 $(t')^*\overline{\tilde{\sigma}_{\mathcal{U}}}=(\tau'_{\vec{m}}(t'))^{-1}\cdot \overline{\tilde{\sigma}_{\mathcal{U}}}$ i.e., it is $T$-homogeneous. 
 As $\sigma_W$ is the restriction of $\overline{\tilde{\sigma}_{\mathcal{U}}}$ 
 on $W$, which is a $T'$-invariant subscheme of $\mathcal{U}$, $\sigma_W$ is also $T'$-homogeneous in such a way that 
 \begin{align}\label{sigmaeigen}
     (t')^*\sigma_W=(\tau'_{\vec{m}}(t'))^{-1}\cdot \sigma_W 
 \end{align}
 for any $t'\in N'\otimes \C^*$. 
 Note that $\sigma_W(\xi')=\sigma_W$ for any $\xi'\in \sigma'\cap \Q^l$. 
Hence, we complete the proof of  \eqref{canonicity.sigmaW} 
of Theorem \ref{thm:extendtoW2}. 
 \end{proof}

Note that from the above arguments, we at least 
observe the following (for general case with 
non-degenerate $p(\ge 2)$-forms). 

\begin{Prop}\label{rands}
In our setup of Theorem \ref{thm:extendtoW} 
(see also Notation \ref{ntt1}), $1\in \sum_{i=1}^{l}\Q w_i$. 
In other words, we have $r(\xi)=s(\xi)+1$ in the notation of the 
proof of the previous Theorem \ref{thm:extendtoW2}. 
\end{Prop}

\begin{proof}
As the character $\tau'$ in the above 
proof of Theorem \ref{thm:extendtoW2}, 
$(\tau')^{-1} \overline{\tilde{\sigma}_{\mathcal{U}}}$ 
pulls back to $t^{-pd}p_1^* \sigma_X$ 
(resp., $t^{-pd(k)}p_1^* \sigma_X$) on $\X_{d\xi'(k)}$ (resp., $\X_{\tilde{\xi}'(k)}$), 
recall that $\vec{m}$ is taken so that 
$\langle \vec{m}, \xi'(k)\rangle =p$ 
 for $k=1,2,\cdots,s$. On the other hand, 
 the arguments in the previous Theorem 
 \ref{thm:extendtoW2} 
 implies that 
 $\xi$ lies in the smallest {\it affine}  
 $\Q$-linear subspace $w^{-1}(V_\xi)$ 
 which contains all 
 $\xi'(k) (k=1,\cdots,s)$. Thus, $\langle \frac{\vec{m}}{p},\xi\rangle=1$ holds. Since 
 $w(\xi)=(w_1,\cdots,w_l)$, the assertion follows. 
\end{proof}

After the above analysis of $r(\xi)$ in Proposition \ref{rands}, and 
examining various examples of symplectic singularities, 
we propose the following conjecture. 
\begin{conj}\label{r1conj}
For any symplectic singularity $x\in X$, 
the valuation $v_X$ in Theorem \ref{AGlc} is always 
(integer multiple of) divisorial valuation 
and the rational rank $r(\xi)$ of Theorems \ref{DS.maps}, \ref{AGlc} is 
always $1$. 
\end{conj}

In the last subsection \S \ref{subsec:last}, we 
prove the conjecture among others, under some conditions. 


\section{Proof of $W=C$}\label{sec:WC}

In this section, we prove that $W=C_x(X)$ 
as ($\Q$-)Fano cones under the assumptions of Theorems \ref{Mthm.intro} 
or \ref{Mthm2.intro}. 
Recall that the former is a priori only 
K-semistable Fano cone while the latter is 
K-polystable Fano cone in the sense of \cite{CS, CS2} (also cf., 
\cite[\S 2]{Od24a}). 
We add the following notation 
while keeping the previous ones. 

\begin{ntt}\label{ntt5}
\begin{enumerate}
    \item 
In the setup of Theorems \ref{DSII.thm},  \ref{DS.maps}, 
the tensors  on $\Phi_i(X^{\rm sm})$ 
defined as $2^i(\Phi_i^{-1})^* g|_{X^{\rm sm}}$ 
(resp., $2^i(\Phi_i^{-1})^* \sigma_X|_{X^{\rm sm}}$) 
is denoted by 
$g_i$ (resp., $\sigma_i$). 

\item 
Let us consider $W, C\subset \C^l$ in 
Lemma \ref{XtoW.C}. The torus $T$ acts on $\C^l$ preserving both $C$ and $W$.    
We take a multi-graded Hilbert scheme $B$ 
(\cite{HaiStu}, also cf. \cite[\S 3.3]{DSII}) 
parametrizing the subvarieties of $\C^l$ with $T$-actions, which includes $[W]$ and $[C]$. 
Let $G_\xi$ be the commutator of $T$ in ${\rm GL}(l)$. 
Then both $G_{\xi}$ and $T$ act on the 
universal family $\tilde{U} \to B$.  $T$ acts on the universal family fiberwisely, but $G_{\xi}$ acts nontrivially on the base $B$ (compatible with Notation \ref{ntt1}). 

\end{enumerate}
\end{ntt}

\begin{Thm}\label{ssps}
In the above setup, the metric $g_C$ on $C^{\rm sm}$ 
is a hyperK\"ahler metric and 
there is a Poisson deformation of 
$C$ to $W$, with the natural 
fiberwise $T$-action which extends the ones on $W$ and $C$. 
Furthermore, $W=C$ as affine conical symplectic varieties with respect to their good 
$T$-actions. 
 
\end{Thm}

\begin{proof}[proof of Theorem \ref{ssps}]
     By \cite{DSII, LWX}, there is a ($T$-equivariant)  affine test configuration 
     (see \cite[Definition 5.1]{CS}, \cite[\S 2, Definition 2.15 (i)]{Od24a}) whose general fiber is $W$ and the central fiber is $C$. 
If one can take such test configuration 
as a Poisson deformation, then 
we can apply the rigidity result \cite[\S 3]{Namf} to show 
that it is actually trivial, which implies the 
latter statement of the theorem. 
At the moment, we do not have such construction but only a 
weaker version i.e., isotrivial degeneration 
$\mathcal{W}$ 
of $W$ to $C$ 
of Poisson type a priori without $\C^*$-action (i.e., 
not necessarily a test configuration). Nevertheless, it suffices for our purpose for the same rigidity reason. 

\vspace{2mm}
 
Now we construct the limit holomorphic symplectic form 
$\sigma_C$ on $C$ and 
discuss 
how $\sigma_W$ and $\sigma_C$ are related. 
That is, using (singular) hyperK\"ahler metric on $X$ and the local diffeomorphisms $\Psi_i$s, 
we have two 
{\it differential geometric} 
construction (or description) of holomorphic 
symplectic form on $C^{\rm sm}$ as follows. 

We start with the 
holomorphic symplectic form $\sigma_X(=:\sigma_0)$ on $X^{\rm sm}$. 
Fix a smooth point $y(\neq 0)$ in $C_x(X)$. 
As remarked in Theorem \ref{DS.maps}, 
$C := C_x(X)$ is the polarized limit space (\cite[p330]{DSII}) of 
$(x\in X,J,c^2 g)$ for $c\to \infty$. For any (big)  compact subset $K\subset C^{\rm sm}$, its open neighborhood $U_C$ has a sequence of 
open $C^{\infty}$-embeddings 
$\Psi_i$ to $(X,J,c^2g_X)$ which approximates enough 
in the sense of {\it loc.cit}. 
Suppose $K\ni y$. 
Then consider 
$$\frac{\Psi_i^* \sigma_i|_y}{|\Psi_i^* 
\sigma_i(y)|_{g_C(y)}}$$ 
which have norm $1$ for any $i$, so that for some 
subsequence of $\{i\}$, 
it converges to a vector $\sigma_y \in 
\Omega^2_{X,y}$ of norm $1$. 
We denote the  denominator 
${|\Psi_i^* 
\sigma_i(y)|_{g(y)}}$ as $c_i$. 
We replace  $\{i\}$ by such a subsequence. 
Note that for any choice of $U_C$ and any 
other point $y'(\neq y)\in U_C$, 
there is a further subsequence whose corresponding 
$$\frac{\Psi_i^* \sigma_i|_{y'}}{|\Psi_i^* \sigma_i(y)|_{g_C(y')}}$$ has a limit, which we denote 
by $\sigma_{y'}$. 
On the other hand, 
recall that 
$\nabla_{g_i}\sigma_i=0$ (so that the holonomy of $\Psi_i^* g_i$ at $y$ 
sits inside some conjugate of the unitary compact symplectic group ${\rm Sp}(n))$)  due to a Bochner type theorem  (\cite[Theorem A]{Henri.etc}). 
Let us take a continuous  
path $\gamma\colon [0,1]\to C^{\rm sm}$ as  $\gamma(0)=y, \gamma(1)=y'$ and 
denote the parallel transport 
with respect to the metric $\Psi_i^* g_i$ 
(resp., $g_C$) 
along it from $y$ to $y'$ as 
$P_\gamma(\Psi_i^* g_i)$ (resp., $P_\gamma(g_C)$). 
Then, $\nabla_{g_i}\sigma_i=0$ implies 
\begin{align}\label{par.tr}
\frac{\Psi_i^* \sigma_i|_{y'}}{c_i}=
(P_\gamma(\Psi_i^* g_i))\biggl(\frac{\Psi_i^* \sigma_i|_y}{c_i}\biggr). 
\end{align}
By taking limit of the right hand side, 
Theorem \ref{DS.maps} \eqref{DSiii}  
implies that 
\begin{align}
\sigma_{y'}&=\lim_{i\to \infty} (P_\gamma(\Psi_i^* g_i))\biggl(\frac{\Psi_i^* \sigma_i|_y}{c_i}\biggr) \\
&=(P_\gamma(g_C))(\sigma_C|_y), 
\end{align}
where the left hand side is independent of $\gamma$ 
and the right hand side is independent of  
subsequence of $\{i\}$ once $y'$ is fixed. 
Thus, the 
$g_C$-parallel 
transform of 
$\sigma_y$ 
is well-defined 
on the smooth locus of $C$ 
which we denote by $\sigma_C$, 
and coincides with the limit of 
$\frac{\Psi_i^* \sigma_i}{c_i}$. 
From these, $\sigma_C|_{y'}=\sigma_{y'}$ for any 
$y'$ 
and $\sigma_C$ is holomorphic by 
Montel's theorem. In particular, 
the metric $g_C$ is hyperK\"ahler metric 
on the smooth locus $C^{\rm sm}$. 

By the Hartogs-Koecher extension 
principle, it gives an (a priori analytic) 
global section of 
$(\Omega_{C^{\it an}}^2)^{**}$, 
where $C^{\it an}$ refers to the complex analytification of 
$C=C_x(X)$ and ${}^{**}$ refers to 
the double dual. 

Now we consider the behaviour of the constants $c_i$s: 
\begin{align}
c_i&:=|\Psi_i^* \sigma_i|_{g_C(p)}\\
   &\label{4..} \sim |\Psi_i^* \sigma_i|_{\Psi_i^* (2^i g_X)(p)} \text{ still at }p\in C (\because \text{Theorem }\ref{DS.maps} \eqref{DSiii})\\ 
   &=|\sigma_i|_{2^i g_X(\Psi_i(p))} \text{ at }\Psi_i(p)\in X_i\\ 
   &=|2^i (\Phi_i^{-1})^* \sigma_X|_{2^i g_X(\Psi_i(p))} \text{ at }\Psi_i(p)\in X_i\\
   &\label{10..}=|\sigma_X|_{g_X(p_i)} \text{ at }p_i\in X\\ 
   &=\sqrt[\frac{n}{2}]{|\sigma_X^{\wedge \frac{n}{2}}|}_{({\rm det}(g_X))(p_i)}
   \\ 
   &\to \sqrt[\frac{n}{2}]{|\sigma_X^{\wedge \frac{n}{2}}|}_{({\rm det}(g_X))(x)}=:c \in \R_{>0}  \hspace{2mm} (i\to \infty). 
\end{align}
Here, $\sim$ in the item \eqref{4..}
means the ratio converges to $1$ as $i\to \infty$ and 
$p_i\in X$ in the item \eqref{10..} refers to $\Phi_i^{-1}(\Psi_i(p))$, 
which clearly converges to $x\in X$. 
By multiplying a constant to $\sigma_X$, 
we can and do assume 
$|\sigma_X^{\frac{n}{2}}|_{g_X}=1$ so that $c_i\to 1$. 
From the above discussion, it follows that 
\begin{Claim}[$\sigma_i$ vs $\sigma_C$]\label{XCcomp}
$\sigma_i$ on $X_i$ smoothly converges (in the $C^{\infty}$-sense) to $\sigma_C$ on $C$ 
at the smooth locus i.e., 
$$\sup_{U_C}\hspace{1.5mm}|\Psi_i^*\sigma_i-\sigma_C|_{g_C}\to 0 \quad (i\to \infty).$$
\end{Claim}
Since $\Phi_{i}=(E_i\cdot \Lambda)\circ \Phi_{i-1}$ for $i\ge 1$ 
with $E_i\to {\rm Id}\in G_\xi \hspace{2mm} (i\to \infty)$, 
the above claim implies that 
$\Lambda^*\sigma_C=2\sigma_C$. 
So, from the Fourier expansion 
(cf., \cite[pp.340-342]{DSII}), 
$\sigma_C$ is 
homogeneous with respect to the $\R_{>0}$-action 
with weight $2$. 
Recall also (cf., \cite{CS, CS2, 
LLX}): 
\begin{align}\label{weight 2}    
r_c^* \omega_C=c^2 \omega_C \:\: \text{and} 
\end{align}
\begin{align}\label{wedge}
\omega_C ^n=
n! \sqrt{-1}^{n^2}(\sigma_C^{n/2}\wedge  
\overline{\sigma_C}^{n/2}).
\end{align}
On the other hand, again by the same arguments as 
\cite[the proof of Lemma 2.17]{DSII}, 
$\sigma_C$ is an algebraic form. 
Moreover, 
this $\sigma_C$ is parallel with respect to $g_C$ from the 
construction via the parallel transform. 
From the above algebraicity of $\sigma_C$ and its $\R_{>0}$-homogeneity, 
the following holds. 
\begin{Claim}[$T$-homogeneity of $\sigma_C$]\label{sigmahomog}
$\sigma_C$ is a $T$-homogeneous algebraic $2$-form 
with respect to a character $\tau_{\vec{m}}$ of $T$ for $\vec{m}\in M$
in the sense that for any 
$t\in T(\C)$, we have $t^* \sigma_C=\tau_{\vec{m}}(t)\sigma_C$. 
\end{Claim}

Now we want to compare these $\sigma_i, \sigma_C$ with $\sigma_{W_i}$. 

For that, first we 
temporarily assume $r(\xi)=1$ for simplicity of the 
exposition, 
until near the end of the proof of Claim \ref{sigmaconv3}. 
At the end of the proof of Claim \ref{sigmaconv3}, 
we explain how to modify the arguments for $r(\xi)>1$ 
by using the nice approximant of $\xi$ (Definition \ref{def:nice}) 
constructed in the previous section. 

To continue the proof of Theorem \ref{ssps}, 
we need some preparation for local identifications between $X$ and $W$. 
Note that $\C^*$-equivariant version of Artin's (relative)  analytic 
approximation theorem holds (also cf., 
\cite[\S A.6]{AHR} for even algebraic equivariant 
approximation, though we do not need this strong result 
here). 
That is, in the notations of the original 
\cite[Theorem 1.5a (ii)]{Artin}, 
if $A, B, C$ have $\C^*$-actions 
with which $v,w,\overline{u}$ are $\C^*$-equivariant, $u$ can be 
also taken $\C^*$-equivariant. 

Now we explain its proof. Firstly, 
recall that Artin's original proof of the non-equivariant version 
is an almost direct consequence of 
\cite[Theorem 1.2]{Artin} by a short reduction argument to it, written 
in {\it op.cit} p.281 bottom five lines. 
Now, that \cite[Theorem 1.2]{Artin} is generalized to the  
equivariant version as \cite[Theorem A]{BM}, and 
the reduction arguments (\cite{Artin} p.281 bottom five lines) 
also generalizes equivariantly verbatim because 
$f_i, \overline{\alpha}_{\mu i}, \overline{\beta}_{ji}, g_j$s in 
its notation, can be taken as $\C^*$-semiinvariant functions 
respectively. 

Now we apply such $\C^*$-equivariant version of the 
Artin's analytic approximation 
to the $\G_m$-equivariant {\it formal 
isomorphism} obtained by Theorem \ref{formallocaltriv} 
between the completed stalk 
$\widehat{\mathcal{O}}_{\X_{D\xi},(x,0)}$ of 
$(x,0)\in \X_{D\xi}\to \A^1$ (Lemma \ref{XtoW.C}), 
for $D\in \Z_{>0}$ 
so that $D\xi\in N$, 
and that of 
$(0,0)\in W\times \A^1\to \A^1$, i.e., 
$\widehat{\mathcal{O}}_{D\times \A^1,(0,0)}$. 
This equivariant formal isomorphism is furthermore taken as 
$\widehat{\mathcal{O}}_{\A^1,0}$-algebra hence the obtained 
equivariant analytic isomorphism of germs of $(x,0)\in \X_{D\xi}^{\rm an}$ 
and $(0,0)\in W\times \A^1$ are $(\A^1)^{\rm an}$-morphism. 
Here, $(\A^1)^{\rm an}$ just means the complex analytification i.e., 
the complex line $\C$. 
In particular, for large enough $a\in \Z_{>0}$, 
there is an open neighborhood $U_W$ of $0\in W$ and 
biholomorphism $f\colon U_W\to X$ on to the image. Then, 
after shifting (replacing) $i$ by $i-a$ henceforth, 
we can define local biholomorphisms 
\begin{align}
\varphi_i:=\Lambda^i \circ f\circ \Lambda^{-i}\colon \Lambda^i(U_W)\to 
\Lambda^i(X) \text{ for each }i=0,1,\cdots,
\end{align}
onto the images which automatically satisfy that 

\begin{Property}\label{varphi.property}
\begin{enumerate}
\item $\varphi_i\to {\rm Id}$ smoothly for $i\to \infty$,   
\item \label{commute}
$\Lambda_{\tau_i}\circ \varphi_i=\varphi_0\circ \Lambda_{\tau_i}\colon \varphi_i^{-1}(V_i)\to X$ 
with $\tau_i=2^{-\frac{i}{2}}$, 
where $\Lambda_{\tau_i}$ in the left hand side  maps 
$X_{\tau_i}\to X_0$ while $\Lambda_{\tau_i}$ on the right hand side means the rescale down in $W$. 
\end{enumerate}
\end{Property}
Then, note that 
for any fixed open subset $U'_X\Subset X^{\rm sm}$, 
there is some $c>0$ such that 
\begin{align}\label{almostisometry}
c^{-1}
\varphi_0^* (g_\xi|_{U'_X})
<g_\xi|_{\varphi_0^{-1}(U'_X)}
<c
\varphi_0^* (g_\xi|_{\varphi_0^{-1}(U'_X)}), 
\end{align}
on $\varphi_0^{-1}(U'_X)$ 
as it follows straightforward from the 
diffeomorphismness of $\varphi_0$. 
We also define conjugates of $\{\varphi_i\}_i$ as 
$$\Psi''_i:=(E_i\circ\cdots E_1)\circ \varphi_i \circ(E_i\circ\cdots E_1)^{-1}\colon ((E_i\circ\cdots E_1)(\Lambda^i(U_W)))\to X_i$$ 
which are again diffeomorphisms 
onto the images. We consider 
$\Psi'_i :=(\Psi''_i)^{-1}\circ \Psi_i$. 
(The facts that these local diffeomorphisms are somewhat non-canonical and 
choices will not affect our proof resembles \cite[Lemma 3.9]{HS}, as 
pointed out by Junsheng Zhang.) 
Also, note that 
since $U_W$ contains the neighborhood of $0_W$ in $W^{\rm sm}$, 
for large enough $i$, the domain of $\Psi'_i$ 
still remains the same as that of $\Psi_i$ i.e., $U_C$. 

By Theorem \ref{thm:extendtoW2} 
\eqref{canonicity.sigmaW}, we know 
\begin{Claim}\label{sigmaconv2}
\begin{enumerate}
\item \label{simgaconv2.1}
$(\Lambda^{-j})^* (2^j\sigma_X)$ on $\Lambda^j(X)$ converges to 
$\sigma_W$ on $W^{\rm sm}$ as smooth convergence via $\Psi''_i$. Similarly, for each fixed positive integer $i$, 
$2^j (\Lambda^{-j})^* \sigma_X$ on $\Lambda^j (X_i)$ smoothly converges to 
$$((E_i\circ E_{i-1}\circ \cdots \circ E_1)^{-1})^* \sigma_W=:\sigma_{W_i}$$ 
as $j\to \infty$ at the smooth locus. 
$\sigma_{W_i}$ is again a $T$-homogeneous 
algebraic form with the same character $\tau_{\vec{m}}$ by the same theorem. 

\item \label{sigmaconv2.2}
$(\Lambda^{-j})^* (\Psi_i)^* (2^j\sigma_i)$ converges to 
$(\Psi'_i)^{*} \sigma_{W_i}$ for $j\to \infty$, as the pullback of \eqref{simgaconv2.1} 
by $\Psi'_i$. 

\end{enumerate}
\end{Claim}

We take an(analytically) open subset 
$W''\Subset {\rm Arg}_\xi(W^{\rm sm})$ 
and set $W':={\rm Arg}_\xi^{-1}(W'')$. 
In $W'$, $g_\xi$ is smooth and 
take a relatively compact open subset $U_C\Subset C^{\rm sm}$ so that 
we can assume that 
$\varphi_i(\overline{\Lambda_i^{-1}(\psi_i(U_C))})$ is contained in $W'$ for each $i$. 
For this $U_C$, we prove the following by using the estimates of 
subsection \ref{subsec:extendtensor}. 

\begin{Claim}\label{sigmaconv3}
For the above $U_C\Subset C^{\rm sm}$, we have 
$\sup_{\Psi_i(U_C)}|((\Psi''_i)^{-1})^*\sigma_{W_i}-
\sigma_{i}|_{g_i}\to 0$. 
\end{Claim}

\begin{proof}
We first slightly enlarge $U_C$ as 
$\overline{U_C}\subset V_C\subset C^{\rm sm}$, 
where $\overline{U_C}$ denotes the closure of 
$U_C$ in $C^{\rm sm}$. We can and do assume that 
$\Psi_i$s are defined over $V_C$ for all $i$s. 
Suppose that the geodesic distance of the vertex $0_C$ and any point of 
$\overline{V_C}$, the closure of $V_C$, 
with respect to $g_C$ on $C^{\rm sm}$, 
takes the value in the open interval $(d'_1,d'_2)$ for some $0<d'_1<d'_2$. 
Because of the convergence of $\Lambda_i(X)\to C$ and $\Psi_i^*g_i\to g_C$, 
it follows that for large enough $i$, 
the geodesic distance of any point of 
$\Lambda_i^{-1}(\Psi_i(V_C))$ to $x\in X$ with respect to $g_X$ 
takes its value in $\bigl(\frac{d_1}{\sqrt{2}^i}, \frac{d_2}{\sqrt{2}^i}\bigr)$ i.e., 
\begin{align}\label{diam.bd}
d_X(\Lambda_i^{-1}(\Psi_i(V_C)))\subset \bigl(\frac{d_1}{\sqrt{2}^i}, \frac{d_2}{\sqrt{2}^i}\bigr), 
\end{align}
for some $d_1,d_2$ with $0<d_1<d'_1<d'_2<d_2$. 
Now we want to apply Lemma \ref{SZlem}, so we follow its notation. 
We take a sufficiently large enough $B_0\Subset((\C^*)^l/\R_{>0})\setminus 
{\rm Arg}_{\xi}({\rm Sing}(C)))$, 
so that ${\rm Arg}_{\xi}^{-1}(B_0)\cap V_C\supset 
\overline{U_C}$. 
Then, 
$X^o$ of Lemma \ref{SZlem} is large enough so that 
\begin{align}\label{Uidef}
U_i:=X^o\cap \Lambda_i^{-1}(\Psi_i(V_C))   
\end{align}
contains $\Lambda_i^{-1}(\Psi_i(U_C))$ for $i\gg 0$. 
Now we apply Lemma \ref{SZlem} to $U_i$. 
Thus, it follows that 
\begin{align}
\label{met.comp2}    \frac{1}{2^{\epsilon i} C_{\epsilon}} g_X|_{U_i}
\le &g_\xi|_{U_i}
\le 2^{\epsilon i}C_{\epsilon} g_X|_{U_i} \text{ and} \\ 
\label{dist.comp2}   \frac{1}{\sqrt{2}^{\epsilon i} D_{\epsilon}} d_X|_{U_i}
\le &d_\xi|_{U_i}
\le \sqrt{2}^{\epsilon i}D_{\epsilon} d_X|_{U_i}, 
\end{align}
with certain positive real constants $C_{\epsilon}$ and $D_\epsilon$, 
for any $i\gg 0$. 
These properties hold for large enough $i$, 
but by shifting the index $i$ to $i-c$ for some constant 
$c$ 
if necessary, one can still assume it starts with $i=1$ 
(just for notational convenience). 

Consider $V_i:=\Lambda^{i}(U_i)\subset \Lambda^i(X)$. 
Then, by the homogeneity of $d_\xi$ with respect to $\Lambda$ (cf., 
Notation \ref{ntt3}), 
above \eqref{diam.bd} and \eqref{dist.comp2} imply that 
\begin{align}
r_\xi(V_i)&\subset (D_\epsilon^{-1} \frac{d_1}{\sqrt{2}^{i\epsilon}}, \sqrt{2}^{i\epsilon}d_2 D_\epsilon)\\ 
&\subset 
\biggl(\frac{1}{\sqrt{2}^{i\epsilon} D'_\epsilon}, \sqrt{2}^{i\epsilon} D'_\epsilon \biggr),\\
\label{diam.bd2} r_\xi(\varphi_i^{-1}(V_i))&\subset \biggl(\frac{1}{\sqrt{2}^{i\epsilon} D'_\epsilon}, \sqrt{2}^{i\epsilon} D'_\epsilon \biggr), 
\vspace{-5mm}
\end{align}
\noindent
for some $D'_\epsilon>0$. 
To prove Claim \ref{sigmaconv3}, it is enough 
to give upper bounds of the following functions 
for each $i$, 
which converge to $0$. 

\begin{align}
&\sup_{\Psi_i(U_C)}|(E_i \circ \cdots \circ E_1)^{*}(((\Psi''_i)^{-1})^*\sigma_{W_i}-
\sigma_{i})|_{2^i (\Phi_i^{-1})^* g_X}\\
\le&\sup_{V_i}|((\varphi_i)^{-1})^* \sigma_W-2^i (\Lambda^{-i})^{*}\sigma_X|_{2^i (\Lambda^{-i})^* g_X}\\
=&\label{difffun}\sup_{U_i}|\frac{((\varphi_i\circ \Lambda^i)^{-1})^* \sigma_W}{2^i}-\sigma_X|_{g_X}.
\end{align}
Here, we use \eqref{met.comp2} to observe that 
\begin{align}
&\frac{1}{2^{\epsilon i} C_\epsilon}
\sup_{U_i}|\frac{((\varphi_i\circ \Lambda^i)^{-1})^* \sigma_W}{2^i}-\sigma_X|_{g_\xi}\\
\le 
&\sup_{U_i}|\frac{((\varphi_i\circ \Lambda^i)^{-1})^* \sigma_W}{2^i}-\sigma_X|_{g_X}\\
\le 
&2^{\epsilon i} C_\epsilon 
\sup_{U_i}|\frac{((\varphi_i\circ \Lambda^i)^{-1})^* \sigma_W}{2^i}-\sigma_X|_{g_\xi}. 
\end{align}
By scaling up again, the homogeneity of $g_\xi$ implies that 
\begin{align}
&\sup_{U_i}|\frac{((\varphi_i\circ \Lambda^i)^{-1})^* \sigma_W}{2^i}-\sigma_X|_{g_\xi}\\
=2^i &\sup_{V_i}|\frac{(\varphi_i^{-1})^* \sigma_W}{2^i}-(\Lambda^{-i})^*\sigma_X|_{g_\xi} \\
=&\sup_{V_i}|\frac{(\Lambda^{-i})^* \sigma_X}{\tau_i^{2}}-((\varphi_i)^{-1})^* \sigma_W|_{g_\xi} \text{ for }\tau_i=2^{-\frac{i}{2}}\\
=&\sup_{\varphi_i^{-1}(V_i)}
|\frac{\varphi_i^*(\Lambda^{-i})^* \sigma_X}{\tau_i^{2}}-\sigma_W|_{\varphi_i^* g_\xi} \text{ for }\tau_i=2^{-\frac{i}{2}}\\
=\label{diff.fun}&\sup_{\varphi_i^{-1}(V_i)}
|\frac{\Lambda_{\tau_i}^*(\varphi_0^*\sigma_X)}{\tau_i^{2}}-\sigma_W|_{\varphi_i^* (g_\xi|_{V_i})} 
\text{ for }\tau_i=2^{-\frac{i}{2}}, 
\end{align}
\noindent 
where the last equality uses 
Property \ref{varphi.property} \eqref{commute}). 
Further, by 
\eqref{almostisometry}, 
we have 
\begin{align}
&(C')^{-1}\sup_{\Lambda^i(\varphi_0^*(U_0))}|\frac{\Lambda_{\tau_i}^*(\varphi_0^*\sigma_X)}{\tau_i^{2}}-\sigma_W|_{(\varphi_i^* g_\xi|_{\Lambda^i(\varphi_0^*(U_i))})}\\
=&\label{86}
(C')^{-1}\sup_{\varphi_i^{-1}(V_i)}
|\frac{\Lambda_{\tau_i}^*(\varphi_0^*\sigma_X)}{\tau_i^{2}}-\sigma_W|_{\varphi_i^* (g_\xi|_{V_i})} \\
\le 
&\label{diff.fun2}\sup_{\varphi_i^{-1}(V_i)}
|\frac{\Lambda_{\tau_i}^*(\varphi_0^*\sigma_X)}{\tau_i^{2}}-\sigma_W|_{(g_\xi|_{\varphi_i^* (V_i)})} \\
\label{diff.fun2.5}\le 
&C' \sup_{\varphi_i^{-1}(V_i)}
|\frac{\Lambda_{\tau_i}^*(\varphi_0^*\sigma_X)}{\tau_i^{2}}-\sigma_W|_{\varphi_i^* (g_\xi|_{V_i})}\\ 
=\label{diff.fun3}
&C' \sup_{\Lambda^i(\varphi_0^*(U_0))}
|\frac{\Lambda_{\tau_i}^*(\varphi_0^*\sigma_X)}{\tau_i^{2}}-\sigma_W|_{(\varphi_i^* g_\xi|_{\Lambda^i(\varphi_0^*(U_i))})}
\end{align}
for some $C'>0$, 
where \eqref{86} and \eqref{diff.fun3} use 
that $\varphi_i^{-1}(V_i)=\Lambda^i(\varphi_0^*(U_0))$ 
by 
Property \ref{varphi.property} \eqref{commute}, 
and \eqref{diff.fun2}, \eqref{diff.fun2.5} use the comparison of
$\varphi_i^* g_\xi$ and $g_\xi$ 
(
\eqref{almostisometry}). Thus, 
our proof of Claim \ref{sigmaconv3} is now reduced to 
estimate of $\sup_{\varphi_i^{-1}(V_i)}
|\frac{\Lambda_{\tau_i}^*(\varphi_0^*\sigma_X)}{\tau_i^{2}}-\sigma_W|_{(g_\xi|_{\varphi_i^* (V_i)})}$. 

By \eqref{diam.bd2} and the way we took $U_i$s 
(see before the Claim \ref{sigmaconv3} and \eqref{Uidef}),  
it follows that 
$$V'_i:=
\bigcup_{\frac{1}{\sqrt{2}^{i\epsilon}D'_\epsilon}\le \tau'\le \sqrt{2}^{i\epsilon} D'_\epsilon} 
\Lambda_{\tau'}\bigl(\varphi_i^{-1}(V_i)\cap 
r_\xi^{-1}(1/\tau' )\bigr)$$
is inside $W'\cap r_\xi^{-1}(\frac{1}{D'_\epsilon},D'_{\epsilon})$, 
hence in a relativley compact bounded region in $W^{\rm sm}$ 
where $g_\xi$ is smooth. 

Note that 
by the homogeneity of $\sigma_W$ and $g_\xi$, 
for general $\tau'\in \R_{>0}$, 
\begin{align}\label{translate}
|\Lambda_{\tau'}^*\biggl(\frac{\Lambda_{\tau_i}^*(\varphi_0^*\sigma_X)}{\tau_i^{2}}-\sigma_W\biggr)|(q)_{g_\xi}=
(\tau')^2\cdot |\biggl(\frac{\Lambda_{\tau'\tau_i}^*(\varphi_0^*\sigma_X)}{(\tau'\tau_i)^{2}}-\sigma_W\biggr)|(\Lambda_{\tau'}(q))_{g_\xi}. 
\end{align}
The above together with \eqref{diam.bd2} implies that 
\begin{align}
&\frac{1}{E_\epsilon}2^{-\epsilon i}
\sup_{V'_i}
|\frac{\Lambda_{\tau_i}^*(\varphi_0^*\sigma_X)}{\tau_i^{2}}-\sigma_W|_{(g_\xi|_{V'_i})}\\ 
\le 
&\sup_{\varphi_i^{-1}(V_i)}
|\frac{\Lambda_{\tau_i}^*(\varphi_0^*\sigma_X)}{\tau_i^{2}}-\sigma_W|_{(g_\xi|_{\varphi_i^* (V_i)})}\\
\le 
&E'_\epsilon 2^{\epsilon i}
\sup_{V'_i}
|\frac{\Lambda_{\tau_i}^*(\varphi_0^*\sigma_X)}{\tau_i^{2}}-\sigma_W|_{(g_\xi|_{V'_i})}. 
\end{align}
for some $E'_\epsilon>0$. 
On the other hand, Theorem \ref{thm:extendtoW} implies that 
\begin{align}
\sup_{V'_i}
|\frac{\Lambda_{\tau_i}^*(\varphi_0^*\sigma_X)}{\tau_i^{2}}-\sigma_W|_{g_\xi}=O\biggl(\frac{1}{\sqrt{2}^{\frac{i}{D}}}\biggr)
\end{align}
since $\frac{\Lambda_{\tau}^*(\varphi_0^*\sigma_X)}{\tau^{2}}
\rightsquigarrow \sigma_W (\tau\to 0)$ fits into a 
family of $2$-forms on 
$\frac{\Lambda_\tau^*(\sigma_X)}{\tau^{2D}}$ on $\X_{D\xi}\to \A^1$, 
and the relative compactness of $\cup_i V'_i$s in $W^{\rm sm}$. 
Summing up, 
we completed the proof of 
the desired claim \ref{sigmaconv3} for $r(\xi)=1$ case. 

For $r(\xi)>1$ case, recall from the previous section 
\S \ref{sec:XW} that, one can take a nice approximant 
$\xi'=\frac{\tilde{\xi}'}{D}$ of 
$\xi$ as in the sense of Definition \ref{def:nice} such that 
$d(\xi,\xi')$ is arbitrarily small. This is proved in 
Claim \ref{claim:dio}. Hence, if we replace 
$\xi$, $\Lambda_\tau$ and $\X_{D\xi}$ 
in the above arguments 
in the Claims \ref{sigmaconv2} and \ref{sigmaconv3} 
by $\xi'$, $\Lambda^{(\xi')}$, $\X_{D\xi'}$ respectively, 
the desired estimates still hold because of the smallness of the 
exponents of $\tau$ caused by \eqref{ggcomp} (by 
Claim \ref{claim:dio}). Hence, 
the desired claim \ref{sigmaconv3} follows the same proof also 
for $r(\xi)>1$ case. 
\end{proof}

Now we prove that $\sigma_{W_i}$ satisfies 
\begin{Claim}[$\sigma_{W_i}$ vs $\sigma_C$]\label{sigmaconv4}
$\sigma_{W_i}$ on $W_i$ (in the Claim \ref{sigmaconv2}) 
converges to $\sigma_C$ 
on $C^{\rm sm}$ as $i\to \infty$ as smooth convergence 
with respect to $\Psi'_i$. 
\end{Claim}

\begin{proof}
By Claim \ref{sigmaconv3}, pulling back by $\Phi_i$, 
it follows that 
$\sup_{U_C}|\Psi_i^*(((\Psi''_i)^{-1})^*\sigma_{W_i}-
\sigma_{i})|_{\Psi_i^* g_i}\to 0$ for $i\to \infty$. 
On the other hand, 
from Theorem \ref{DS.maps} (as the recap of \cite{DSII}), 
$\Psi_i^* g_i\to g_C$ for $i\to \infty$. 
Hence, combining together, we obtain 
\begin{align}\label{sigmaconv5}
\sup_{U_C}|\Psi_i^*(((\Psi''_i)^{-1})^*\sigma_{W_i}-
\sigma_{i})|_{g_C}\to 0
\end{align}
for $i\to \infty$. The above \eqref{sigmaconv5} and 
Claim \ref{sigmaconv3} imply 
$\sup_{U_C}|(\Psi''_i)^* \sigma_{W_i}-\sigma_C|_{g_C}\to 0 (i\to \infty)$ 
by the triangle inequality. This completes the proof of Claim 
\ref{sigmaconv4}. 
\end{proof}

Now, we want to use the smooth convergence 
in the above 
Claim \ref{sigmaconv4} to construct a Poisson 
deformation 
$(W,\sigma_W)\rightsquigarrow (C,\sigma_C)$ 
as an algebro-geometric enhancement. 

We consider the multi-graded Hilbert scheme in Notation 
\ref{ntt5} and the universal family 
$\tilde{\pi}\colon \tilde{\mathcal{U}}\to B$. We restrict 
it 
to $\overline{G_\xi \cdot [W]}$ and obtain a family over 
$\overline{G_\xi \cdot [W]}$. 
We newly put $B:=\overline{G_\xi \cdot [W]}$ 
and denote the obtained family simply by 
$\pi\colon \mathcal{U}\to B$. 
We put $B^o:=G_\xi[W]$. 
Let $\V_B^{(2,\tau)}$ be the 
$\tau$-eigen-subsheaf of $\pi_* \Omega^2_{\mathcal{U}^{\rm sm}/B}$. 
Then $\V_B^{(2,\tau)}|_{B^o}$ is 
locally free and $G_\xi[(W,\sigma_W)]\subset \V_B^{(2,\tau)}|_{B^o}$ 
is a fiber bundle over $B^o:=G_\xi[W]$. 
We partially compactify 
$G_\xi[(W,\sigma_W)] \to B^o$ to a proper morphism 
$$\overline{G_\xi [(W,\sigma_W)]} \to B$$ 
so that a subsequence of the sequence 
$\{([W_i], \sigma_{W_i})\}_i$ 
has a limit point in $\overline{G_\xi [(W,\sigma_W)]}$ 
with respect to the complex analytic topology, 
say ${\bf 0}$. 
Obviously, ${\bf 0}$ is mapped to $[C] \in B$ 
by the map $\overline{G_\xi [(W,\sigma_W)]} \to B$. 
We pull back $\mathcal{U} \to B$ by the map 
$\overline{G_\xi [(W,\sigma_W)]} \to B$ to get 
$$\pi'\colon\mathcal{U}'\to\overline{G_\xi [(W,\sigma_W)]}.$$ Define the sheaf $\V'(2,\tau)$ on $\overline{G_\xi [(W,\sigma_W)]}$ 
as the $\tau$-eigensubsheaf of $\pi'_*\Omega^2_{\mathcal{U}'^{\rm sm}/\overline{G_\xi [(W,\sigma_W)]}}$. 
By definition, there is a canonical section
$$s_{\rm can} \in \Gamma (G_\xi[(W,\sigma_W)],\V'(2,\tau)).$$ 
Let ${\bf 0} \in U \subset \overline{G_\xi [(W,\sigma_W)]}$ 
be an open neighborhood and put $U^o := U \cap  G_\xi[(W,\sigma_W)]$. Then $s_{\rm can}$ 
determines an element $s^o \in \Gamma(U^o,\V'(2,\tau))$. 
We may assume that $U$ is smooth, $U \setminus U^o$ 
is a divisor of $U$ with simple normal crossing, and $U^o$ 
is affine. 

In the following we write $\pi': \mathcal{U}' \to U$ for $\pi' \vert_{(\pi')^{-1}(U)}: (\pi')^{-1}(U) \to U$.  
We take a $\pi'$-smooth open subset inside $\mathcal{U}'$ and denote it by 
$(\mathcal{U}')^{\rm sm}$. Then we take a (small enough) 
affine open subset $\mathcal{W}$ of $(\mathcal{U}')^{\rm sm}$ 
which still intersects the $\pi'$-fiber over ${\bf 0}$ i.e., $C$, and 
$\Omega^2_{\mathcal{W}/U}$ is trivial bundle i.e., 
\begin{align}\label{loctriv}
i\colon \Omega^2_{\mathcal{W}/U}\xrightarrow{\simeq}
\mathcal{O}_{\mathcal{W}}^{\oplus n(n-1)/2}
\end{align}
We fix such trivialization and denotes its restriction 
over ${\bf 0}$ as $i_{\bf 0}\colon \Omega^2_{C^o}
\xrightarrow{\simeq}
\mathcal{O}_{C^o}^{\oplus n(n-1)/2}$. 
We denote the restriction of $\pi'$ simply as $p\colon  
\mathcal{W}\to U$. 
Denote by $C^o$ the Zariski open subset $p^{-1}({\bf 0})$ of $C$. By abuse of notation, we simply write $W_i$ for the Zariski open subset $p^{-1}([W_i, \sigma_{W_i}])$ of $W_i$.  

Put $\mathcal{W}^o:=\mathcal{W}\cap p^{-1}(U^o)$. 
After the local trivialization \eqref{loctriv}, the canonical 
section $s_{\rm can}$ gives a morphism 
$$f\colon \mathcal{W}^o \to \A^{n(n-1)/2}.$$ 
We regard it as a rational map 
$\mathcal{W}\dashrightarrow (\mathbb{P}^1)^{n(n-1)/2}$.  
We resolve its indeterminancy by a blow up 
$\widetilde{\mathcal{W}}\to \mathcal{W}$,  
so that we obtain a morphism 
$$\tilde{f}\colon \widetilde{\mathcal{W}}\to (\P^1)^{n(n-1)/2}.$$
Then, take a 
flattening of $\widetilde{\mathcal{W}}\to U$ (\cite[5.2.2]{RG}) 
which we denote as $p'\colon \mathcal{W}'\to U'$ 
with a birational proper morphism 
$U'\to U$ (so-called $U^o$-admissible blow up) as $q_U$. 
We denote the obtained 
birational proper morphism $\mathcal{W}'\to 
\mathcal{W}$ as $q$. 

Since $q_U$ is proper and their images in $U$ 
converge to ${\bf 0}\in U$, $[W_i, \sigma_{W_i}]\in q_U^{-1}(U^o)$ also have a subsequence which 
converges to point which we denote as ${\bf 0}'\in U'$. 
Note $q_U({\bf 0}')={\bf 0}$. 
We set $C':=(p')^{-1}({\bf 0}')$, which maps to $C$ 
by $q$. 
Now we analyze $\tilde{f}$ restricted to $C'$. 
Because $q$ is birational proper, it is surjective and in 
particular $C'\to C^o$ is a surjection. 

Now, let us consider what our smooth convergence result 
(Claim \eqref{sigmaconv4}) implies. Roughly put, 
$f\colon \mathcal{W}\to \A^{n(n-1)/2}$
restricted to $[W_i,\sigma_{W_i}]$ encodes $\sigma_{W_i}$ 
via the above local trivialization $i$ in \eqref{loctriv}. 
On the other hand, 
we have constructed $\sigma_C$ in an earlier argument 
and the Claim \eqref{sigmaconv4} says 
the above data converges to that of $\sigma_C$. 

To give precise 
arguments, consider any 
closed point 
$\tilde{y}_\infty$ in $C'$. Since $p'$ is flat, it is an open map in the classical complex analytic topology (\cite[Theorem 2.12]{B-S}, \cite[p73. Corollary]{Douady}). By using this property, we can take a sequence of 
points $\tilde{y}_i\in W_i\subset \mathcal{W}'$ which 
converges to $\tilde{y}_\infty$ as $i\to \infty$. 
Then, since $q$ is continuous, 
if we set $y_i:=q(\tilde{y}_i)$ and 
$y_\infty:=q(\tilde{y}_\infty)$, 
$y_i\in \mathcal{W}$ converges to $y_\infty\in C^o$ as 
$i\to \infty$. Now, we have 
\begin{align}
\tilde{f}(\tilde{y}_\infty)
&=\lim_{i\to \infty}\tilde{f}(\tilde{y}_i)\\
&=\lim_{i\to \infty}f(y_i)\\ 
&=i_{\bf 0}(\sigma_C|_{C^o})(y_\infty), 
\end{align}
where the last equality crucially uses 
Claim \ref{sigmaconv4}. In summary, we have 
\begin{align}\label{sigmaconv4.5}
q^* i_{\bf 0}(\sigma_C|_{C^o})=\tilde{f}|_{C'}. 
\end{align}
In particular, the right hand side descends to 
$C^o$ and takes only finite value i.e., 
$C'\subset \tilde{f}^{-1}(\A^{n(n-1)/2})$. 
Since the right hand side also contains 
$(p')^{-1}(q_U^{-1}(U^o))$ due to the presence of $f$, 
there is a (small enough) smooth affine curve 
$(U'\supset) D'\ni {\bf 0}'$ with which $D'\cap q_U^{-1}
(U^o)=D'\setminus {\bf 0}'$. 
We set $\mathcal{W}^o_{D'}:=(p')^{-1}(D')$ 
and $\bar{D}:=q_U(D')$. Take the normalization 
$\nu: D \to \bar{D}$ and let $q \in D$ be a point 
such that $\nu(q) = 0$.

By \eqref{sigmaconv4.5}, it follows that 
$$\tilde{f}\colon \mathcal{W}^o_{D'}\to \A^{n(n-1)/2}$$ 
exists and further it descends to 
$p^{-1}(\bar{D})$, which we denote by 
$\bar{f}: \mathcal{W}^o_{\bar D} \to \A^{n(n-1)/2}$.  
Note that $i|_{\mathcal{W}^o_{\bar D}}^{-1}\circ \overline{f}$ 
gives a family of fiberwise algebraic (hence holomorphic) 
$2$-forms on $\mathcal{W}^o_{\bar D}$ which 
are translates of $\sigma_W$ generically and 
$\sigma_C|_{C^o}$ on the fiber over ${\bf 0}$.
By pulling back the family $\mathcal{W}^o_{\bar D} \to \bar{D}$ by $\nu: D \to \bar{D}$, we have a family 
$\mathcal{W}^o_D \to D$, which admits a relative symplectic form. 

Now let us make the situation a little bit global. 
For $\pi': \mathcal{U}' \to U$, 
we newly put $\mathcal{W}_D := \mathcal{U}' \times_{U} D$. Let $(\mathcal{W})^{sm}_D \subset \mathcal{W}_D$ 
be the open subset where the map $\mathcal{W}_D \to D$ 
is smooth. 
The canonical section $s_{\mathrm{can}}$ gives a meromorphic relative 2-form of $(\mathcal{W})^{\mathrm  sm}_D \to D$, 
which may possibly have a pole along the central fiber 
$C^{\mathrm sm}$ over $q \in D$. However, by the argument just above, we see that the relative $2$-form does not have a pole and actually is a regular relative 2-form.    
This relative $2$-form extends to whole $\mathcal{W}_D$ 
by the reflexibility of 
$(\Omega^2_{\mathcal{W}_D/D})^{**}$ 
and is a family of symplectic forms on its 
relative smooth locus. Clearly, the restriction to 
its generic fiber $W$ is $\sigma_W$ while the restriction 
to the special fiber $C$ is $\sigma_C$ as they are 
in their open dense subsets. 

In summary, 
we obtain a pointed smooth curve $(D\ni q)$ 
with an isotrivial family of $\Q$-Fano cones  $p\colon \mathcal{W}_D\to D$ with fiberwise $T$-action. 
From our construction, 
it is a Poisson deformation with $p^{-1}(q)=C$ 
and other fibers are all isomorphic to $W$. 

Now, we prove that 
$(W, \sigma_W) \cong (C, \sigma_C)$ by using the above family $\mathcal{W}_D$. 
Let $\mathcal{C}^{\rm univ} \to \mathbb{A}^d$ be the universal Poisson deformation of $C$. The $T$-action on $C$ naturally induces a $T$-action on 
the base space $\mathbb{A}^d$, which turns out be a good action. 
Let $t$ be a local parameter of $D$ at $q$ and let $S_n := \mathrm{Spec}\: \mathcal{O}_{D,q}/(t^{n+1})$. 
Put $\mathcal{W}_n := \mathcal{W}_D \times_D S_n$. Then $\mathcal{W}_n$ is a Poisson deformation of $C$ over $S_n$. Moreover, $T$ acts on $\mathcal{W}_n$ fiberwise. 
By the universality, it uniquely determines a $T$-equivariant map $S_n \to \mathbb{A}^d$. Here the $T$-action on the left hand side is trivial, but the $T$-action on the right hand side is good. Then we see that this map must be the constant map to the origin $0 \in\mathbb{A}^d$.
Using an argument similar to Lemma \ref{Lem3}, we have a 
$T$-equivariant isomorphism of Poisson schemes ${\mathcal W}^{\hat{}} := \{\mathcal{W}_n\}$ 
and $(C \times \mathbb{A}^1)^{\hat{}} := \{C \times S_n\}$. 
On the right hand side, $T$ acts trivially on 
$\mathbb{A}^1$.  
The $T$-equivariant isomorphisms of the formal schemes determines a $T$-equivariant isomorphism of the $\C[[t]]$-algebras 
$\Gamma (C, \mathcal{O}_{{\mathcal W}^{\hat{}}})$
and $\Gamma (C, \mathcal{O}_{(C \times \mathbb{A}^1)^{\hat{}}})$. 

Let us consider the $\C[[t]]$-subalgebras of these algebras 
generated by the $T$-eigenvectors. The isomorphism identifies these two subalgebras 
and we have an isomorphism 
$$\mathcal{W}_D \times_D \mathrm{Spec}\: \C[[t]] 
\cong C \times_{\C} \mathrm{Spec}\:\C[[t]].$$
Then the same argument in \cite[Corollary 3.2]{Namf} can be applied and we have an isomorphism  
$$(W, \sigma_W) \cong (C, \sigma_C)$$ of affine conical symplectic varieties with respect to the $T$-actions.  
We complete the proof of Theorem \ref{ssps}.   
\end{proof}

\begin{Rem}
As noted in
\cite[3.6]{HeinSun} and \cite[5.1]{Zha24}, 
Theorem \ref{ssps} implies that
we can retake Donaldson-Sun degeneration data so that 
$W=W_i=C$ for all $i$. 
\end{Rem}

\begin{Rem}
Also note that 
$W=C$ is proved for the case of affine toric varieties 
in \cite{FOW}, \cite{Ber.toric}, \cite[\S 1]{CS2}, 
\cite[\S 2]{Od24a}. 
\end{Rem}
\section{The Proof of the main theorems}\label{sec:sum}

In this section, we summarize our whole arguments to show the 
main theorems on the symplectic singularities. The first subsection 
is still a preparation and the second subsection provides the proofs of the main theorems (Theorems \ref{Mthm.intro}, 
\ref{Mthm2.intro}). 

\subsection{Symplectic resolution vs smoothability}

Before discussing Theorem \ref{Mthm}, as a preparation, 
we show the equivalence of 
smoothability and existence of 
symplectic resolutions in the 
global polarized setup, 
which may be of independent interest. 

\begin{Thm}\label{smoothing}
Let $(Z, L)$ be a polarized projective symplectic variety of even dimension $n$. Then, it has 
a symplectic projective resolution $\pi\colon Y \to Z$ if and only if 
there is a polarized smoothing
$(\mathcal{Z}_{\Delta}, \mathcal{L}) \to \Delta$ where $(\mathcal{Z}_t, \mathcal{L}_t)$ is a polarized symplectic manifold for $t \in \Delta - \{0\}$ and $(\mathcal{Z}_0, \mathcal{L}_0) = (Z, L)$.  
\end{Thm}

\begin{proof}
Firstly, we show the only if direction. 
We note that, based on \cite[Theorem 4.8]{Fujiki},  
\cite[Theorem (2.2), Claim 3]{Nam.def0} shows the same statement without a polarization. 
Below, we closely follow Fujiki's idea in 
\cite[Theorem 4.8]{Fujiki} and 
check that the smoothing can be chosen together with the 
polarization. 
We divide the proof in the following four steps. 

\begin{description}
\item [{\bf Step 1}] (proves the theorem in the case) when $Y$ is irreducible 
\item [{\bf Step 2}] when the  universal cover is the self-product of 
some irreducible symplectic manifold 
\item [{\bf Step 3}] when the universal covering decomposes into irreducible symplectic manifolds
\item [{\bf Step 4}] general case
\end{description}

\begin{Step}\label{s1} 
This first step treats 
 the case when $Y$ is an {\it irreducible} symplectic manifold.
\vspace{0.2cm}

Let $S$ be the Kuranishi space of $Y$, which is smooth by the 
unobstructedness theorem of Bogomolov, Tian and Todorov (\cite{Bogo}, \cite{Tian}, \cite{Todorov}). 
Let $f\colon  \mathcal{Y} \to S\ni 0$ be the universal family. 
For $s \in S$, we denote by $Y_s$ the fiber $f^{-1}(s)$. Note that $Y_0 = Y$.  
There is a relative holomorphic symplectic 
form $\tilde{\sigma} \in \Gamma (S, f_*\Omega^2_{\mathcal{Y}/S})$ 
which restricts to a holomorphic symplectic form $\sigma_s$ on $Y_s$ for $s \in S$. 
By using variuos cohomological 
comparison theorems mainly due to 
Fujiki (\cite{Fujiki, Fujiki 2}, cf., also \cite{Huy99}), 
we prove the  following claim: 
\begin{Claim}\label{cd}
We have the following commutative diagram 
\begin{equation} 
\begin{CD} 
H^1(Y_s, \Omega^1_{Y_s}) \times H^1(Y_s, \Theta_{Y_s}) @>{\langle \:, \:\rangle_s}>> H^2(Y_s, \mathcal{O}_{Y_s}) \\
@V{id \: \times \: \bigl(-\frac{n}{2}  \sigma_s^{\frac{n}{2}-1}\bar{\sigma}_s^{\frac{n}{2}-1}\: \cup\: {\sigma_s}\bigr)}VV  @V{ \sigma_s^{\frac{n}{2}}\bar{\sigma}_s^{\frac{n}{2}-1} \cup }VV \\
H^1(Y_s, \Omega^1_{Y_s}) \times H^{n-1}(Y_s, \Omega^{n-1}_{Y_s}) @>{(\: , \:)_s}>> H^{n}(Y_s, \Omega^n_{Y_s}) \\ 
@A{-\sigma_s \times id}AA 
@A{\sigma_s \: \cup}AA \\ 
H^1(Y_s, \Theta_{Y_s}) \times 
H^{n-1}(Y_s, \Omega^{n-1}_{Y_s}) 
@>{\langle \: ,  \:  \rangle'_s}>> 
H^n(Y_s, \Omega_{Y_s}^{n-2})
\end{CD}
\end{equation}
Here the map 
$-\frac{n}{2}  \sigma_s^{\frac{n}{2}-1}\bar{\sigma}_s^{\frac{n}{2}-1}\: \cup\: {\sigma_s}$ means the composite $$H^1(Y_s, \Theta_{Y_s}) \stackrel{\sigma_s}\to H^1(Y_s, \Omega^1_{Y_s}) 
\stackrel{-\frac{n}{2}  \sigma_s^{\frac{n}{2}-1}\bar{\sigma}_s^{\frac{n}{2}-1}\: \cup}\longrightarrow H^{n-1}(Y_s, \Omega^{n-1}_{Y_s}).$$
Moreover, all vertical maps are isomorphisms, and the 
horizontal pairing maps are all perfect. 
\end{Claim}
\begin{proof}[proof of Claim \ref{cd}]
Let us check the commutativity of the first square. For $\eta \in H^1(Y_s, \Omega^1_{Y_s})$ and $v \in  H^1(Y_s,\Theta_{Y_s})$ 
with the holomorphic tangent sheaf $\Theta_{Y_s}$, we consider $\eta \cup \sigma_s^{\frac{n}{2}}\bar{\sigma_s}^{\frac{n}{2} -1}$, which is clearly zero. Then we compute 
\begin{align*}
0 &= v \rfloor (\eta \cup \sigma_s^{\frac{n}{2}}\bar{\sigma_s}^{\frac{n}{2} -1}) \\
&= (v \rfloor \eta) \cup 
\sigma_s^{\frac{n}{2}}\bar{\sigma_s}^{\frac{n}{2} -1} + \eta \cup (v \rfloor \sigma_s^{\frac{n}{2}}\bar{\sigma_s}^{\frac{n}{2} -1})\\ 
&= (v \rfloor \eta) \cup 
\sigma_s^{\frac{n}{2}}\bar{\sigma_s}^{\frac{n}{2} -1} + \eta \cup \frac{n}{2}(v \rfloor \sigma_s) \cup  
\sigma_s^{\frac{n}{2}-1}\bar{\sigma_s}^{\frac{n}{2} -1}.
\end{align*}
In the last equality, we use the fact that 
\begin{align*}
v \rfloor \sigma_s^{\frac{n}{2}}\bar{\sigma_s}^{\frac{n}{2} -1} &= \frac{n}{2}(v \rfloor \sigma_s) \cup  
\sigma_s^{\frac{n}{2}-1}\bar{\sigma_s}^{\frac{n}{2} -1} + \bigl(\frac{n}{2} - 1\bigr)(v \rfloor \bar{\sigma_s}) \cup \sigma_s^{\frac{n}{2}} \bar{\sigma_s}^{\frac{n}{2} - 2}\\
&= \frac{n}{2}(v \rfloor \sigma_s) \cup  
\sigma_s^{\frac{n}{2}-1}\bar{\sigma_s}^{\frac{n}{2} -1}
\end{align*}
because 
$v \rfloor \bar{\sigma_s} = 0$. 
Then we have 
$$-\frac{n}{2}(v \rfloor \sigma_s) \cup  
\sigma_s^{\frac{n}{2}-1}\bar{\sigma_s}^{\frac{n}{2} -1} \cup \eta = 
(v \rfloor \eta) \cup 
\sigma_s^{\frac{n}{2}}\bar{\sigma_s}^{\frac{n}{2} -1},$$ which implies the commutativity. The commutativity 
of the second square is similar. 

Now we look at the vertical maps on the right hand side 
of the diagram. 
There is an identification 
$H^n(Y_s, \Omega^n_{Y_s}) \cong \mathbb{C}$ determined 
by the natural orientation $H^{2n}(Y_s, \mathbb{Z}) \cong \mathbb{Z}$. The map 
$H^n(Y_s, \Omega^{n-2}_{Y_s}) \stackrel{\sigma_s}\to H^n(Y_s, \Omega^n_{Y_s})$ is an isomorphism 
by the irreducibility of $Y_s$ 
together with the Serre duality. 
Moreover, the map  
$H^2(Y_s, \mathcal{O}_{Y_s}) \stackrel{\sigma_s^{\frac{n}{2} }\bar{\sigma_s}^{\frac{n}{2}-1}}\longrightarrow H^n(Y_s, \Omega^n_{Y_s})$ is also an isomorphism because the composite  
$$H^0(Y_s, \mathcal{O}_{Y_s}) \stackrel{\bar{\sigma}_s}\to H^2(Y_s, \mathcal{O}_{Y_s})
\stackrel{\sigma_s^{\frac{n}{2} }\bar{\sigma_s}^{\frac{n}{2}-1}}\to H^n(Y_s, \Omega^n_{Y_s})$$ is an isomorphism and the first map  
is an isomorphism. By these isomorphisms, 
we identify 
$H^2(Y_s, \mathcal{O}_{Y_s})$ and $H^n(Y_s, \Omega^{n-2}_{Y_s})$ respectively with $H^n(Y_s, \Omega_{Y_s}^n)$, hence with $\mathbb{C}$. 

We next look at the vertical maps on the left hand side. The second one is an isomorphism because $\sigma_s$ is non-degenerate and the first one is an isomorphism by the holomorphic hard Lefschetz theorem (\cite{Fujiki 2}, Theorem 4.5) together with this fact. 

Since $(\:, \:)_s$ is a perfect pairing, 
the horizontal pairings 
$\langle \:, \:\rangle_s$ and $\langle \:, \:\rangle'_s$ are also perfect. 
We complete the proof of Claim \ref{cd}. 
\end{proof}

\noindent
For simplicity of notation for $s=0$ case, we write respectively $\langle \:, \: \rangle$ for $\langle \:, \: \rangle_0$, $(\:, \:)$ for 
$(\:, \:)_0$ and $\langle \:, \: \rangle'$ for $\langle \:, \: \rangle'_0$. Finally, we write $\sigma$ for 
$\sigma_0$.

\vspace{2mm}

Let $\mathrm{Def}(Y, \pi^*L)$ be the locus of $S$ where $\pi^*L$ extends sideways. 
Note that the Kuranishi space $S$ 
can be assumed to be small enough and simply connected (e.g., polydisk). 
By the identification
$$H^2(Y, \mathbb{Q}) \cong 
\Gamma (S, R^2f_*{\mathbb Q}) 
\cong H^2(Y_s, \mathbb{Q}),$$
the first Chern class $c_1(\pi^*L) \in H^2(Y, \mathbb{Q})$ determines a cohomology class of $H^2(Y_s, \mathbb{Q})$, which we denote by 
$c_1(\pi^*L)_s$. 
For $s \in \mathrm{Def}(Y, \pi^*L)$, the cohomology class $c_1(\pi^*L)_s \in H^2(Y_s, \mathbb{Q})$ is of type $(1,1)$. Then the tangent space 
$T_s\mathrm{Def}(Y, \pi^*L)$ of $\mathrm{Def}(Y, \pi^*L)$ at $s$ is isomorphic to 
$$c_1(\pi^*L)_s^{\perp} := 
\{\eta \in H^1(Y_s, \Theta_{Y_s})\: \vert \: 
\langle c_1(\pi^*L)_s, \eta \rangle_s = 0\},$$
which is a hyperplane of $H^1(Y_s, \Theta_{Y_s})$. Hence 
$\mathrm{Def}(Y, \pi^*L)$ is a 
smooth hypersurface of $S$ passing through $0 \in S$. 

We next consider an element $b \in H^{2n-2}(Y, \mathbb{Q})$.  By the  
identification $$H^{2n-2}(Y, \mathbb{Q}) \cong 
\Gamma (S, R^{2n-2}f_*{\mathbb Q}) 
\cong H^{2n-2}(Y_s, \mathbb{Q}),$$ we have an element $b_s \in H^{2n-2}(Y_s, \mathbb{Q})$.  
Let $R_b \subset S$ be the locus 
where $b_s$ is an element of $H^{2n-2}(Y_s, \mathbb{Q})$ of type $(n-1, n-1)$. 
Then the tangent space $T_sR_b$ 
at $s \in R_b$ coincides with 
$$b_s^{\perp} := \{\eta \in H^1(Y_s, \Theta_{Y_s})\: \vert \: \langle \eta, b_s \rangle'_s = 0\},$$ which is a
hyperplane of $H^1(Y_s, \Theta_{Y_s})$. Hence $R_b$ is a 
smooth hypersurface of $S$. 
In the remainder we do not use the smoothness of $R_b$, but only use the information on $b^{\perp}$. 

For the origin $0 \in S$,  
we compare two tangent spaces  $c_1(\pi^*L)^{\perp}$ and  
$b^{\perp}$ in $H^1(Y, \Theta_Y)$.

Let us consider $\pi: Y \to Z$. Let $q$ be the Beauville-Bogomolov-Fujiki form of $Y$. Then by Lemma 3.5 of 
\cite{BakkerLehn}, 
we have an orthogonal decomposition  $H^2(Y, \mathbb{R}) = \pi^*H^2(Z, \mathbb{R}) \oplus N$ with respect to $q$ and 
$q\vert _N$ is negative definite.
Take $b\in H^{2n-2}(Y,\Q)\cap H^{n-1,n-1}(Y)$ which is represented by an effective algebraic $1$-cycle of $Y$ contracted by $\pi$. 
Since $\cup 
\sigma^{\frac{n}{2}-1}
\bar{\sigma}^{\frac{n}{2}-1}\colon H^{1,1}(Y)\to H^{n-1,n-1}(Y)$ is 
isomorphic by \cite{Fujiki 2}, 
for such $b$, there is a unique element $v_b \in H^1(Y, \Theta_Y)$ such that 
$$b = (v_b \rfloor \sigma) \sigma^{\frac{n}{2}-1}\bar{\sigma}^{\frac{n}{2}-1}
.$$ 
Since $\bar{b} = b$,    
we have $v_b \rfloor \sigma \in H^2(Y, \mathbb{R})$. 
For any element $\alpha \in \pi^*H^2(Z, \mathbb{R})$, we have $q(v_b \rfloor \sigma, \alpha) = 0$. 
In fact, with a suitable positive constant $c$, we have $$q(v_b \rfloor \sigma, \alpha) = c (v_b \rfloor \sigma) \alpha  
\sigma^{\frac{n}{2}-1}\bar{\sigma}^{\frac{n}{2}-1} = c \alpha b = 0.$$
Then $v_b \rfloor \sigma \in N$, and $q(v_b \rfloor \sigma) < 0$. This means that 
$(v_b \rfloor \sigma)^2 \sigma^{\frac{n}{2}-1}\bar{\sigma}^{\frac{n}{2}-1} < 0$.  
We then compute 
\begin{align*}
\langle v_b, b \rangle' &= \langle v_b, (v_b \rfloor \sigma) \sigma^{\frac{n}{2}-1}\bar{\sigma}^{\frac{n}{2}-1} \rangle'\\
&= (-\sigma \rfloor v_b, (v_b \rfloor \sigma) \sigma^{\frac{n}{2}-1}\bar{\sigma}^{\frac{n}{2}-1}) \\
&= - c^{-1}q(\sigma \rfloor v) > 0.
\end{align*}
On the other hand, we have 
\begin{align*}
\langle c_1(\pi^*L),  v_b \rangle &= (c_1(\pi^*L), -\frac{n}{2}(v_b \rfloor \sigma) \sigma^{\frac{n}{2}-1}\bar{\sigma}^{\frac{n}{2}-1}) \\ &=  - \frac{n}{2} (c_1(\pi^*L), b) \\ &= 0.
\end{align*}
Hence $v_b \in c_1(\pi^*L)^{\perp}$ but $v_b \notin b^{\perp}$, 
which is important for our proof. 
We put $H_b := c_1(\pi^*L)^{\perp} \cap b^{\perp}$. 
Then $H_b$ is a hyperplane of $T_0\mathrm{Def}(Y, \pi^*L)$.  Let $\Gamma  \subset H^{2n-2}(Y, \mathbb{Q})$ be the subset 
consisting of all elements $b$ such that $b$ are represented by effective algebraic 1-cyles of $Y$ which are contracted by 
$\pi$ to points. 
Note that $\Gamma$ is a countable set.   
Now we take a smooth complex curve $0\in \Delta \subset \mathrm{Def}(Y, \pi^*L)$ so that 
$T_0\Delta$ is not contained in any $H_b$ with $b \in \Gamma$.  
If we restrict the universal family $\mathcal{Y} \to S$ to 
$\Delta$, then we get a flat deformation $\mathcal{Y}_{\Delta} \to \Delta$ of $Y$. 
Recall that $\pi\colon Y \to Z$ induces a map of Kuranishi spaces $S := \mathrm{Def}(Y) \to \mathrm{Def}(Z)$ (cf., \cite[11.4]{KM}, 
\cite[\S 2, 2.1, 2.2]{Nam.def0}). 
We pull back the universal family $\mathcal{Z} \to \mathrm{Def}(Z)$ by the composite $\Delta \to \mathrm{Def}(Y) \to \mathrm{Def}(Z)$ and get a flat deformation $\mathcal{Z}_{\Delta} \to \Delta$. 

We prove that this is the desired smoothing of $Z$. 
In fact, there is a birational map $\Pi: \mathcal{Y}_{\Delta} \to \mathcal{Z}_{\Delta}$ over 
$\Delta$. For each $s \in \Delta$, $\Pi$ induces a birational map $\Pi_s: Y_s \to Z_s$ of the fibers. 
Then the expeptional locus 
$\mathrm{Exc}(\Pi)$ is mapped onto a closed analytic subset 
$F \subset \Delta$ because $\mathcal{Y}_{\Delta} \to \Delta$ is a proper map. 
We want to prove that $F = \{0\}$ if we shrink $\Delta$ enough. If not, we may assume that $F = \Delta$. Then, for any 
$s \in \Delta$, we have $\mathrm{Exc}(\Pi_s) \ne \emptyset$. By the Chow lemma \cite{Hironaka}, the map $\Pi_s$ is dominated by a projective birational morphism $\tilde{Z}_s \to Z_s$. Hence $\mathrm{Exc}(\Pi_s)$ must contain a curve $C$ such that $\Pi_s(C)$ is a point of $Z_s$. We consider the relative Douady space $D(\mathcal{Y}_{\Delta}/\mathcal{Z}_{\Delta})$ 
parametrizing compact curves  on $\mathcal{Y}_{\Delta}$ contracted to points on $\mathcal{Z}_{\Delta}$. 
By \cite{Fujiki 3}, there are countably many irreducible components of $D(\mathcal{Y}_{\Delta}/\mathcal{Z}_{\Delta})$. By our assumption, 
there is an irreducible component $D$ of $D(\mathcal{Y}_{\Delta}/\mathcal{Z}_{\Delta})$ which dominates $\Delta$. On the other hand,  
each irreducible component of $D(\mathcal{Y}_{\Delta}/\mathcal{Z}_{\Delta})$ is proper over $\Delta$ by \cite{Fujiki 4}. Hence 
$D \to \Delta$ 
is a surjection. Let $C \subset Y$ a curve corresponding to a point of the central fiber $D_0$. Then $C$ extends sideways in 
$\mathcal{Y}_{\Delta} \to \Delta$. This $C$ determines a class $[C] \in H^{2n-2}(Y, \mathbb{Q})$. Moreover, $[C] \in \Gamma$. 
This contradicts the choice of $\Delta$. Therefore, $F = \{0\}$ and $\Pi_s$ is an isomorphism for any $s \in \Delta - \{0\}$. 
Then $Y_s \cong Z_s$ and since $Y_s$ is smooth, $Z_s$ is smooth. Since $\pi^*L$ extends sideways in the flat deformation $\mathcal{Y}_{\Delta} \to \Delta$, we see that $L$ extends sideways in the flat deformation $\mathcal{Z}_{\Delta} \to \Delta$. 
\vspace{0.2cm}

\end{Step}
\begin{Step}\label{s2} The second step treats 
the case when the  universal covering  $Y'$ of $Y$ decomposes into a direct product  $Y_1 \times \cdots \times Y_r$ of 
the {\it isomorphic} irreducible symplectic manifold 
$Y_1\simeq \cdots \simeq Y_r$ i.e., the self-product. 
\vspace{0.2cm}

Let $\nu\colon Y' \to Y$ be the 
universal covering. 
For the simplicity of notation, 
in this step, we identify a differential form on each $Y_i$ 
and its pullback by the $i$-th projection $Y'\to Y_i$. 
Following such convention, we write $$\nu^*\sigma 
= \sigma_1 + \cdots  + \sigma_r$$  
with (the pullback of) 
holomorphic symplectic form $\sigma_i$ on $Y_i$. 
For $g \in  \pi_1(Y)$, we have $g^*(\nu^*\sigma) = \nu^*\sigma$. 
By the uniqueness of the Beauville-Bogomolov decomposition,  
there is a permutation $u: \{1, \cdots , r\} \to \{1, \cdots , r\}$ and 
symplectic isomorphisms $g_i: (Y_i, \sigma_i) \to (Y_{u^{-1}(i)}, \sigma_{u^{-1}(i)})$ such that $g$ acts on $Y'$ as 
\begin{align*}
Y_1 \times \cdots  \times Y_r &\to Y_1 \times \cdots  \times Y_r, \:\: \\ 
(x_1, \cdots , x_r) &\mapsto (g_{u(1)}(x_{u(1)}), \cdots , 
g_{u(r)}(x_{u(r)})).
\end{align*}
We assume that $\pi_1(Y)$ 
permutes the factors $Y_i$ transitively.  

We put $m := \dim Y_i$ and  $$\tau_i := \sigma_1^{\frac{m}{2}}\bar{\sigma_1}^{\frac{m}{2}} \cdots \sigma_{i-1}^{\frac{m}{2}}\bar{\sigma}_{i-1}^{\frac{m}{2}} \sigma_{i+1}^{\frac{m}{2}}\bar{\sigma}_{i+1}^{\frac{m}{2}} \cdots \sigma_r^{\frac{m}{2}}\bar{\sigma_r}^{\frac{m}{2}}.$$
For each $i$ we have a commutative diagram  

\begin{equation} 
\begin{CD} 
H^1(Y_i, \Omega^1_{Y_i}) \times H^1(Y_i, \Theta_{Y_i}) @>{\langle \:, \:\rangle_i}>> H^2(Y_i, \mathcal{O}_{Y_i}) \\
@V{id \: \times \: ((-\frac{m}{2}  \sigma_i^{\frac{m}{2}-1}\bar{\sigma}_i^{\frac{m}{2}-1} \: \cup\: {\sigma_i}) \otimes \tau_i)}VV  
@V{ \cup \sigma_i^{\frac{m}{2}}\bar{\sigma}_i^{\frac{m}{2}-1} \otimes \tau_i }VV \\
H^1(Y_i, \Omega^1_{Y_i}) \times H^{m-1}(Y_i, \Omega^{m-1}_{Y_i}) \otimes \mathbb{C}\tau_i @>{(\: , \:)_i \otimes id}>> H^{m}(Y_i, \Omega^m_{Y_i}) \otimes \mathbb{C}\tau_i \\ 
@A{-\sigma_i \times id }AA 
@A{\cup \sigma_i \otimes id}AA \\ 
H^1(Y_i, \Theta_{Y_i}) \times 
H^{m-1}(Y_i, \Omega^{m-1}_{Y_i}) \otimes \mathbb{C}\tau_i
@>{\langle \: ,  \:  \rangle'_i \otimes id}>> 
H^m(Y_s, \Omega_{Y_i}^{m-2}) \otimes \mathbb{C}\tau_i 
\end{CD}
\end{equation}
By taking the direct sum of these commutative diagrams, we have 

\begin{equation} 
\begin{CD} 
H^1(Y', \Omega^1_{Y'}) \times H^1(Y', \Theta_{Y'}) @>{\langle \:, \:\rangle}>> H^2(Y', \mathcal{O}_{Y'}) \\
@V{\cong}VV  @V{\cong}VV \\
H^1(Y', \Omega^1_{Y'}) \times H^{n-1}(Y', \Omega^{n-1}_{Y'}) @>{(\: , \:)}>> \bigoplus_{1 \le i \le r} (H^{m}(Y_i , \Omega^m_{Yi}) \otimes \mathbb{C}\tau_i) \\ 
@A{\cong}AA 
@A{\cong}AA \\ 
H^1(Y', \Theta_{Y'}) \times 
H^{n-1}(Y', \Omega^{n-1}_{Y'}) 
@>{\langle \: ,  \:  \rangle'}>> 
H^n(Y', \Omega_{Y'}^{n-2})
\end{CD}
\end{equation}

Let $b \in H^{2n-2}(Y, \mathbb{Q})$ be a class determined by an effective algebraic 1-cycle contracted by $\pi$ to a point. In other words, 
$b \in \Gamma$.   
We write $$\nu^*c_1(\pi^*L) = l_1 + \cdots  + l_r$$ with $l_i \in H^1(Y_i, \Omega^1_{Y_i})$ and 
$$\nu^*b = b_1 \otimes \tau_1 + \cdots  + b_r \otimes \tau_r$$ with $b_i \in H^{m-1}(Y_i, \Omega^{m-1}_{Y_i}) \cap H^{2m-2}(Y_i, \mathbb{Q})$. 
Consider the map $Y' \to Y \to Z$ and take its Stein factorization $Y' \stackrel{\pi'}\to Z' \to Z$. 
Then $\nu^*b$ is represented by an effective  algebraic 1-cycle which is contracted to a point by $\pi'$. 

By \cite[Lemma 4.6]{Druel}, we can write $Z' = Z_1 \times \cdots  \times Z_r$ and there are birational morphisms  
$\pi_i : Y_i \to Z_i$ such that $\pi' = \pi_1 \times \cdots  \times \pi_r$. Let $p_i: Z' \to Z_i$ be the $i$-th projection. 
Take an element $\alpha_i \in H^2(Z_i, \mathbb{R})$. Write $c_1(\nu_Z^*L) = \bar{l}_1 + \cdots  + \bar{l}_r$ with 
$\bar{l}_i \in H^2(Z_i, \mathbb{Q})$. Then we have $l_i = \pi_i^*{\bar{l}_i}$.   
Since $\nu^*b$ is represented by an effective algebraic 1-cycle which is contracted by $\pi'$, we have 
$(\nu^*b, (\pi')^*p_i^*\alpha_i) = 0$. We identify $H^m(Y_i, \Omega^m_{Y_i})$ with $\mathbb{C}$ by using the natural orientation of 
$Y_i$. Then $\sigma_i^{\frac{m}{2}}\bar{\sigma}_i^{\frac{m}{2}} = d_i$ for a positive number $d_i$.  
Then we have  
\begin{align*}
(\nu^*b, (\pi')^*p_i^*\alpha_i) &= (b_i, \alpha_i)_{Y_i} \otimes \tau_i\\ 
&= (b_i, \pi_i^* \alpha_i)_{Y_i}\cdot d_1 \cdots  d_{i-1}d_{i+1}\cdots  d_r \\
&= 0.
\end{align*}
Therefore $(b_i, \pi_i^*\alpha_i)_{Y_i} = 0$.  Since $l_i = \pi^*\bar{l_i}$, we have $(b_i, l_i)_{Y_i} = 0$ for any $i$. 
For $b_i$, there is a unique element $v_{b_i} \in 
H^1(Y_i, \Theta_{Y_i})$ such that $$b_i = (v_{b_i} \rfloor 
\sigma_i) \sigma_i^{\frac{m}{2}-1}\bar{\sigma_i}^{\frac{m}{2}-1}.$$ 
Let $q_{Y_i}$ be the Beauvill-Bogomolov-Fujiki form of $Y_i$.  
Then this means that $q_{Y_i}(v_{b_i} \rfloor \sigma_i, l_i) = 0$.    
Since $\nu^*b \ne 0$, we have $b_{i_0} \ne 0$ for some $i_0$. 
By applying again Lemma 3.5 of \cite{BakkerLehn} 
to $\pi_{i_0}: Y_{i_0} 
\to Z_{i_0}$, we see that $q_{Y_{i_0}}(v_{b_{i_0}} \rfloor \sigma_{i_0}) < 0$.  

The fundamental group $\pi_1(Y)$ acts on $Y'$.
Note it is a finite group. In fact, if it is infinite, then, by the Beauville-Bogomolov decomposition (cf. \cite[Th\'{e}or\`{e}me 1]{Beauville83}), the  universal cover $Y'$ of $Y$ is not compact, which contradicts our assumption.  
We put $G = \pi_1(Y)$ and define $v^G_{b_{i_0}} := \sum_{g \in G} g_*v_{b_{i_0}}$. 
By definition $v^G_{b_{i_0}} \in H^1(Y, \Theta_Y)$. 
Then we prove that $v^G_{b_{i_0}} \in c_1(\pi^*L)^{\perp}$ 
and $v^G_{b_{i_0}} \notin b^{\perp}$. 

Since $q_{Y_{i_0}}(v_{b_{i_0}} \rfloor \sigma_{i_0}, l_{i_0}) = 0$, we have $\langle l_{i_0}, v_{b_{i_0}} \rangle_{i_0} = 0$; hence we have 
$$\langle l_1 + \cdots  + l_r, v_{b_{i_0}} \rangle = \langle \nu^*\pi^*L, v_{b_{i_0}} \rangle = 0.$$ 
Note that $\langle \:, \: \rangle$ is $G$-invariant and $\nu^*\pi^*L$ is $G$-invariant. Then $\langle \nu^*\pi^*L, g_*v_{b_{i_0}} \rangle = 0$ 
for any $g \in G$. As a consequence, we have $\langle \nu^*\pi^*L, v_{b_{i_0}}^G \rangle = 0$ and  $v_{b_{i_0}}^G \in c_1(\nu^*\pi^*L)^{\perp}$.      
Recall that 
$g \in G$ acts on $Y'$ as 
\begin{align*}
Y_1 \times \cdots  \times Y_r &\to Y_1 \times \cdots  \times Y_r, \:\:\\ (x_1, \cdots , x_r) &\mapsto (g_{u(1)}(x_{u(1)}), \cdots , 
g_{u(r)}(x_{u(r)}))
\end{align*}
for some permutation $u \in S_r$ and some 
symplectic isomorphisms $g_i$. 

For the element  
\begin{align*}
\omega_i &:= (0, \cdots ,  \sigma_i^{\frac{m}{2}-1} \bar{\sigma}_i^{\frac{m}{2}} \otimes \tau_i, 0, \cdots , 0)\\ 
&\in  H^n(Y', \Omega^{n-2}_{Y'}) 
= \bigoplus_{1 \le i \le r} H^m(Y_s, \Omega_{Y_i}^{m-2}) \otimes \mathbb{C}\tau_i, 
\end{align*}
we have 
$g^*\omega_i =  \omega_{u(i)}$. 
Since $q_{Y_{i_0}}(v_{b_{i_0}} \rfloor \sigma_{i_0}) <0$, we have $\langle  v_{b_{i_0}}, b_{i_0} \rangle'_{i_0} > 0$. More exactly 
$\langle  v_{b_{i_0}}, b_{i_0} \rangle'_{i_0} = c_{i_0}\omega_{i_0}$ with a positive number $c_{i_0}$.
Then $\langle v_{b_{i_0}}, \nu^*b \rangle'  =  \langle  v_{b_{i_0}}, b_{i_0} \rangle'_i  = c_{i_0}\omega_{i_0}$.  
Since $\nu^*b$ is $G$-invariant, we have $\langle g_*v_{b_{i_0}}, \nu^*b \rangle' = c_{i_0} g^*\omega_{i_0} = c_{i_0}\omega_{u(i_0)}$. 
This means that $\langle v^G_{b_{i_0}}, \nu^*b \rangle' \ne 0$, and hence $v^G_{b_{i_0}} \notin (\nu^*b)^{\perp}$. 
Recall that $v^G_{b_{i_0}} \in H^1(Y, \Theta_Y)$. Therefore, we have proven that $v^G_{b_{i_0}} \in c_1(\pi^*L)^{\perp}$ 
and $v^G_{b_{i_0}} \notin b^{\perp}$.

Finally we check that $\mathrm{Def}(Y, \pi^*L) \subset S$ is smooth (possibly of high codimension).  In order to do that, we must check 
that dimension of the tangent spaces $T_s\mathrm{Def}(Y, \pi^*L)$ is constant when $s \in \mathrm{Def}(Y, \pi^*L)$. 
Let $\mathcal{Y} \to S$ be the universal family. Take the universal covering $\mathcal{Y}'$ of $\mathcal{Y}$. Then 
$\mathcal{Y}' \to S$ is a flat deformation of $Y'$ and each fiber $Y'_s$ decompose as $Y_{1,s} \times \cdots  \times Y_{r,s}$ and 
each factor $Y_{i,s}$ is a deformation of $Y_i$. 
By the natural identification $H^2(Y_i) \cong H^2(Y_{i,s})$, the cohomology class $l_i \in H^2(Y_i, \mathbb{Q})$ determines a cohomology class $l_{i,s} \in H^2(Y_{i,s}, \mathbb{Q})$. We note that $l_i \ne 0$ for all $i$ because $\pi^*L$ is nef and big. Hence $l_{i,s} \ne 0$. 
Assume that $s \in \mathrm{Def}(Y, \pi^*L)$. Then we have 
$$c_1(\pi^*L)_s^{\perp} = (c_1(l_{1,s})^{\perp} \oplus \cdots  \oplus c_1(l_{r,s})^{\perp})^G.$$
Here, note that $c_1(l_{1,s})^{\perp} \oplus \cdots  \oplus c_1(l_{r,s})^{\perp}$ lies in the exact sequence 
$$0 \to c_1(l_{1,s})^{\perp} \oplus ... \oplus c_1(l_{r,s})^{\perp} \to 
\bigoplus_{1 \le i \le r} H^1(Y_{i,s}, \Theta_{Y_{i,s}}) 
\stackrel{\oplus \langle c_1(l_{i,s}), \: \rangle_i}\to \bigoplus_{1 \le i \le r} H^2(Y_{i,s}, \mathcal{O}_{Y_{i,s}}) \to 0.$$ 
We take the $G$-invariant parts of the exact sequence. Then we have 
$$H^1(Y_s, \Theta_{Y_s})=
(\bigoplus_{1 \le i \le r} H^1(Y_{i,s}, \Theta_{Y_{i,s}}))^G $$ and 
$$H^2(Y_s, \mathcal{O}_{Y_s})=
(\bigoplus_{1 \le i \le r}H^2(Y_{i,s}, \mathcal{O}_{Y_{i,s}}))^G.$$ 
 
Since the map $\oplus \langle c_1(l_{i,s}), \cdot \rangle_i$ is surjective, the map 
$H^1(Y_s, \Theta_{Y_s}) \stackrel{\langle c_1(\pi^*L)_s, \: \rangle_s}\to H^2(Y_s, \mathcal{O}_{Y_s})$ of the $G$-invariant parts is also  surjective. The dimensions of the spaces on the both sides are constant when $s \in \mathrm{Def}(Y, \pi^*L)$. 
Now we see that the dimensions of 
$c_1(\pi^*L)_s^{\perp}$ are constant. 

The rest of the proof is the same as in Step \ref{s1}. 
\vspace{0.2cm}

\end{Step}
\begin{Step}\label{s3} 
Next, we treat 
 the case when the universal covering  $Y'$ of $Y$ decomposes as a direct product of irreducible symplectic manifolds. 
 \vspace{0.2cm}
 
In this case we can write $Y' = Y^{(1)} \times \cdots  \times Y^{(q)}$ so that each factor $Y^{(j)}$ is a self-product of an  irreducible 
symplectic manifold and each deck transformation 
$g \in G := \pi_1(Y)$ acts diagonally on $Y'$ as 
$$Y^{(1)} \times \cdots  \times Y^{(q)} \stackrel{g_1 \times \cdots  \times g_q}\longrightarrow Y^{(1)} \times \cdots  \times Y^{(q)}.$$
The self-product $Y^{(j)}$ is already discussed in Step \ref{s2}. 
We consider the composite $Y' \to Y \to Z$ and take its Stein factorization $Y' \stackrel{\pi'}\to Z' \to Z$. 
We write $c_1(\nu^*L) = l^{(1)} + \cdots  + l^{(q)}$ with $l^{(j)} \in H^2(Y^{(j)}, \mathbb{Q})$. Note that $Y^{(j)}$ is the direct product 
$Y_1^{(j)} \times \cdots  Y_{r(j)}^{(j)}$ of the (same) irreducible symplectic manifold. Correspondingly,  
we write $l^{(j)} = l_1^{(j)} + \cdots  + l_{r(j)}^{(j)}$ with $l_i^{(j)} \in H^2(Y_i^{(j)}, \mathbb{Q})$. 
 
Take $b \in \Gamma \subset H^{2n-2}(Y, \mathbb{Q})$. Put $m_j := \dim Y^{(j)}$. As in 
Step \ref{s2}, 
$\nu^*b$ is written in terms of $\{b^{(j)}\}$, where $b^{(j)}$ is an element of $H^{2m_j -2}(Y^{(j)}, \mathbb{Q})$. Furthermore, $b^{(j)}$ is written in terms of $\{b_i^{(j)} \}$, where 
$b_i^{(j)} \in H^{2m_{i,j} -2}(Y_i^{(j)}, \mathbb{Q})$. Here $m_{i,j} = \dim Y_i^{(j)}$.   

By the same argument as in Step  \ref{s2}, 
we can find an irreducible factor 
$Y_{i_0}^{(j_0)}$ of $Y^{(j_0)}$ and an element $b_{i_0}^{(j_0)}$ such that $q_{Y_{i_0}}(v_{b_{i_0}^{(j_0)}}, l_{i_0}^{(j_0)}) = 0$ and 
$q_{Y_{i_0}}(v_{b_{i_0}^{(j_0)}}) < 0$. By Step \ref{s2}, 
we see that $v^G_{b_{i_0}^{(j_0)}} \in (l^{(j_0)})^{\perp}$ and 
$v^G_{b_{i_0}^{(j_0)}} \notin (b^{(j_0)})^{\perp}$.
As a result, we see that 
\begin{align*}
v^G_{{b_{i_0}^{(j_0)}}} 
&\in (l^{(1)})^{\perp} \oplus \cdots \oplus (l^{(q)})^{\perp} \\ 
&= (\nu^*c_1(\pi^*L))^{\perp}, \:\:\:\:  \\
v^G_{{b_{i_0}^{(j_0)}}} &\notin (b^{(1)})^{\perp} \oplus \cdots  \oplus (b^{(q)})^{\perp}\\ 
&= (\nu^*b)^{\perp}. 
\end{align*}
Note that $v^G_{{b_{i_0}^{(j_0)}}} \in H^1(Y, \Theta_Y)$. Then this means that 
$$v^G_{{b_{i_0}^{(j_0)}}} \in c_1(\pi^*L)^{\perp}, \:\:\: 
v^G_{{b_{i_0}^{(j_0)}}} \notin b^{\perp}.$$
By the same argument as in 
Step \ref{s2}, 
we can prove that $\mathrm{Def}(Y, \pi^*L)$ is smooth. Now the rest of the proof is the same as in 
Step \ref{s1}. 
\vspace{0.2cm}

\end{Step}
\begin{Step}\label{s4} 

In the general case, we have an  \'etale covering $\nu\colon Y' \to Y$ such that 
$Y' =  Y^{(1)} \times \cdots  \times Y^{(q)} \times T$, where $T$ is an abelian variety of even dimension, 
and each $Y^{(i)}$ as in Step \ref{s3}. 
We may assume that $\nu$ is a Galois covering. Let $G={\rm Gal}(Y'/Y)$ 
be the Galois group of $\nu$. 
We consider the composite $Y' \to Y \to Z$ and take its Stein factorization $Y^{(1)} \times \cdots  \times Y^{(q)} \times T \stackrel{\pi'}\to Z' \to Z.$. By \cite[Lemma 4.6]{Druel}, 
we can write $Z' = Z'_1 \times T$ and there is a birational morphism $\pi'_1: Y^{(1)} \times \cdots  \times Y^{(q)} \to Z'_1$ so 
that $\pi' = \pi'_1 \times id$. Take $b \in \Gamma \subset H^{2n-2}(Y, \mathbb{Q})$. Then this means that $\nu^*b$ is written only 
in terms of $\{b^{(j)}\}$ with $b^{(j)} \in H^{2m_j -2}(Y^{(j)}, \mathbb{Q})$. 
This means that $$\nu^*b^{\perp} = (b^{(1)})^{\perp} \oplus \cdots  \oplus (b^{(q)})^{\perp} \oplus H^1(T, \Theta_T).$$ 
Write $c_1(\nu^*\pi^*L) = l^{(1)} + \cdots  + l^{(q)} + l_T$ with $l_T \in H^1(T, \Omega^1_T) \cap H^2(T, \mathbb{Q})$. 
Then $$c_1(\nu^*\pi^*L)^{\perp} = (l^{(1)})^{\perp} \oplus \cdots  \oplus (l^{(q)})^{\perp} \oplus l_T^{\perp}.$$
As in Step \ref{s3}, 
we can find an element $v^G_{{b_{i_0}^{(j_0)}}} \in H^1(Y, \Theta_Y)$ such that 
$v^G_{{b_{i_0}^{(j_0)}}} \in c_1(\pi^*L)^{\perp}$ and $v^G_{{b_{i_0}^{(j_0)}}} \notin b^{\perp}$.
 
Put $d := \dim T$. Since $\nu^*\pi^*L$ is nef and big, we see that $(l_T)^d > 0$ and $l_T$ is ample. Then the map induced by the 
cup product 
$$H^1(T, \Theta_T) \stackrel{\langle l_T, \: \cdot \:  \rangle_T}\longrightarrow H^2(T, \mathcal{O}_T)$$ is a surjection. 
Together with the results in the previous steps, we see that the 
similar map for $Y$-factor 
$$H^1(Y', \Theta_{Y'}) \stackrel{\langle c_1(\nu^*\pi^*L), \: \cdot \: \rangle_{Y'}}\longrightarrow H^2(Y', \mathcal{O}_{Y'})$$ is a surjection. Then we can prove that $\mathrm{Def}(Y, \pi^*L)$ is smooth 
by the same argument as in 
Step \ref{s2}. 

Now the rest of the proof is the same as in Step \ref{s1}. 

\end{Step}

Conversely, we show ``if direction" 
of Theorem \ref{smoothing}. 
If $(Z, L)$ has a polarized smoothing $(\mathcal{Z}_{\Delta}, \mathcal{L}) \to \Delta$, then 
$Z$ has a projective symplectic 
resolution $\pi: Y \to Z$ by 
\cite[Corollary 2]{Nam.def} together with \cite{BCHM}.  
We complete the proof of 
Theorem \ref{smoothing}. 
\end{proof}


\subsection{Proofs of main theorems: existence of canonical 
torus action}\label{subsec:6.2}

Now, we are ready to prove 
our main theorems on the 
algebraic torus action 
on symplectic singularities. 

\begin{Thm}\label{Mthm}
Let $(\bar{X},L)$ be a polarized 
projective symplectic variety.  Suppose that 
$(\bar{X},L)$
satisfies either of the  following equivalent conditions  (cf. Theorem \ref{smoothing}): 
\begin{enumerate}
\item $\bar{X}$ has a 
symplectic projective resolution, or 
\item \label{1.1ii.sec6} $(\bar{X},L)$ has a smoothing (as a polarized variety). 
\end{enumerate}

Then, the analytic germ of $x\in \bar{X}$ is that of a 
(canonical) 
conical affine symplectic variety $C$ at the vertex $0\in C\curvearrowleft (\mathbb{G}_m)^r$ with $r\ge 1$. 

Furthermore, $0\in C$ has a (singular) hyperK\"ahler cone  metric, 
which in particular has a canonical rescaling action of 
the multiplicative group $\R_{>0}$ (as a real Lie subgroup of the $(\C^*)^r$) with positive weights 
of generators of $\mathcal{O}_{C,0}$.

\end{Thm}

\begin{proof}
Applying Theorem \ref{smoothing} to $\pi\colon \tilde{X} \to \bar{X}$, we have a polarized smoothing 
$\bar{\mathcal X} \to \Delta$ 
of $(\bar{X}, L)$. We are now in a situation of Theorem \ref{DSII.thm}. Let $x \in X \subset \bar{X}$ be an open neighborhood of $x$ and start with $(X, \sigma_X)$. By Theorem 
\ref{thm:extendtoW}, we have a scale up Poisson deformation of $X$ degenerating to $W$. By Corollary \ref{formal.isom}, there is an isomorphism of Poisson formal schemes 
$(X, x)^{\hat{}} \cong (W, 0)^{\hat{}}$. Next, by Theorem 
\ref{ssps}, there is a Poisson deformation of $W$ degenerating to $C$, and actually $W = C$. 
By combining these two degenerations, we have an isomorphism of Poisson formal schemes $\hat{f}: (X, x)^{\hat{}} \cong (C, 0)^{\hat{}}$. By Artin's approximation theorem (\cite{Artin}, Corollary (1.6), 
also \cite{HironakaRossi} for isolated singularities case), we have an isomorphism of complex analytic germs $f': (X, x) \cong (C, 0)$, which induces an isomorphism 
$((X, x), (f')^*\sigma_C) \cong 
((C,0), \sigma_C)$ of symplectic singularities. Since $f'$ is an approximation of $\hat{f}$, the symplectic form $(f')^*\sigma_C$ 
does not necessarily coincide with the original $\sigma_X$. 
The $T(\mathbb{C}) (\cong (\mathbb{C}^*)^r)$-action on the right hand side induces a local action of $(\mathbb{C}^*)^r$ on 
$(X, x)$ and we have the result. 
\end{proof}

\begin{Rem}\label{Kaledin's conj}
Let $X_i^0 \subset X$ be the symplectic leaf of $X$ passing through $x$ (cf. \cite[Theorem 2.3]{Kaledin}).  
Then the Poisson structure of $X$ induces a non-degenerate Poisson structure of $X_i^0$. By the Poisson structure,  
$X_i^0$ is a holomorphic symplectic manifold.  
By {\it loc.cit}, 
there is a product decomposition of formal Poisson schemes 
$$(X, x)^{\hat{}} \cong (X_i^0, x)^{\hat{}} \: \hat{\times} \: Y_x$$ where $Y_x$ is a transversal slice for 
$x \in X_i^0$ 
(cf., also \cite[Appendix]{KapSch} for its enhancement to 
analytic statements). 

The original Kaledin's conjecture (cf. \cite[Conjecture 1.8]{Kaledin 2}) claims 
that $Y_x$ admits a good $\mathbb{G}_m$-action so that its Poisson bracket has a 
negative weight (or equivalently, its symplectic form has a positive weight). 
This conjecture holds true. In fact, by the proof of the Theorem, $(X, x)^{\hat{}}$ admits a good $\mathbb{G}_m$-action (which depends on the choice of a homomorphism 
$\mathbb{G}_m \to T$), where the Poisson bracket has a negative weight. In this situation one can take the 
product decomposition above in a $\mathbb{G}_m$-equivariant way so that both 
$(X_i^0, x)^{\hat{}}$ and $Y_x$ have good $\mathbb{G}_m$-actions.      
\end{Rem}

\begin{ex}[Singularities of O'Grady's $10$-folds]\label{O'Grady10}

Let $S$ be a projective K3 surface. For a coherent sheaf $E$ on $S$, 
the so-called Mukai vector 
$v(E) \in H^0(S) \oplus H^2(S) \oplus H^4(S)$ means 
$$v(E) := ch(E)\sqrt{td(S)} = (r(E), c_1(E), \frac{1}{2}(c_1^2(E) - 2c_2(E)) + r(E)).$$
Here $r(E)$ is the rank of $E$ and $c_i(E)$ are the $i$-th Chern classes of $E$. 
For $v = (1, 0, -1)$, take a $2v$-generic ample divisor $H$ and consider the moduli space $M_{2v}$ of $H$-semistable 
torsion free sheaves $E$ on $S$ with Mukai vector $2v$ (equivalently that $(r(E), c_1(E), c_2(E)) = (2, 0, 4)$). 
O'Grady \cite{O} showed that $M_{2v}$ has a projective symplectic resolution $\tilde{M}_{2v} \to M_{2v}$ and $\tilde{M}_{2v}$ 
is a 10 dimensional irreducible symplectic manifold with $b_2 = 23$. There is a stratification 
$M_{2v} \supset S^2M_v \supset \Delta_{M_v}$ where $S^2M_v = \mathrm{Sing}(M_{2v})$.  
Here $S^2M_v$ parametrizes the sheaves $E$ of the form $E = F_1 \oplus F_2$, where $F_i$ are stable sheaves with 
$v(F_i) = v$. The smallest stratum $\Delta_{M_v}$ parametrizes the sheaves $E$ of the form $E = F \oplus F$ with 
$v(F) = v$. $M_{2v}$ has $A_1$-singularities along $S^2M_v - \Delta_{M_v}$. Lehn and Sorger \cite{LS} has proved that, 
for any point $[E] \in \Delta_{M_v}$, the complex anaytic germ $(M_{2v}, [E])$ is isomorphic to $(\bar{O}_{[2,2]}, 0) \times (\mathbf{C}^4, 0)$. 
Here $\bar{O}_{[2,2]}$ is the closure of the nilpotent orbit $O_{[2,2]} \subset \mathfrak{sp}_4$ of Jordan type $(2,2)$.       

The result of \cite{LS} can be re-interpreted by using our theory. Let us briefly recall the arguments of \cite{LS}. 
Take $[E] \in \Delta_{M_{2v}}$ and let $\mathcal{O}_{{M_{2v}}, [E]}$ and $\hat{\mathcal{O}}_{{M_{2v}}, [E]}$ be the local ring of $M_{2v}$ at 
$[E]$ and its completion with respect to the maximal ideal. 
Let $A = \mathbb{C}[\mathrm{Ext}^1(E, E)]$ be the polynomial ring over $\mathrm{Ext}^1(E, E)$ and let $\hat{A}$ be 
the completion of $A$ with respect to the maximal ideal $\mathfrak{m}$ consisting of the polynomials vanishing at the 
origin. Put $$\mathrm{Ext}^2(E, E)_0 := \mathrm{Ker}[\mathrm{Ext}^2(E, E) \stackrel{tr}\to H^2(S, \mathcal{O}_S)].$$
Since $E = F \oplus F$ with a stable sheaf $F$ with $v(F) = v$, we have $\mathrm{Ext}^1(E, E) = \mathfrak{gl}_2 \otimes 
\mathrm{Ext}^1(F, F)$. In particular, $\dim \mathrm{Ext}^1(E, E) = 16$.  
Moreover, $\mathrm{Ext}^2(E, E)_0 = \mathfrak{sl}_2$ and $\mathrm{Aut}(E) = GL_2$. Then $\hat{\mathcal{O}}_{M_{2v}, [E]}$ can 
be described by means of the Kuranishi map. More precisely, there is an element $f \in \mathrm{Ext}^2(E, E)_0 \otimes 
\mathfrak{m}^2\hat{A}$ which satsifies the following properties: 

(i) $f$ is $\mathrm{Aut}(E)$-equivariant.   

(ii) The initial part $f_2 \in \mathrm{Ext}^2(E, E)_0 \otimes \mathfrak{m}^2/\mathfrak{m}^3$ of $f$ is given by 
the cup product $f_2(e) := e \cup e$ with $e \in \mathrm{Ext}^1(E, E)$. 

(iii) Let $\mathfrak{a} \subset \hat{A}$ be the ideal generated by the image of the map 
$f: \mathrm{Ext}^2(E, E)_0^* \to \mathfrak{m}^2\hat{A}$. Then $\hat{\mathcal{O}}_{{M_{2v}}, [E]} \cong (\hat{A}/\mathfrak{a})^{\mathrm{Aut}(E)}$. 

Note that $\mathrm{Aut}(E)$ does not act on $\hat{A}$ effectively, but $P\mathrm{Aut}(E)$ acts effectively on 
$\hat{A}$. Hence we have $(\hat{A}/\mathfrak{a})^{\mathrm{Aut}(E)} = (\hat{A}/\mathfrak{a})^{SL_2}$. Moreover, if we put $\mathfrak{a}^0 := \mathfrak{a} \cap \hat{A}^{SL_2}$, then  $(\hat{A}/\mathfrak{a})^{SL_2} = \hat{A}^{SL_2}/\mathfrak{a}_0$. 

We introduce a decreasing filtration $F^{\cdot}$ of $\hat{A}$ as follows. Writing a nonzero element $h \in \hat{A}$ as a (infinite) sum  
$h = h_0 + h_1 + \cdots $ of homogeneous elements 
$h_\nu$ 
with degree $\nu$. 
Define $\mathrm{ord}(h) := \mathrm{min}\{\nu \:  \vert  \: h_{\nu} \ne 0\}.$ 
Then we define $F^i = \{h \in \hat{A}\: \vert \: \mathrm{ord}(h) \geq i\}$. This filtration induces a filtration $\bar{F}^{\cdot}$ of $\hat{A}^{SL_2}$, which also induces a filtration $\bar{\bar{F}}^{\cdot}$ of $\hat{A}^{SL_2}/\mathfrak{a}_0$. In fact, we only have to put 
$\bar{\bar{F}}^i := \bar{F}^i /(\bar{F}^i \cap \mathfrak{a}_0)$. By the inclusion $\mathcal{O}_{{M_{2v}}, [E]} \subset 
\hat{\mathcal{O}}_{{M_{2v}}, [E]}$, we can also introduce a decreasing filtration $\mathcal{F}^{\cdot}$ of    
$\mathcal{O}_{{M_{2v}}, [E]}$.  Lehn and Sorger showed in \cite[\S 4]{LS} that 
$$\mathrm{Spec}(\gr_{\mathcal{F}}\mathcal{O}_{{M_{2v}}, [E]}) = \mathrm{Spec}(\gr_{\bar{\bar{F}}}\hat{\mathcal{O}}_{{M_{2v}}, [E]})$$ is isomorphic to the cone 
$\bar{O}_{[2,2]} \times \mathbb{C}^4 \subset \mathfrak{sp}_4 \times \mathbb{C}^4 (= \mathbb{C}^{10} \times \mathbb{C}^4)$. 
Note that the coordinates of the first factor $\mathbb{C}^{10}$ have weights $2$ and the coordinates of the second factor 
$\mathbb{C}^4$ have weights 1. Afterwards, they proved that $(M_{2v}, [E]) \cong (\mathrm{Spec}(\mathrm{gr}_{\mathcal F}(\mathcal{O}_{{M_{2v}}, [E]})), 0)$ by direct calculations (cf. \cite{LS}, \S 3, \S 5). 

This final part can be explained as follows by using our theory. 
For the filtration $\mathcal{F}^{\cdot}$, we get an associated order function $$v_{\mathcal{F}} : \mathcal{O}_{{M_{2v}}, [E]} - \{0\} 
\to \mathbb{Z}_{\geq 0}$$ by putting 
$$v_{\mathcal{F}}(h) := \mathrm{max}\{m \:\: \vert \:\: h \in \mathcal{F}^m\}$$ for $h \in \mathcal{O}_{{M_{2v}}, [E]} - \{0\}$. 
In our case, $\mathrm{gr}_{\mathcal{F}}\mathcal{O}_{{M_{2v}}, [E]}$ is an integral domain with the same Krull dimension as 
$\mathcal{O}_{{M_{2v}}, [E]}$. This implies that $v_{\mathcal F}$ is a quasi-monomial valuation by  \cite{Tei03}, \cite{Tei14} (cf. \cite{LLX}, Lemma 2.7).    
On the other hand, the cone $\bar{O}_{[2,2]} \times \mathbb{C}^4$ is a hyperK\"ahler cone by \cite[Theorem 1.1]{KS} (cf., also \S \ref{subsec:HKQ}). 
In particular, it is a K-polystable Fano cone. By  \cite[Theorem 1.3]{LX} (cf.,\cite[Theorem 4.14]{LLX}), the degeneration of $([E] \in M_{2v})$  to $(0 \in \mathrm{Spec}(\mathrm{gr}_{\mathcal F}(\mathcal{O}_{{M_{2v}}, [E]})))$ 
coincides with the $2$-step degeneration of \cite{DSII} 
(recall Theorems \ref{DSII.thm}, \ref{DS.maps} and 
\ref{AGlc}). 
Then our main theorem implies that 
$(M_{2v}, [E]) \cong (\mathrm{Spec}(\mathrm{gr}_{\mathcal F}(\mathcal{O}_{{M_{2v}}, [E]})), 0)$. 
    
\end{ex}

\vspace{3mm}
Now we would like to discuss generalization of 
Theorem \ref{Mthm} to more general symplectic 
singularities. 
From our proof of Theorem \ref{Mthm} above, 
we see that only problem is 
the validity of Donaldson-Sun theory for 
a certain nice Ricci-flat (singular) K\"ahler metric. 
Thus, our methods naturally extend to show the following. 

\begin{Thm}\label{Mthm2}
Suppose a symplectic singularity 
$x\in X$ has 
a singular hyperK\"ahler metric $g_X$ and 
a holomorphic symplectic form $\sigma_X$ which is 
parallel with respect to $g_X$ on $X^{\rm sm}$, 
with which Conjecture \ref{conj:genDS} holds. 

Then, the analytic germ of $x\in X$ is that of some 
(canonical) 
affine conical symplectic variety $C$ at the vertex $0\in C\curvearrowleft (\mathbb{G}_m)^r$. 
Furthermore, $0\in C$ has a (singular) hyperK\"ahler cone  metric, 
which in particular has 
a canonical rescaling action of 
$\R_{>0}$ (as a real Lie subgroup of the the $(\C^*)^r$) with positive weights 
of generators of $\mathcal{O}_{C,0}$. 
\end{Thm}

\begin{Rem}[$r=1$]
We recall again that in the setups of Theorems \ref{Mthm} and \ref{Mthm2}, 
the rank $r=r(\xi)$ of the torus is also proven to be $1$ conditionally, 
in Theorem \ref{r1prop} (cf., Conjecture \ref{r1conj}). 
See Remark \ref{LSconj} for its further consequence. 
\end{Rem}

Now we turn to more differential geometric side of our results. 
From our main theorems \ref{Mthm} and \ref{Mthm2}, it follows that 
as the singularity of hyperK\"ahler {\it metric}, 
symplectic singularity is close to (Riemannian) 
cone in the following differential geometric sense. 
Indeed, from \cite[1.1]{ChiuSzek} and \cite[1.2, 1.4]{Zha24}, 
Theorem \ref{ssps} 
imply the following differential geometric asymptotics of $g_X$. 

\begin{Cor}[Conicity of the hyperK\"ahler metric]\label{cor:HSZha}

\begin{enumerate}
\item (cf., \cite[1.1]{ChiuSzek}) \label{HSZha1}
In the setup of Theorem \ref{Mthm.intro}, 
there are positive constants $c, \alpha, r_0\in \R_{>0}$ 
and a biholomorphism $\Phi^{(r_0)}$ 
from $B(0\in C,r_0)$ into 
open neighborhood of $x\in X$ which maps $0\mapsto x$, 
such that for any $0<r<r_0$, there is a 
real function $u_r\colon B(0\in C,r)\to \R$ 
which is $C^\infty$ on the regular locus 
$B(0\in C,r)\cap C^{\rm sm}$ 
such that 
$(\Phi^{(r_0)})^* \omega_X-\omega_C=\sqrt{-1}\partial 
\overline{\partial}u_r$ on $B(0\in C,r)\cap C^{\rm sm}$ and 
\begin{align*}
\sup_{B(0\in C,r)}|u_r|\le cr^{2+\alpha}. 
\end{align*}

\item (cf., \cite{HeinSun}, \cite[4.3]{ChiuSzek}, \cite[1.1, 1.4]{Zha24}) \label{HSZha2}
In the setup of Theorem \ref{Mthm.intro} 
with only isolated singularity $x\in X$, 
there is a local diffeomorphism  $\Psi\colon U_C \to X, 0\mapsto x$ on to their images where 
$U_C \subset C$ is an open neighborhood of $0\in C$
(similarly as Theorem \ref{ssps} \eqref{DSiii}) and a 
positive real number $\delta$ 
which satisfy the following: 
$$|\nabla^k_{g_C}(\Psi ^*g_X -g_C)|_{g_C}=O(r^{\delta-k}),$$ 
for any $k\in \Z_{\ge 0}$. Here, $r$ denotes the 
distance function on $C$ from the vertex with respect to 
$g_C$. 
\end{enumerate}
\end{Cor}

\begin{proof}
\eqref{HSZha1} follows from \cite[Thereom 1.1]{ChiuSzek} 
combined with our Theorem \ref{Mthm}. 
\eqref{HSZha2} follows from \cite[Corollary 4.3]{ChiuSzek} 
at least in the setup of Theorem \ref{Mthm}. 
\end{proof}
\begin{Rem}\label{Zha1exp}
The latter \eqref{HSZha2} is expected to generalize, 
without any new difficulty, also to 
the setup of Theorem \ref{Mthm2} as discussed in 
\cite[\S 1]{Zha24} (cf., \cite{HS, ChiuSzek}). 
\end{Rem}


\subsection{HyperK\"ahler quotients case}\label{subsec:HKQ}
In this subsection, we discuss a particular class of examples where 
our main Theorem \ref{Mthm2} applies, i.e., 
hyperK\"ahler quotients under some conditions. For many of 
classical or known examples, 
some standard local $\C^*$-actions were known and 
our subsection aims at fitting them to the framework of this paper 
in more general cases. 

We make a setup of more general hyperK\"ahler quotients as follows, 
recalling \cite[\S 3]{HKLR} (cf., also \cite{Nakajima, Kir}). 

\begin{Setup}\label{set:HKLR}

For a smooth complex affine variety $M$ with a 
hyperK\"ahler structure $(g,I_1,I_2,I_3)$,  
we write the K\"ahler forms as $\omega_i (i=1,2,3)$. 
We consider a compact Lie group $K$ and 
its complexification of $K$, which we denote 
as $K^{\bf C}=G$ and assume it is an affine 
(automatically reductive) algebraic group. 
Suppose $M$ is complete with respect to $g$ and 
there is a tri-Hamitonian action of a 
compact Lie group $K$ on $M$ which preserves the hyperK\"ahler structure on 
$M$, and its action extends to an algebraic action of 
$G$. In particular, there is a hyperK\"ahler moment map 
$\mu=(\mu_1,\mu_2,\mu_3)\colon M\to \mathfrak{k}^*\otimes_{\R}\R^3$, 
where $\mathfrak{k}^*$ means the dual of Lie algebra $\mathfrak{k}$ 
of $K$. 
Denote the $K$-invariant part of $\mathfrak{k}^*$
with respect to the co-adjoint action as $(\mathfrak{k}^{*})^K$. 
For any 
$\zeta_i\in (\mathfrak{k}^{*})^K (i=1,2,3)$, 
we put $\zeta:=(\zeta_1,\zeta_2,\zeta_3)$ and set 
\begin{align}
X_\zeta:=\mu^{-1}
(\zeta)/K.     
\end{align}
Denote the locus in $\mu^{-1}(\zeta)$ where $K$ acts freely, 
by $\mu^{-1}(\zeta)^{\rm sm}$ and set $X_\zeta^{\rm sm}:=\mu^{-1}
(\zeta)/K$. This is known to admit a natural hyperK\"ahler structure 
$(g_\zeta,I_1,I_2,I_3)$ 
by \cite[\S 3 D]{HKLR}, where $g_\zeta$ is the 
hyperK\"ahler metric 
and $I_i$ are complex structures as they descend from that of $M$. 
We denote the corresponding K\"ahler forms as $\omega_i$ or 
$\omega_i(\zeta)$. 
There are also general results on the structure of 
whole $X_{\zeta}$ with singularities (cf., e.g., \cite{LSj, DancerSwann, May}). 

We consider a complex variable $\zeta_\C:=\zeta_2+\sqrt{-1}\zeta_3$ and 
also define the map $\mu_\C:=\mu_2+\sqrt{-1}\mu_3 
\colon M\to \mathfrak{k}\otimes_{\R} \C$, 
which we assume to be algebraic. 

If there is a 
character $\chi\colon G\to \C^*$ whose derivation 
gives $\chi_*\colon \mathfrak{k}\to \mathfrak{u}(1)=\sqrt{-1}\R$, which we regard as 
an element of $i\mathfrak{k}^*$, and $\zeta_1=\sqrt{-1}\chi_*$ holds, 
one can also consider the GIT quotient 
\begin{align}
X_\zeta^{\rm GIT}:=\mu_\C^{-1}(\zeta_2+\sqrt{-1}\zeta_3)//_{\chi}G. 
\end{align}
\end{Setup}
By \cite[Theorem 1.4 (i)]{May} (generalizing the classical 
Kempf-Ness theorem), we have a canonical homeomorphism 
$X_\zeta \simeq X_\zeta^{\rm GIT}$ in this case. 

\vspace{2mm}
Below, we assume mild 
conditions to make the structure of $X_\xi$ more tractable 
and show basic properties. In particular,  Donaldson-Sun theorey (\cite{DSII}) 
extends to this setup. 

\begin{Prop}[Basic local structure of hyperK\"ahler quotients]\label{DSHKQ}
In the above Setup \ref{set:HKLR}, 
take any character $\chi\colon G\to \C^*$ and the derivative 
$\chi_*$, which we regard as 
an element of $i\mathfrak{k}^*$. 
Fix $\zeta_1:=\sqrt{-1}\chi_*$ and some other 
$\zeta_2, \zeta_3\in (\mathfrak{k}^{*})^K$, 
we make the following standard assumptions: 
\vspace{-6mm}
\begin{quote}
\begin{Ass*}
\begin{enumerate}[label=(\alph*)]
\item \label{as:hodge}
$[\omega_i] (i=1,2,3)$ is an integral class 
i.e., lies in ${\rm Im}(H^2(M,\Z)\to H^2(M,\R))$. 
\item \label{as:equidim}
$\mu_\C$ have equi-dimensional fibers with smooth image 
(or alternatively, $X_\zeta$ is Cohen-Macaulay). 
\item \label{asi} (Connectedness) 
There is an analytic 
neighborhood $U$ of $\zeta_\C$ in $(\mathfrak{k}^{*})^K\otimes_{\R} 
\C$ 
such that for {\bf any} pair of 
$\eta_\C=(\eta_2, \eta_3)\in U$, 
$X_{\zeta=(\zeta_1,\eta_2,\eta_3)}$ is connected. 
\item \label{asii} (Generic regularity I) 
For {\bf some} pair of 
$\eta_2, \eta_3\in (\mathfrak{k}^{*})^K$, we have  $X_{\zeta=(\zeta_1,\eta_2,\eta_3)}=X_{\zeta=(\zeta_1,\eta_2,\eta_3)}^{\rm sm}\neq \emptyset$. 
\item \label{asiii} (Generic regularity II) 
For our 
$\zeta_2, \zeta_3\in \mathfrak{k}^*$ fixed above, 
the $K$-action on 
an open subset of $\mu^{-1}(X_\zeta)$ 
has finite stabilizer groups. 
\end{enumerate}
\end{Ass*}
\end{quote}
Then, it follows that for {\bf any} pair of 
$\zeta_2, \zeta_3\in (\mathfrak{k}^{*})^K$, 
the hyperK\"ahler quotient 
$$X_{\zeta=(\sqrt{-1}\chi_*,\zeta_2,\zeta_3)}$$ 
with its any point $x\in X_\zeta$ satisfies all of the following: 
\begin{enumerate}
    \item (Symplectic singularity) it is a symplectic singularity, 
    \item \label{SAHKQ} (Standard $\C^*$-action cf., \cite{May}) its analytic germ  is isomorphic to that of the vertex of 
    a conical symplectic variety 
    with a standard $\C^*$-action
    \item ($r(\xi)=1$) \label{R1HKQ} the $\C^*$-action in \eqref{SAHKQ} 
    corresponds to the volume-minimizing $\xi$ (in the sense of 
    \cite{MSY2, Li}), which gives affirmative answer to 
 Conjecture \ref{r1conj} (review our \S \ref{sec:2} or \cite[\S 2]{Od24a} 
 for more details of the context) 
\item \label{DSHKQiii} (Donaldson-Sun theory) 
the singular hyperK\"ahler metric $g_{X_\zeta}$ 
on $X_\zeta\ni x$ 
which comes from the 
hyperK\"ahler quotient description of Setup \ref{set:HKLR} 
satisfies 
Conjecture \ref{conj:genDS} (Theorem \ref{DS.maps}) 
i.e., Donaldson-Sun theory \cite{DSII} on the local metric tangent cone 
holds. 
\end{enumerate}
\end{Prop}

We will also prove more properties on this $X_\zeta$ later cf., 
Theorem \ref{r1prop} (case \eqref{r1prop:HKQcase}). 

\begin{ex}
There are many 
examples which appear as above $X_\zeta$, e.g.,  
A-type nilpotent orbit closures, A-type Slodowy slices, 
many of Nakajima quiver varieties (\cite{KobSwann}, \cite{Nakajima} but 
also cf., \cite{BS}) 
and many of 
toric hyperK\"ahler varieties (\cite{Goto, BD, HSt}). 
\end{ex}

\begin{proof}

Firstly, note that 
the following is essentially known to experts 
on recent developments of K\"ahler geometry, after \cite{DSII}. 

\begin{Claim}[cf., {\cite[\S 3.1]{Uniqueness}, also \cite{Zha24} and the  original \cite{DSII}}]\label{completeDSII}
Take any sequence of Ricci-flat K\"ahler manifolds $(x_i\in  M_i,J_i,L_i,h_i,g_i,\omega_i) (i=1,\cdots)$ where $(L_i,h_i)$ are 
Hermitian metrized line bundles on $(M_i,J_i)$ 
whose curvature form $\omega_i$ 
is K\"ahler and corresponds to the K\"ahler metrics $g_i$ of $M_i$, 
which are all {\it complete}. 
We also assume there is a sequence of parallel holomorphic volume form
\footnote{This mild condition is put just to prove 
log terminality of the polarized limit space and often automatically holds 
(as in our Setup \ref{set:HKLR}). 
We thank S.Sun for pointing out this subtle issue.}
$\Omega_i$ on $M_i$ for each $i$. 

Suppose it has a (pointed) polarized limit space 
$(x\in X,L,h,g_X,\omega_X)$
in the sense of 
\cite[\S 2.1]{DSII} (cf., also \cite{DS}), which in particular implies 
the smooth convergence of 
$J_i, g_i$ outside the singular locus of $X$. 

Then, 
$(x\in X,J,g_X)$ is log terminal and  has a unique local metric tangent 
cone 
with natural affine algebraic variety structure, and further 
satisfies Conjecture \ref{conj:genDS} (Theorem \ref{DS.maps}) 
verbatim as \cite{DSII}. 
\end{Claim}

\begin{proof}[proof of Claim \ref{completeDSII}]
Let us begin by recalling that \cite{DSII} firstly proves that the local metric tangent cone $C_x(X)$ of $x\in X$ 
is a complex analytic space ({\it op.cit} Theorem 1.1), by constructing some holomorphic sections on $L_i$ by the theory of L.~H\"ormander on solutions of $\bar{\partial}$-equation with $L^2$-estimates, as a technical core. Then they use  that to obtain the analytic local 
realization of open subsets of $C_x(X)$ 
inside $\C^N$ for some positive integer $N$. 
Since the H\"ormander theory with $L^2$-estimates works also for 
{\it complete} non-compact K\"ahler manifolds as well (cf., e.g., \cite[Chap 
VIII \S  4 Theorem 4.5 (p.370)]{DemaillyAGbook}, 
also \cite{Uniqueness}, \cite{Zha24}), 
the proof of \cite[Theorem 1.1]{DSII} extend to our setup. 

Then as in {\it op.cit} subsection \S 2.3, one can prove that 
$C_x(X)$
is an affine algebraic variety ($\Q$-Fano cone or log terminal cone) 
in the same way. 
Further, for the realization of the 2-step degeneration (Theorem 
\ref{DS.maps}), 
one can again just follow their arguments of subsection 
\S 3.3 (and somewhat before that) of \cite{DSII} in the same way. Hence, Claim \ref{completeDSII} holds.     
\end{proof}

Now we consider the application of the above, to prove 
Proposition \ref{DSHKQ}. We consider the set 
$\{X_{(\sqrt{-1}\chi_*,\eta_2,\eta_3)}\}_{\eta_\C\in U}$ 
with the first complex structure $I_1$. 
These form a family of connected complex analytic spaces 
by Assumption \ref{asi}. 
For generic $\eta_2,\eta_3$, 
they are smooth due to Assumption \ref{asii} and the 
hyperK\"ahler metrics are complete, 
since the K\"ahler reduction is 
Riemannian submersion at the locus where $\mu$ is (topologically) 
submersive 
(see e.g., \cite[\S 3 A, C, D]{HKLR}). 
Also note that 
this can be regarded as the family of the GIT 
quotient by Kempf-Ness type theorem (cf., \cite[1.4(i)]{May}, \cite[9.66]{Kir} etc). 
We denote the Ricci-flat K\"ahler form on $X_{\zeta=(\zeta_1,\eta_\C)}$ 
for $I_1$ as 
$\omega_{1}(\zeta)$. Then, from the fact that $X_\zeta$ can be interpretted 
as GIT quotients with the twist $\chi$ by Assumption \ref{as:hodge}, 
a general theory in GIT implies 
that for some positive integer $m$, there is a (descended) 
complex algebraic 
line bundle $L_{\zeta}(m)$ on $X_\zeta$ for each 
$\zeta=(\zeta_1,\eta_2,\eta_3)$ 
which forms a holomorphic family 
with respect to the variation of $\eta_\C=(\eta_2,\eta_3)$ 
such that $m\omega_{1}(\zeta_1,\eta_2,\eta_3)$ 
are the curvature of $L_{\zeta}(m)$ with its Hermitian metrics. 
(If $\chi$ is trivial, $m$ can be taken as $1$ and one can take 
the line bundles $\mathcal{O}_{X_\zeta}$.) 

By Assumption \ref{asiii}, 
$\omega_1(\zeta)$ forms a continuous family 
in an open subset of the family 
$\cup_{\zeta=(\zeta_1,\eta_2,\eta_3)}(X_\zeta,I_1)$ which is defined as the quotient of 
the locus in $\mu^{-1}(\zeta)$s 
with finite stabilizers of the $K$-action. 
It easily follows by locally 
taking vertical slice to the $K$-orbit in $\mu^{-1}(\zeta)$ 
by the Riemannian submersiveness. 
If we apply Claim \ref{completeDSII} to these 
$X_{\zeta_1,\eta_2,\eta_3}$ 
for $(\eta_2, \eta_3)$ goes to $(\zeta_2, \zeta_3)$, 
we obtain some 
polarized limit space $X_\zeta^{\rm DS}$ which contains $X_\zeta^{\rm sm}$. 
Since $X_\zeta^{\rm DS}$ is a in particular complete as a metric space, this coincides with 
the metric completion of $\mu^{-1}(0)^{\rm sm}/K$. 

Now we define the distance $\overline{d}$ on the latter 
$\mu^{-1}(\zeta)/K$ 
as $\overline{d}(xK,yK):=\min\{d(x',y')\mid x'\in xK, y'\in yK\}$, 
where $d(-,-)$ is induced by restriction of the hyperK\"ahler metric on 
$M$ to $\mu^{-1}(\zeta)$. 
From the Riemannian submersion condition on the regular locus 
(cf., e.g., \cite[\S 3A]{HKLR}), combined with 
the completeness of $M$ and 
compactness of $K$, 
it follows that this $(X_\zeta^{\rm GIT},\overline{d})$ 
is the metric completion of  
$(X_\zeta^{\rm sm},\overline{d}|_{X_\zeta^{\rm sm}})$. 
By \cite[Theorem 2.1]{DancerSwann}, \cite[Theorem 1.4, \S 2.5]{May} (cf., 
also \cite{LSj}, \cite[\S 6]{Nakajima}), 
$X_\zeta^{\rm GIT}$ admits a natural topological stratification, along 
which we have a smooth, stratum-wise hyperK\"ahler 
description. 
Therefore, there is a bijective (finite) 
birational morphism 
$f\colon X_\zeta^{\rm DS}\to X_\zeta^{\rm GIT}$, 
which extends the identity on their common 
open dense subsets $X_\zeta^{\rm sm}$. Thus, $f$ is the 
(semi-)normalization. Since 
$\mu_\C^{-1}(\zeta_2+\sqrt{-1}\zeta_3)$ is locally complete intersection 
(hence Gorenstein) by Assumption \ref{as:equidim}, 
$X_\zeta^{\rm GIT}$ is also Cohen-Macaulay (\cite{HR}) so that 
the semi-normalization map 
$f$ is automatically a biholomorphism. 

Hence, 
$X_{\zeta=(\sqrt{-1}\chi_*,\zeta_\C)}$ with its any point $x$ 
is an example of the (pointed) polarized limit space in the 
Claim \ref{completeDSII}. 
In particular, $x\in X_\zeta$ is a 
log terminal singularity 
by \cite{DSII} due to the local volume finiteness of the 
adapted measure (cf., also e.g., \cite{EGZ}) 
but, since it has holomorphic volume form, 
it is canonical Gorenstein singularity. Hence, 
from \cite[Theorem 4]{Nam.ex}, it follows that 
$x\in X_\zeta$ is a symplectic singularity. 

On the other hand, \cite[Theorems 1.3, 1.4(iv), \S 4.3]{May} 
(cf., also \cite{LSj}, 
\cite[Proposition 4.4]{NakajimaTakayama} for related similar 
arguments) 
identify the complex analytic germ of $x\in X_\zeta$ 
with that of the 
origin of the hyperK\"ahler quotient of a holomorphic symplectic 
representation ($H\curvearrowright V$ in {\it loc.cit}). Here, $V$ is equipped with a 
standard flat (hyperK\"ahler) metric $g_{\rm flat}$ and an algebraic 
symplectic form. 
Since $V$ is a complex 
vector space, it is acted by 
a standard $\C^*$-action i.e., 
the multiplication and its subaction of $\R_{>0}$, 
which rescales the metric $g_{\rm flat}$. 
The $\C^*$-action descends to 
the quotient 
which gives the standard $\C^*$-action 
claimed in \eqref{SAHKQ}. Furthermore,  $g_{\rm flat}$ 
descends to 
singular hyperK\"ahler metric by \cite[\S 3]{HKLR}, 
which is a 
Calabi-Yau cone metric in the sense of 
\cite[around 2.9]{Linotes} i.e., with 
locally bounded potential. Indeed, by the same 
construction as $\overline{d}$ on whole $X_\zeta$, 
one can take the 
K\"ahler potential of the reduction metric of 
$g_{\rm flat}$, of the form $(r_{V//H})^2$ 
where $r_{V//H}$ denotes the distance function of 
the origin of $V$ and the $K\cap H$-orbit in the 
$0$-level set of $V$ (cf., e.g., \cite{MSY2, CS, Od24a}). 
By definition, it clearly takes  finite values and continuous. 
It implies \eqref{R1HKQ}, by (a weaker version of) the Yau-Tian-Donaldson 
type correspondence (e.g., by \cite[1.3 or 1.5]{LX} or 
\cite[Theorem 2.9]{Linotes}). 
Note that the above proofs of \eqref{SAHKQ} and \eqref{R1HKQ} do not use Assumptions in Proposition \ref{DSHKQ}. 

Finally, we finish the proof of the remained item \eqref{DSHKQiii} on 
Donaldson-Sun theory (Conjecture \ref{conj:genDS}) for $X_\zeta$ in 
Proposition \ref{DSHKQ}, by applying Claim \ref{completeDSII}. \end{proof}

For the above 
proof of symplectic singularity-ness, 
there are also different approaches by 
cf., e.g., \cite{BS, HSS, BS2, HSS2} under other (more explicit) conditions. 
Given above Proposition \ref{DSHKQ}, 
our methods in this paper imply the following finer information on 
the reduction metric. 

\begin{Cor}\label{torus.HKQ}
Under Setup \ref{set:HKLR} and Assumption in Proposition \ref{DSHKQ}, 
any analytic germ of the hyperK\"ahler quotient 
$x\in X_\zeta$ with the hyperK\"ahler metric 
$g_{X_\zeta}$ 
(induced by the hyperK\"ahler reduction description) 
satisfies the following: 
\begin{enumerate}
\item \label{lastcor1}
The analytic germ 
$x\in X_\zeta$ 
is that of (canonical) 
affine conical symplectic variety $C_x(X_{\zeta})$ 
at the vertex $0\in C_x(X_{\zeta})\curvearrowleft \mathbb{G}_m$. 
Furthermore, $0\in C$ has a (singular) hyperK\"ahler cone  metric 
$g_{C_x(X_{\zeta})}$, 
which in particular induces a canonical action 
of 
the multiplicative group $\R_{>0}$, 
as rescaling up of the metric. This action is the restriction of  
the algebraic action of $\C^*$ 
via the natural embedding $\R_{>0}\hookrightarrow \C^*$ 
as Lie groups. 

\item \label{lastcor2}
If $x\in X_\zeta$ is an isolated singularity, 
then there is a local diffeomorphism  $\Psi\colon U_{C_x(X_{\zeta})} \to X_\zeta, 0\mapsto x$ on to their images where 
$U_{C_x(X_{\zeta})} \subset C_x(X_{\zeta})$ is an open neighborhood of $0\in C_x(X_{\zeta})$
(similarly as Theorem \ref{ssps} \eqref{DSiii}) and a 
positive real number $\delta$ 
which satisfy the following: 
$$|\nabla^k_{g_{C_x(X_{\zeta})}}(\Psi ^*g_{X_\zeta} -g_{C_x(X_{\zeta})})|_{g_{C_x(X_{\zeta})}}=O(r^{\delta-k}),$$ 
for any $k\in \Z_{\ge 0}$. Here, $r$ denotes the 
distance function on $C_x(X_{\zeta})$ from the vertex with respect to 
$g_{C_x(X_{\zeta})}$. 
\end{enumerate}
\end{Cor}

\begin{proof}
A simple combination of 
Proposition \ref{DSHKQ} and our Theorem \ref{Mthm2.intro} 
implies \eqref{lastcor1}. 
Similarly, a simple combination of 
Proposition \ref{DSHKQ} and our Theorem \ref{Mthm2.intro}, 
together with \cite[Theorem 1.4]{Zha24} 
implies \eqref{lastcor2}. For the latter, note that the local 
K\"ahler potential around $x\in X_\zeta$ can be taken as continuous (hence bounded) 
by Proposition \ref{DSHKQ} \eqref{DSHKQiii} 
(cf., also \cite[\S 3 E]{HKLR}). 
\end{proof}

\begin{Rem}
This remark reflects 
discussions with H.Nakajima and C.Spotti, who had related thoughts 
and we are grateful for sharing them. 
Note that 
the cone metric $g_{C_x(X_{\zeta})}$ is obtained as the 
metric tangent cone of the hyperK\"ahler reduction metric 
$g_{X_\zeta}$ 
on $X_\zeta$ near 
$x$. A natural but subtle differential geometric 
question is if 
$g_{X_\zeta}$ is actually equivalent to $g_{C_x(X_{\zeta})}$. 
The above Corollary \eqref{lastcor2} only partially answers to it 
(in isolated singularities cases). 
A stronger conjecture would be that 
the local normal form in \cite[1.4(iv)]{May} 
(cf., also \cite[4.4]{NakajimaTakayama}) 
can be enhanced to the level of equivalence of 
hyperK\"ahler structures. 
Furthermore, if (expected) uniqueness of Ricci-flat cone metric on 
K-polystable $\Q$-Fano cone holds (cf., \cite{BandoMabuchi}, 
\cite{CS2}, \cite[\S 4.4]{HanLi}, \cite[2.9]{Linotes}), that 
would imply that  $g_{C_x(X_\zeta)}$ is nothing but the 
reduction metric of $g_{\rm flat}$ on the regular locus. 
We leave to future for discussion on this topic in further depth. 
\end{Rem}

\begin{Rem}
The main discussion in this subsection also 
shows a new way of understanding 
the local behaviour of singular hyperK\"ahler metric e.g.,  
in the hyperK\"ahler quotient, using the theories 
of \cite{DSII, ChiuSzek, Zha24}, combining with 
our main discussions via our formal local rigidity 
of Poisson deformation (\S \ref{sec:PD}). 
In general, we hope that if one can extend \cite{DSII, ChiuSzek, Zha24} 
to more general case 
(see also Remark \ref{Zha1exp}), or relate with \cite{Od24c} for instance, 
we anticipate that the approach in this paper will develop with further 
fruits. 
\end{Rem}


\subsection{Quasi-regularity and other 
consequences}\label{subsec:last}

In this subsection, we 
enhance our results by showing 
Conjecture \ref{r1conj} (hence $r=1$ in 
particular) conditionally, and also 
write 
some other consequences: such as  
existence of local 
$SU(2)$- or $\mathbb{H}^*$-symmetry, and construction of log K-polystable 
{\it contact} orbifolds, 
with singular conical K\"ahler-Einstein metric, 
in the spirit of 
LeBrun-Salamon conjecture (see Remark 
\ref{LSconj}). 

\begin{Thm}\label{r1prop}
Assume that a symplectic singularity germ 
$x\in X$ satisfies at least one of the following conditions. 
\begin{enumerate}
\item \label{r1prop:isolcase}
    $x\in X$ is an isolated symplectic singularity which 
satisfies Conjecture \ref{conj:genDS}, 
or 
\item \label{r1prop:withisotrop}
$x\in X$ is realized in a smoothable projective irreducible 
symplectic variety 
$\overline{X}$ as in Theorem~\ref{Mthm.intro}, such that
$c_1(L)^{\perp}\subset H^2(\overline{X},\mathbb{Z})$
contains a non-zero isotropic class with respect to the
Beauville--Bogomolov form; or similarly 

\item \label{r1prop:bigb2case}
$x\in X$ is realized in a smoothable projective irreducible 
symplectic variety 
$\overline{X}$ as in Theorem~\ref{Mthm.intro}, with 
$b_2(\overline{X})\ge 6$, or 

\item \label{r1prop:HKQcase}
$x\in X=X_\zeta$ in the setup of Proposition \ref{DSHKQ} 
($X_\zeta$ is a hyperK\"ahler reduction). 
\end{enumerate}
Under any of the above assumptions, 
Conjecture \ref{r1conj} holds so that 
$r(\xi)=1$ (quasi-regularity), and 
furthermore, all of 
the following properties hold for $x\in 
X$: 
\begin{enumerate}[label=(\alph*)]
   \item (Quasi-regularity) \label{r1propcons.a} the Reeb vector field $\xi=J(r\partial_r)$ 
   on $C_x(X)^{\rm sm}$ is periodic with 
   the length $\frac{2\pi}{m}$ for some $m\in \Z_{>0}$. In 
   particular, it is quasi-regular and provides canonical algebraic $\C^*$-action on $C_x(X)$. 
   \item ($SU(2)$-action and $\mathbb{H}^*$-action) \label{r1propcons.b} The (canonical) 
   $\C^*$-action extends to (non-holomorphic) action of $SU(2)$ and $\mathbb{H}^*$ on 
   $C_x(X)^{\rm sm}$. Here, $\mathbb{H}^*$ means the multiplicative 
   group of the quarternion $\mathbb{H}$. 
    \item (divisoriality) \label{r1propcons.c}
    the (volume-minimizing cf., 
Theorem \ref{AGlc}) valuation $v_X$ is multiple of divisorial valuation 
over $x\in X$ i.e., $v_X=m\cdot {\rm ord}_F$ for some divisor $F$ over $X$ 
with its center $x$. 
\item (Plt blow up) \label{r1propcons.d}
Moreover, $F$ is the (only) exceptional divisor of a plt blow up
\footnote{in the sense of 
\cite[3.1]{Shokurov}, 
\cite[Definition 2.1]{Prokhorov}. The exceptional divisor is 
also called Koll\'ar component.}
$B\to X$ 
of $x\in X$ i.e., 
$(B,F)$ is purely log terminal 
(cf., \cite[\S 2.3]{KM}) 
and 
$-(K_B+F)$ is relatively ample. 
Further, the log 
discrepancy satisfies 
$A_X(v_X)=m A_X(E)=\dim_\C(X)=n$. 
\item (weight is $2$) \label{r1propcons.e}
The weight 
\footnote{see the next Remark \ref{nonfaithful}} 
of canonical $\C^*$-action 
on the holomorphic symplectic form 
$\sigma_C$ on $C_x(X)^{\rm sm}$ is always 
$2$. 
\item (Contact K\"ahler-Einstein variety) \label{r1propcons.f} 
The quotient of $C_x(X)$ by the $\C^*$-action has a natural structure of 
$(n-1)$-dimensional 
log K-polystable Fano pair with complex contact orbifold structure 
(\cite{Name}) 
and conical singular 
K\"ahler-Einstein metrics (with natural associated angles along each 
prime branch divisor). 
\end{enumerate}
\end{Thm}
\begin{Rem}\label{nonfaithful}
Regarding \ref{r1propcons.e}, 
note that our $\C^*$-action may {\it not} be faithful 
and may factor through power map $t\mapsto t^l$ of $\C^*_t$ ($t$ is the 
coordinates of $\C^*$ and $l$ is a positive integer) in general, 
so that the weight of the holomorphic symplectic form for 
``natural" $\C^*$-action 
may first look to be $\frac{2}{l}$. 
Indeed, it corresponds to a possibly non-primitive element of $N$ 
that identifies with the Reeb vector field 
e.g., $-1\in \C^*$ acts trivially on the $2$-dimensional 
$A_{m}$-singularities with odd $m$. For instance, for the simplest surface 
$A_1$-singularity $X=(xy=z^2)\subset \C^3$ case, 
the $\C^*_t$-action is $x\mapsto t^2x, y\mapsto t^2y, z\mapsto t^2z$ (as a 
quotient of that of $\C^2$). 
\end{Rem}

\begin{Rem}
As the following proof shows, the main reason of 
our temporary assumption - \eqref{r1prop:isolcase} or 
\eqref{r1prop:withisotrop} or  \eqref{r1prop:bigb2case} - 
is to ensure some natural 
completeness of Reeb vector fields, 
obtained by 
hyperK\"ahler rotations (see e.g., Claim \ref{claim:completevf3}). 
Hence, as in Theorem \ref{Mthm2}(=Theorem\ref{Mthm2.intro}), we have: 
\begin{Thm}
For any symplectic singularity germ $x\in X$, all the above statements 
\ref{r1propcons.a}-\ref{r1propcons.f} hold 
if we assume the existence of hyperK\"ahler metric $g_X$ 
on $X^{\rm reg}$ near $x$ 
that satisfies Conjecture \ref{conj:genDS} and such natural 
completeness of the hyperK\"ahler rotations of the 
Reeb vector fields (see Claim \ref{claim:completevf3}). 
\end{Thm}
\end{Rem}
We expect there are always such hyperK\"ahler singular metric $g_X$, 
slightly strengthening Conjecture \ref{conj:genDS}. 

The possibility of this type of arguments is kindly 
pointed out by C.Spotti and 
goes back to the work of S.~Tanno \cite{Tanno}
(cf., also Kuo-Tachibana \cite{KuoT}, 
Boyer-Galicki 
\cite[Chapter 13]{BoyerGalicki}) in the context of 
$3$-Sasakian geometry. 
\begin{proof}[proof of Theorem \ref{r1prop}]
Firstly, recall from our Theorem \ref{ssps} and its proof that 
the metric tangent cone $C_x(X)^{\rm sm}$ is hyperK\"ahler 
(or equivalently, gives $3$-Sasakian structure on the regular locus of the link). Further, 
from Theorem \ref{Mthm}, its analytic germ is the same as 
$x\in X$. 
From these, 
we obtain the action of $\mathfrak{sp}(1)$ on $C_x(X)^{\rm sm}$
(cf., \cite[13.2.5, 13.3.1]{BoyerGalicki}). 
Recall that $\mathfrak{sp}(1)$ is the Lie algebra of 
$Sp(1):=\{q\in \mathbb{H}\mid |q|=1\}$ and hence 
generated by $I,J,K$ with $[I,J]=IJ-JI=2K$ and 
$[J,K]=2I,[K,I]=2J$ similarly. 
The action means a Lie algebra homomorphism 
$\phi\colon \mathfrak{sp}(1)\to 
\mathcal{X}(C_x(X)^{sm})$, the space of 
$C^\infty$ real vector fields with Lie brackets. 
This sends $I, J, K$ to 
$\xi_I:=I(r\partial_r), \xi_J:=J(r\partial_r), 
\xi_K:=K(r\partial_r)$ respectively. 
Here, $I$ is the original complex structure of $\overline{X}^{sm}$ 
so that $\xi=\xi_I$ 
and $J,K$ are those whose 
corresponding K\"ahler forms are 
${\rm Re}(\sigma_X), {\rm Im}(\sigma_X)$ respectively. 

For the case \eqref{r1prop:isolcase}, 
since the link of $C_x(X)$ is compact and $Sp(1)\simeq SU(2)$ 
is simply connected, 
$d\xi_I+e\xi_J+f\xi_K$ on $C_x(X)^{sm}$ is complete 
for any $d,e,f\in \R^3$ so that $\phi$ can be integrated 
to $Sp(1)$-action on $C_x(X)$ by the 
Palais integrability theorem. 
This $Sp(1)\simeq SU(2)$ 
should contain $r$-dimensional compact torus by the 
definition of the rank $r$ (review our \S \ref{sec:2}), so that 
$\xi=\xi_I$ is periodic with ${\rm exp}(2\pi \xi)={\rm Id}$. 
Further, this $Sp(1)$-action is isometric and preserves the link, say 
$L(C_x(X)^{sm})$, 
of $C_x(X)^{sm}$, so it naturally extends to a (non-holomorphic) action of 
$\mathbb{H}^*=\mathbb{R}_{>0}\times Sp(1)$ on 
$C_x(X)^{sm}=\mathbb{R}_{>0}\times L(C_x(X)^{sm})$. 
Thus,  
\ref{r1propcons.a} and \ref{r1propcons.b} follows (in this case of 
\eqref{r1prop:isolcase}). 
Using this $\G_m(\C)$-action, if we write $C_x(X)={\rm Spec}
\oplus_{m\ge 0}R_m$, then its grading degree gives the valuation $v_X$ and 
one can take the plt blow up explicitly as 
${\rm Proj}(\oplus_m I^{(m)})\to C_x(X)$ 
where $I^{(m)}:=\oplus_{l\ge m}R_l$. 
(For more general case, see \cite{LX.qreg, Blum}). 
Note that its exceptional divisor is 
${\rm Proj}\oplus_{m\ge 0} I^{(m)}/I^{(m+1)}\simeq 
{\rm Proj}(\oplus_{m\ge 0}R_m)$ which is 
integral because $\oplus_{m\ge 0}R_m$ is integral. 
These show 
\ref{r1propcons.c} and the former half of \ref{r1propcons.d}. 

On the other hand, recall that 
${\rm det}(\sigma_C):=\sigma_C^{\wedge \frac{n}{2}}$ is a nowhere-vanishing (algebraic) 
holomorphic 
volume form on $C_x(X)^{sm}$ 
which satisfies the gauge-fixing condition 
\begin{align}\label{gaugefix}
\mathcal{L}_\xi({\rm det}(\sigma_C))=n\sqrt{-1}
{\rm det}(\sigma_C)
\end{align}
(\cite{MSY2}, \cite[\S 6]{CS2}, \cite{Li}). 
This weight $n$ is also identified with the log discrepancy 
(see \cite[3.1, 3.13]{Kol13}, \cite[\S 4]{LSs}, \cite{Li}, \cite[3.6]{LLX}). 
From \eqref{gaugefix}, the constant $c$ with 
$\mathcal{L}_\xi(\sigma_C)=c\sqrt{-1}\sigma_C$ 
should be automatically $2$. 
This implies \ref{r1propcons.e}. 
Or alternatively, if one prefers more algebro-geometric arguments, 
we can use our Theorem \ref{thm:extendtoW} \eqref{canonicity.sigmaW} with $\xi'=\xi, D=1$ in its notation (cf., also Theorem \ref{thm:extendtoW2} \eqref{extension}) combined with the above 
that we $t^* \sigma_W=t^2\sigma_W$ with respect to the action 
\ref{r1propcons.a} of $\C^*\ni t$ and combination with the arguments in Theorem 
\ref{ssps} imply the same result \ref{r1propcons.e}. 
In the case when the $\C^*$-action on $C_x(X)$ is free outside the vertex (regular), 
see also \cite[3.1, 3.14]{Kol13} which also implies that 
$C_x(X)$ is a cone for the ample line bundle $L$ on $F$ 
with $L^{\otimes (-2)}\equiv_\Q (K_F+\Delta)$. 

Given the above statements 
\ref{r1propcons.a}-\ref{r1propcons.e}, 
the last statements \ref{r1propcons.f} 
follow directly from general theories. 
Indeed, 
the quotient of $C_x(X)$ 
naturally underlies a 
klt log $\Q$-Fano pair of the form 
$(F,\Delta=\sum_{i}(1-\frac{1}{m_i})\Delta_i)$ 
for a normal projective variety $F={\rm Proj}(\oplus_m R_m)$, with 
prime divisors $\Delta_i$ and $m_i\in \Z_{>0}$ 
(see e.g., \cite[\S 1, Lemma]{Namd}, \cite[\S 3.1, \S 9.4]{Kol13}). 
In the notation of 
\ref{r1propcons.d}, 
$-(K_F+\Delta)=-(K_B+F)|_F$. 
Moreover, the klt log $\Q$-Fano pair 
comes with the 
{\it contact orbifold structure} 
(\cite[4.4.1]{Name}, \cite[\S 2]{Namf}, see also 
earlier \cite[\S 2]{Le}, \cite[C.16]{Buc} and \cite{Smi}). 
Finally, the restriction of the 
hyperK\"ahler metric on the cone $C_x(X)^{sm}$ to the link descend to 
K\"ahler-Einstein metric on $F$ 
as the Reeb orbit space (cf., 
\cite{BoyerGalicki}) with the conical 
singularity of angles $\frac{2\pi}{m_i}$ 
around general points of each prime divisor 
$\Delta_i$. Alternative algebro-geometric  
arguments are as follows: 
as the extracted divisor of $v_X$, 
$(F,\Delta)$ is 
log K-semistable 
(in the sense of \cite{Don11, OS}
\footnote{note that in more recent literature, the prefix ``log" is 
often omitted but we should keep in mind that the definition 
depends on the boundary $\Q$-divisor $\Delta$}) 
by \cite{LX} and its log K-polystability 
also follows once we combine with 
our Theorem \ref{ssps}. This finishes the proof of \ref{r1propcons.f}. 


\vspace{2mm}

Next, we turn to prove the same results 
for the case \eqref{r1prop:withisotrop}. Basic methods are the same but 
we need more involved arguments to ensure the desired completeness of the 
vector fields, because the link of $C_x(X)^{sm}$ 
is {\it not} necessarily compact. 

Let us set 
$\Lambda^:=H^2(\overline{X},\Z)$ with the Beauville-
Bogomolov-Fujiki 
form as a quadratic form on it and also write 
$\Lambda_{\Q}:=\Lambda\otimes_{\Z}\Q$. 
As we follow the 
setup of Theorem \ref{Mthm.intro}, there is an ample line bundle 
$L$ on $\overline{X}$ with which $(\overline{X},L)$ is smoothable and a holomorphic symplectic form $\sigma_{\overline{X}}$ 
on $\overline{X}^{sm}$. We also take a symplectic resolution 
$\tilde{X}\to X$ 
by \cite[Theorem 4.8]{Fujiki} or Theorem 
\ref{smoothing}, 
with $\tilde{\Lambda}:=H^2(\overline{X},\Z)$. 
We set $c_1(L)=\lambda\in \Lambda$, $a:=\lambda^2$,  
$\tilde{\Lambda}_\lambda:=\{v\in \tilde{\Lambda}\mid 
\langle \lambda, v\rangle=0\}$, and 
$\Lambda_\lambda:=\{v\in \Lambda \mid 
\langle \lambda, v\rangle=0\}$. 
Consider the period domains as 
\[\mathcal{D}(\Lambda_\lambda):=\{[\sigma]\in 
\mathbb{P}_*(\Lambda_\lambda\otimes \C)\mid [\sigma]^2=0, 
\langle[\sigma],\overline{[\sigma]}\rangle>0 \}^o,\]
\[\mathcal{D}(\tilde{\Lambda}_\lambda):=\{[\sigma]\in 
\mathbb{P}_*(\tilde{\Lambda}_\lambda\otimes \C)\mid [\sigma]^2=0, 
\langle[\sigma],\overline{[\sigma]}\rangle>0 \}^o,\]
where the superscript 
$o$ means connected components (in a compatible manner). 
Recall that the pair of conditions 
$[\sigma]^2=0, 
\langle[\sigma],\overline{[\sigma]}\rangle>0$ is 
equivalent to 
$({\rm Re}[\sigma])^2=({\rm Im}[\sigma])^2>0$, 
$({\rm Re}[\sigma]), {\rm Im}[\sigma]) = 0$. 
By multiplying a suitable nonzero scalar to $\sigma_{\bar X}$, we may assume that $[{\rm Re}(\sigma_{\bar X})]^2 = a$.
Now we prepare the following rather arithmetic statement. 
\begin{Claim}\label{claim:approxperiod}
There is 
a close enough 
approximation of 
$([{\rm Re}(\sigma_{\overline{X}})],[{\rm Im}(\sigma_{\overline{X}})])
\in (\Lambda_\lambda\otimes \R)^2$ 
as 
\footnote{$(\Lambda_\lambda\otimes 
\Q,\Lambda_\lambda\otimes \R)$ is not a 
typo, 
only rationality of $v_1$ is used below  
where 
flexibility of $v_2$ is also conveniently 
used.} 
$(v_1,v_2)\in 
(\Lambda_\lambda\otimes \Q,\Lambda_\lambda\otimes \R)$ 
which satisfies 
\[v_1\in \Lambda_\lambda\otimes \Q, 
v_1^2=v_2^2=a, v_1\perp v_2.\]
Further, it is achieved as the period of 
(marked, locally trivial polarized) small deformation of $(\overline{X},L)$. 
\end{Claim}
\begin{proof}[proof of Claim \ref{claim:approxperiod}]
Firstly, following the assumption of \eqref{r1prop:withisotrop}, 
there is an isotropic vector in $\Lambda_\lambda$. 
Thus, there is 
$v'_1\in \Lambda_\lambda\otimes \Q$ 
with $(v'_1)^2=a$ by 
Proposition 3' (p33) in \cite{Serre.arith}.

Take a close enough approximation 
$v''_1\in \Lambda_\lambda\otimes \Q$ 
of $[{\rm Re}(\sigma_{\overline{X}})]$ and define the 
(rational) affine 
line $l$ of $\Lambda_\lambda\otimes \Q$ 
which passes through $v'_1$ and $v''_1$. 
Set $v_1:=(l\cap 
\{v\in \Lambda_\lambda\otimes \Q\mid v^2=a\})\setminus v'_1$. 
Then, this $v_1$ is 
automatically in $\Lambda_\lambda\otimes \Q$ 
and it approximates $[{\rm Re}(\sigma_{\overline{X}})]$ enough 
from the construction. 
Next, we take $v_2\in (\Lambda_\lambda\otimes \R)\cap 
v_1^\perp$ with $v_2^2=a$, 
which approximates $[{\rm Im}(\sigma_{\overline{X}})]$ enough. 

Then, 
by \cite[1.3]{BakkerLehn} (cf., also 
\cite[\S 8.6]{BakkerLehn2}) or simply 
by the local Torelli theorem as 
a much weaker statement, 
the period map image of 
the Kuranishi space of (marked) 
locally trivial deformations of 
$(\overline{X},L)$, contains small neighborhood of 
$[\sigma_{\overline{X}}]$. In particular, that 
period image contains $[v_1+\sqrt{-1}v_2]$. 
We end the proof of Claim \ref{claim:approxperiod}. 
\end{proof}

From the above claim, we may and do assume that
$[{\rm Re}(\sigma_{\overline{X}})] \in \Lambda_\lambda \otimes \Q$ henceforth,
since the rank $r$ is determined by the singularity germ,
which is preserved under locally trivial deformations. 
Take $x$ in the new $\overline{X}$ with the isomorphic 
analytic germ. 
Now, we take a sequence 
of polarized smooth projective irreducible symplectic 
manifolds $(X_i,L_i)$ for $i=1,2,\cdots$ 
which are 
marked, non-locally trivial polarized small deformation 
of $(\overline{X},L)$ and converges to 
$(\overline{X},L)$. We denote the corresponding 
hyperK\"ahler metric and hyperK\"ahler triple 
to $(\overline{X},L)$ 
as $g_X$ and $(\omega_{I,\infty}=\omega_X, 
\omega_{J,\infty}={\rm Re}(\sigma_{\bar X}), \omega_{K,\infty}={\rm Im}(\sigma_{\bar X}))$ respectively. 
Similarly, we denote the corresponding hyperK\"ahler metric of 
$(X_i,L_i)$ and its hyperK\"ahler triple as 
$g_i$ and $(\omega_{I,i},\omega_{J,i},\omega_{K,i})$ 
respectively. We can assume $[\omega_{I,i}]=\lambda$, 
$[\omega_{J,i}]=[{\rm Re}(\sigma_{\bar X})]\in 
\Lambda_\lambda\otimes \Q$ and only $[\omega_{K,i}]$ 
varies. 

\begin{Claim}\label{claim:completevf12}
$\xi_I, \xi_J$ are complete vector fields, 
whose integration extends to isometry on (the link of) 
$C_x(X)$. The same holds for 
$a\xi_I+b\xi_J$ for any $a,b\in \R$. 
\end{Claim}
\begin{proof}[proof of Claim \ref{claim:completevf12}] 
Firstly, the completeness of $\xi_I$ follows from 
\cite[Lemma 2.17]{DSII}, since 
$(\overline{X},L)$ (with the complex structure $I$ on 
the regular locus) is a polarized limit space of 
$(X_i,L_i,g_i,\omega_{I,i})$. 

Now, 
at the level of 
compact 
manifolds with hyperK\"ahler triple structures, 
we take a hyperK\"ahler rotation 
of $(X_i,\omega_{I,i},\omega_{J,i},\omega_{K,i})$ 
whose hyperK\"ahler triple becomes  
$(\omega_{J,i}, -\omega_{I,i}, \omega_{K,i})$, 
for each $i$ and simply denote as 
$X_i^{\vee}$. 
Since $[\omega_{J,i}]\in \Lambda_\lambda\otimes \Q$, 
$X_i^{\vee}$ is again 
naturally equipped with a structure of 
polarized irreducible symplectic manifold. 
By the definition of polarized limit space in 
\cite[\S 2]{DSII}, if we denote 
the polarized limit space $X_\infty^{\vee}$ of 
any subsequence\footnote{we do not change the index set for the simplicity of exposition} of 
$X_i^{\vee}$ as $X_\infty^{\vee}$, 
its hyperK\"ahler structure on the regular locus 
coincides with the hyperK\"ahler rotation of 
$\overline{X}^{sm}$ whose hyperK\"ahler triple is 
$({\rm Re}(\sigma_{\bar X}), -\omega_{I,\infty}, 
{\rm Im}(\sigma_{\bar X}))$. 
If we again apply \cite[Lemma 2.17]{DSII} 
to the $X_i^{\vee}\to X_\infty^{\vee} (i\to \infty)$, 
it follows that $\xi_J$ on the original 
$C_x(X)^{\rm sm}$ is complete. 

Now we proceed to prove the latter statement of 
Claim \ref{claim:completevf12}. 
Take any point $y\in C_x(X)^{sm}$. 
From the 
former statement of Claim \ref{claim:completevf12}, 
$d({\rm exp}(\xi_I)\cdot y,{\rm Sing}(C_x(X)))$ and 
$d({\rm exp}(\xi_J)\cdot y,{\rm Sing}(C_x(X)))$ are both 
constant. 
Here, $d(-,-)$ denotes the distance for the 
metric $g_X$ and its (metric) completion. 
In particular, 
it follows that for any shortest geodesic $\gamma 
\colon [0,1]\to C_x(X)$ 
between $y$ and the singular locus 
${\rm Sing}(C_x(X))$ 
of $C_x(X)$, with $\gamma(0)=y, \gamma(1)\in {\rm Sing}
(C_x(X))$, $(d\gamma/dt)|_{t=0}\perp \xi_I(y)$ 
and $(d\gamma/dt)|_{t=0}\perp \xi_J(y)$. 
Thus, $(d\gamma/dt)|_{t=0}\perp (a\xi_I+b\xi_J)(y)$ 
for any $a,b$. Therefore, since $y$ is taken 
arbitrarily, 
$(a\xi_I+b\xi_J)$ is again complete on $C_x(X)^{sm}$. 
We end the proof of Claim \ref{claim:completevf12}. 
\end{proof}

\begin{Claim}\label{claim:completevf3}
For any $d,e,f\in \R$, 
$d\xi_I+e\xi_J+f\xi_K$ on $C_x(X)^{sm}$ is complete. 
\end{Claim}
\begin{proof}
For the general case of Claim \ref{claim:completevf3} 
(possibly with $f \neq 0$), we employ the well-known isomorphism 
$\mathfrak{sp}(1) \simeq \mathfrak{o}(3)$ as follows
to make the arguments more intuitively 
understandable. 
We recall its basic as follows. 
Write the quaternion as $\mathbb{H} = \R \oplus \R I \oplus \R J \oplus \R K$, 
define $Sp(1) = \{q \in \mathbb{H}\: \vert \: q\bar{q} = 1\}$ and consider its 
the adjoint action of $Sp(1)$ on $\mathfrak{sp}(1) :=\R I \oplus \R J \oplus \R K$ as $q\colon x \mapsto qxq^{-1}$ for $x \in \mathfrak{sp}(1)$. 
Through the natural identification $\R^3 = \R I \oplus \R J \oplus \R K$, 
the adjoint action induces a $2:1$ 
covering $Sp(1) \twoheadrightarrow SO(3)$ of Lie groups and an isomorphism 
$\mathfrak{sp}(1) \cong \mathfrak{o}(3)$ of Lie algebras.

The proof then reduces to the previous case by the 
following trick. 

Below, we show 
that for any $d,e,f\in \R$, 
$dI+eJ+fK \in \R^3$ can be written as 
$({\rm exp}(aI+bJ))\cdot (cI)$ with the exponential map 
${\rm exp}\colon \mathfrak{sp}(1)\simeq \mathfrak{o}(3)
\to SO(3)$. Here, this $\cdot$ means the standard 
action. 
Indeed, $a,b,c$ can be taken using $d,e,f$ explicitly as follows. 
First, $c$ must represent the norm of the vector, i.e., 
$c = \sqrt{d^2 + e^2 + f^2}$. Set $d':=\frac{d}{c}, e':=\frac{e}{c}, 
f':=\frac{f}{c}$. 
We also want to also introduce $\theta$ as angle required to rotate the $x$-axis 
to the direction of $\R(d, e, f)$ with 
the rotation axis as $\R(a,b,0)$. 
We can define them explicitly: 
let $(a',b')\in \R^2$ as 
$a'^2+b'^2=1$ and $(a',b',0)\perp (1-
d',e',f')$. Explicitly, 
$a':=\frac{e'}{\sqrt{(1-d')^2+e'^2}}, 
b':=\frac{d'-1}{\sqrt{(1-d')^2+e'^2}}$. 
Then, since the distance between 
$(a',b')$ and $(1,0,0)$ is the same as 
that between $(d',e',f')$ and $(a',b')$, 
there is $\theta\in [0,2)$ 
so that angle $\pi\theta (\in [0, 2\pi\theta))$ rotation 
along the axis $\R_{\ge 0}(a',b')$ brings 
$(1,0,0)$ to $(d',e',f')$. Accordingly, 
we set 
$a:= a' \theta, \quad b:=b' \theta$. 
Then, the element $dI+eJ+fK$ is realized 
as the image of $cI$ 
under the adjoint action of $\exp(aI+bJ)$. 
Thus, 
\begin{align*}
\phi(dI+eJ+fK)&=\phi(({\rm exp}(aI+bJ))(cI))\\
&=({\rm exp}(aI+bJ))_*\phi(cI)
\end{align*}
is still complete on $C_x(X)^{sm}$ because 
it is simply the pushforward of a complete vector field by a (global) 
diffeomorphism. We finish the proof of 
Claim \ref{claim:completevf3}. 
\end{proof}

We are now ready to conclude the proof of 
Theorem \ref{r1prop} \eqref{r1prop:withisotrop}. 
Thanks to 
Claim \ref{claim:completevf3} and the simply-
connectedness of $Sp(1)$, 
the $\mathfrak{sp}(1)$-action on $C_x(X)^{sm}$ 
can be integrated to $Sp(1)$-action due to the 
Palais integrability theorem again. 
Thus, as in the proof of 
Theorem \ref{r1prop} 
\eqref{r1prop:isolcase}, 
$Sp(1)\supset (U(1))^{r(\xi)}$ holds, so that $r(\xi)=1$ and all other consequences 
\ref{r1propcons.a}-\ref{r1propcons.f} 
follow in the same manner as 
\eqref{r1prop:isolcase}. 

For the case of \eqref{r1prop:bigb2case}, 
since ${\rm rank}(\Lambda_\lambda)\ge 
5$, there exists an isotropic vector $e$ in $\Lambda_\lambda$ 
due to 
\cite[Chapter IV, Corollary 2 (p43)]{Serre.arith}. 
Thus, the proofs of our assertions 
are reduced to that of \eqref{r1prop:withisotrop}. 
For the case of \eqref{r1prop:HKQcase}, 
due to the existence of hyperK\"ahler rotations, 
Claim \ref{claim:completevf3} and all other assumptions similarly hold. 
We end the proof of Theorem \ref{r1prop}. 
\end{proof}

It might also be interesting to compare with 
Weinstein's conjecture 
\cite{Weinstein} (also cf., Remark 
\ref{LSconj}) in the real setup.


\begin{ack}
The possibility of this research has been 
intermittently discussed since 
the second KTGU Mathematics Workshop for Young Researchers in 
Kyoto (February, 2017) 
as well as Algebraic and Complex geometry session of 
the Eighth Pacific Rim Conference in Mathematics, the latter 
organized with S.Sun in August, 2020. 

The authors thank G.Bellamy, L.Foscolo, C.Lehn, M.Mauri, H.Nakajima, T.Schedler, S.Sun, C.Spotti, and 
J.Zhang for many helpful discussions during this research and comments. 
YN is partially supported by the Grant-in-Aid for 
Grant-in-Aid for Scientific Research (A) 21H04429. 
YO was partially supported by the 
Grant-in-Aid for Scientific Research (A) 21H04429, 
20H00112, 
the Grant-in-Aid for  Scientific Research (B) 23H01069, 
21H00973, 
and Fund for the Promotion of Joint International Research 
(Fostering Joint International Research) 23KK0249 
and the National Science 
Foundation under Grant No. DMS-1928930, while he was in
residence at the Simons Laufer Mathematical Sciences Institute
(formerly MSRI) in Berkeley, California, during the Fall 2024. 
This work was also 
partially supported by the Research Institute for Mathematical Sciences, 
an International Joint Usage/Research Center located in Kyoto University. 
\end{ack}


\vspace{3mm}
\footnotesize 
\noindent
Contact: {\tt namikawa@kurims.kyoto-u.ac.jp} \\
RIMS, Kyoto University, Kyoto 606-8285. JAPAN \\

\noindent
Contact: {\tt yodaka@math.kyoto-u.ac.jp} \\
Department of Mathematics, Kyoto University, Kyoto 606-8285. JAPAN \\

\end{document}